\documentclass{amsart}
\usepackage{amsmath}
\usepackage{amssymb}
\usepackage{amscd}
\usepackage[pdftex]{graphicx}
\usepackage{color}
\usepackage[all]{xy}
\usepackage{amsthm}

\pdfoutput=1

\theoremstyle{plain}
\newtheorem{thm}{Theorem}[section]
\newtheorem{prop}[thm]{Proposition}
\newtheorem{lemma}[thm]{Lemma}
\newtheorem{cor}[thm]{Corollary}

\newtheorem*{fact}{Fact}

\theoremstyle{remark}
\newtheorem*{claim}{Claim}
\newtheorem*{remark}{Remark}

\theoremstyle{definition}
\newtheorem*{dfn}{Definition}
\newtheorem{example}{Example}

\newcommand{\co}{\colon\thinspace}

\newcommand{\bound}{\partial}

\newcommand{\Q}{\mathbb{Q}}
\newcommand{\C}{\mathbb{C}}
\newcommand{\R}{\mathbb{R}}

\renewcommand{\c}{\mathcal{C}}
\newcommand\calh{\mathcal{H}}
\newcommand{\p}{\mathcal{P}}

\newcommand{\s}{\mathcal{S}}
\newcommand{\calc}{\mathcal{C}}
\newcommand{\call}{\mathcal{L}}
\newcommand{\calo}{\mathcal{O}}
\newcommand{\calp}{\mathcal{P}}
\newcommand{\calq}{\mathcal{Q}}
\newcommand{\calt}{\mathcal{T}}

\newcommand{\olp}{\overline{\mathcal{P}}}
\newcommand{\cat}{\ensuremath\mathrm{CAT}}

\newcommand\crush{crushtacean}

\newcommand\calb{\mathcal{B}}

\newcommand\sfa{\mathsf{a}}
\newcommand\sfb{\mathsf{b}}
\newcommand\sfc{\mathsf{c}}
\newcommand\sfd{\mathsf{d}}
\newcommand\sfg{\mathsf{g}}

\newcommand\sfs{\mathsf{s}}

\newcommand\bv{\mathbf{v}}
\newcommand\bw{\mathbf{w}}

\definecolor{gray1}{gray}{0.6}

\begin{document}

\title{Some virtually special hyperbolic $3$-manifold groups}

\author{Eric Chesebro}
\address{Department of Mathematical Sciences, University of Montana}
\email{Eric.Chesebro@mso.umt.edu}

\author{Jason DeBlois}
\address{Department of MSCS, University of Illinois at Chicago}
\email{jdeblois@math.uic.edu}
\thanks{Second author partially supported by NSF grant DMS-0703749}

\author{Henry Wilton}
\address{Department of Mathematics, Caltech}
\email{wilton@caltech.edu}
\thanks{Third author partially supported by NSF grant DMS-0906276}

\begin{abstract}
Let $M$ be a complete hyperbolic 3-manifold of finite volume that admits a decomposition into right-angled ideal polyhedra.  We show that M has a deformation retraction that is a virtually special square complex, in the sense of Haglund and Wise {and deduce that such manifolds are virtually fibered.   We generalise a theorem of Haglund and Wise to the relatively hyperbolic setting and deduce that ${\pi_1}M$  is LERF and that the geometrically finite subgroups of ${\pi_1}M$ are virtual retracts.  }Examples of 3-manifolds admitting such a decomposition include augmented link complements.  We classify the low-complexity augmented links and describe an infinite family with complements not commensurable to any $3$-dimensional reflection orbifold.
\end{abstract}

\maketitle

\section{Introduction} \label{sec:intro}

Let $\{ \p_i \}_{i=1}^n$ be a collection of disjoint ideal polyhedra in $\mathbb{H}^3$.  A \textit{face pairing} on $\{ \p_i \}$ is a collection of isometries $\{\phi_f\, |\, f \in \p_i^{(2)}\, 1 \leq i \leq n\}$ of $\mathbb{H}^3$ with the following properties.  If $f$ is a face of $\p_i$, $\phi_f$ takes $f$ onto a face $f'$ of some $\p_j$, with $\phi_f(\calp_i) \cap \calp_j = f'$, and $\phi_{f'} = \phi_f^{-1}$.  Now let $M$ be a complete hyperbolic $3$-manifold with finite volume.  An \textit{ideal polyhedral decomposition} of $M$ is an isometry between $M$ and a quotient $\bigsqcup_i \p_i/\sim$, where $\sim$ is the equivalence relation generated by a face pairing on $\{\p_i\}$.  If the dihedral angles of every polyhedron $\p_i$ are all equal to $\pi/2$ then the decomposition is called an \emph{ideal right-angled polyhedral decomposition}.

Our first result relates fundamental groups of 3-manifolds that admit ideal right-angled polyhedral decompositions to the class of right-angled Coxeter groups.  A \emph{right-angled Coxeter group} $W$ is defined by a finite, simplicial graph $\Delta$ (called the \emph{nerve} of $W$) and has an easily described presentation: the generators are the vertices; every generator is an involution; and the commutator of two generators is trivial if and only if they are joined by an edge in $\Delta$.  We will refer to the vertices of the nerve as the \emph{standard generating set} for $W$.  The properties of such $W$ discovered in \cite{Agol} and \cite{Hag} will particularly concern us.

\begin{thm} \label{rt ang cox}  Suppose $M$ is a complete hyperbolic $3$-manifold with finite volume that admits a decomposition into right-angled ideal polyhedra.  Then $\pi_1M$ has a subgroup of finite index isomorphic to a word-quasiconvex subgroup of a right-angled Coxeter group (equipped with the  standard generating set).  \end{thm}

See Section \ref{sec:QCERF} for the definition of word quasiconvexity.  In the terminology of \cite{HW}, Theorem \ref{rt ang cox} asserts that $\pi_1M$ is \emph{virtually special}.  The proof relies on work of Haglund--Wise \cite{HW} defining a class of \textit{special} cube complexes --- non-positively curved cube complexes whose hyperplanes lack certain pathologies --- which are locally isometric into cube complexes associated to right-angled Coxeter groups.  In Section \ref{cube cx} we review the relevant definitions and in Section \ref{std square} describe a \textit{standard} square complex associated with an ideal polyhedral decomposition of a hyperbolic $3$-manifold.

When an ideal polyhedral decomposition is right-angled, the associated standard square complex is non-positively curved, and hyperplanes are carried by totally geodesic surfaces.  We will establish these properties in Subsection \ref{std square} and Section \ref{sec:geodesic hyperplanes}.  Separability properties of totally geodesic surfaces then imply that pathologies may be removed in finite covers.  We describe these properties and prove Theorem \ref{rt ang cox} in Section \ref{sec:separability}.

This result has important consequences for the geometry and topology of such manifolds.  The first follows directly from work of Agol \cite{Agol}, and confirms that the manifolds we consider satisfy Thurston's famous Virtually Fibered Conjecture.

\begin{cor}  Suppose $M$ is a complete hyperbolic $3$-manifold with finite volume that admits a decomposition into right-angled ideal polyhedra.  Then $M$ is virtually fibered.  \end{cor}

{A $3$-manifold that satisfies the hypotheses of Theorem \ref{rt ang cox} is necessarily not compact, so its fundamental group is not hyperbolic in the sense of Gromov, but rather hyperbolic \emph{relative to} the collection of its cusp subgroups.  }Nonetheless, {our second theorem } implies that {its }subgroup structure shares the separability properties of {its }compact {cousins}.  This generalizes \cite[Theorem 1.3]{HW} to the relatively hyperbolic setting.  We say a subgroup $H$ of a group $G$ is a \emph{virtual retract} if $H$ is contained in a finite-index subgroup $K$ of $G$ and the inclusion map $H\hookrightarrow K$ has a left inverse.  {(See \cite{long_subgroup_2008} for further details of virtual retractions.)}

\begin{thm}\label{rel qc still sep}
Let $\s$ be a compact, virtually special cube complex and suppose that $\pi_1\s$ is hyperbolic relative to a collection of finitely generated abelian subgroups.  Then every relatively quasiconvex subgroup of $\pi_1\s$ is a virtual retract.  \end{thm}

{As in the case of \cite[Theorem 1.3]{HW}, the proof of }Theorem \ref{rel qc still sep} relies on {a } result of Haglund for separating subgroups of right-angled Coxeter groups \cite[Theorem A]{Hag}, but {it also } requires new ingredients to surmount the technical obstacle that not every relatively quasiconvex subgroup is word-quasiconvex.  The first is Theorem \ref{t: Fully rel. qc subgroups}, a variation of \cite[Theorem 1.7]{manning_separation_2008}, which establishes that every relatively quasiconvex subgroup is a retract of a \textit{fully relatively quasiconvex} subgroup (see the definition above Theorem \ref{t: Fully rel. qc subgroups}).  The second ingredient, Proposition \ref{p: Fully qc implies combinatorially qc}, extends work in \cite{hruska_relative_2008} to show that fully relatively quasiconvex subgroups satisfy the hypotheses of \cite[Theorem A]{Hag}.  

{Even without any restrictions on the types of parabolic subgroups allowed, our results prove that certain subgroups of relatively hyperbolic group{s }are virtual retracts: see Theorem \ref{t: Retracts with no hypotheses} and its corollaries for precise statements.}

The consequences of Theorem \ref{rel qc still sep} follow a long-standing theme in the study of $3$-manifolds and their fundamental groups.  For a group $G$ and a subgroup $H$, we say $H$ is \textit{separable} in $G$ if for every $g \in G - H$, there is a finite-index subgroup $K < G$ such that $H<K$ and $g \not\in K$.  If $G = \pi_1M$ for some manifold $M$, work of G.P. Scott links separability of $H$ with topological properties of the corresponding cover $M_H \to M$ \cite{Scott}.  A group is called \emph{LERF} if every finitely generated subgroup is separable.

\begin{cor}\label{c: Corollary 2}
Suppose $M$ is a complete hyperbolic $3$-manifold with finite volume that admits a decomposition into right-angled ideal polyhedra.  Then:
\begin{enumerate}
\item \label{item 1} $\pi_1M$ is LERF.
\item \label{item 2} Every geometrically finite subgroup of $\pi_1M$ is a virtual retract.
\end{enumerate}
\end{cor}

We will prove Theorem \ref{rel qc still sep} and Corollary \ref{c: Corollary 2} at the end of Section \ref{sec:QCERF}.

The study of LERF 3-manifold groups dates back to \cite{Scott}.  Although there are examples of graph manifolds with non-LERF fundamental group \cite{burns_notegroups_1987}, it remains unknown whether every hyperbolic 3-manifold group is LERF. Gitik \cite{Gitik} constructed examples of hyperbolic 3-manifolds with totally geodesic boundary whose fundamental groups are LERF, and it is a consequence of Marden's Tameness Conjecture that her closed examples are also LERF.  Agol, Long and Reid proved that the Bianchi groups are LERF \cite{ALR}.  

It is natural to ask to what extent Theorem \ref{rt ang cox} describes \textit{new} examples of $3$-manifold groups that virtually embed into right-angled Coxeter groups, and more generally to what extent it describes new examples of LERF 3-manifold groups.  Hitherto, there have only been a limited number of techniques for proving that finite-volume 3-manifolds are LERF.  The techniques of \cite{Gitik} did not produce non-compact, finite-volume examples, so we shall not consider them here.

Agol, Long and Reid \cite{ALR} proved that geometrically finite subgroups of right-angled, hyperbolic reflection groups are separable.  They deduced a similar result for the Bianchi groups by embedding them as totally geodesic subgroups of higher-dimensional, arithmetic right-angled reflection groups. One might na\"ively suppose that the fundamental group of a $3$-manifold that decomposes into right-angled polyhedra $\{\p_i\}$ is commensurable with the reflection group in one of the $\p_i$, or perhaps a union of several, and therefore that Theorem \ref{rt ang cox} could be deduced using the techniques of \cite{ALR}.  

We address the above possibility in Sections \ref{sec:examples} and \ref{sec: augmented}.  There we describe infinite families of hyperbolic $3$-manifolds that decompose into right-angled polyhedra but are not commensurable with \textit{any} $3$-dimensional reflection orbifold.  Indeed, Section \ref{sec: augmented} considers a very broad class of hyperbolic $3$-manifolds, the augmented link complements (previously considered in \cite{LAD} and \cite{Purcell_cusps}, for example), that decompose into right-angled polyhedra.  Our investigations there strongly support the following hypothesis: a ``generic'' augmented link complement is not commensurable with any $3$-dimensional reflection orbifold.

If $M$ decomposes into isometric copies of a single, highly symmetric polyhedron $\calp$, we show in Proposition \ref{symm comm} that $\pi_1 M$ is indeed commensurable with the reflection group in the sides of $\calp$.  The lowest-complexity right-angled ideal polyhedra (measured by number of ideal vertices) are the $3$- and $4$-antiprisms (see Figure \ref{P_1 and P_2}), and these are sufficiently symmetric for the hypotheses of Proposition \ref{symm comm} to apply.  However, in Section \ref{sec:One cusp}, we describe hybrid examples not commensurable with reflection groups.

\begin{thm}\label{t: One-cusp examples}
For each $n \in \mathbb{N}$, there is complete, one-cusped hyperbolic $3$-manifold $N_n$ that decomposes into right-angled ideal polyhedra, such that $N_n$ is not commensurable with $N_m$ for any $m \neq n$, nor to any $3$-dimensional reflection orbifold.
\end{thm}

Recently, Haglund and Wise have proved that every Coxeter group is virtually special \cite{haglund_coxeter_????}.  Since $\pi_1 N_n$ is not commensurable with any $3$-dimensional reflection group, the results of \cite{haglund_coxeter_????} do not apply to it.  The proof of Theorem \ref{t: One-cusp examples} uses work of Goodman--Heard--Hodgson \cite{GHH} to explicitly describe the commensurator of $\pi_1 N_n$.

A rich class of manifolds that satisfy the hypotheses of Theorem \ref{rt ang cox} consists of the \textit{augmented links} introduced by Adams \cite{Adams}.  Any link $L$ in $S^3$ with hyperbolic complement determines (not necessarily uniquely) an augmented link using a projection of $L$ which is \textit{prime} and \textit{twist-reduced}, by adding a ``clasp'' component encircling each crossing region.  (See Section \ref{sec: augmented} for precise definitions.)  Each link with hyperbolic complement admits a prime, twist reduced diagram, and the augmented link obtained from such a diagram also has hyperbolic complement (cf.~\cite[Theorem 6.1]{Purcell_cusps}).  Ian Agol and Dylan Thurston showed in an appendix to \cite{LAD} that each augmented link satisfies the hypotheses of Theorem \ref{rt ang cox}.

\begin{example}[Agol--Thurston] \label{fully augmented}  Let $M$ be a complete hyperbolic manifold homeomorphic to the complement in $S^3$ of an augmented link.  Then $M$ admits a decomposition into two isometric right-angled ideal polyhedra.  \end{example}

In Section \ref{sec: augmented}, we describe another polyhedron, the ``\crush'', that distills the most important combinatorial features of the Agol--Thurston ideal polyhedral decomposition.  We record criteria, in Lemmas \ref{symmetric links} and \ref{hidden symmetries}, that describe certain situations in which one may conclude that an augmented link complement is commensurable with the reflection orbifold in the associated right-angled polyhedron.  Section \ref{sec:low cx augmented} describes the scissors congruence classification of the complements of augmented links with up to $5$ crossing regions.  Finally, in Section \ref{sec:lobel} we prove:

\newtheorem*{Lobellthm}{Theorem \ref{Lobell thm}}
\begin{Lobellthm}  There is a class of augmented links $L(n)$, $n \geq 3$, such that for all but finitely many $n$, $M(n) \doteq S^3 - L(n)$ is not arithmetic nor commensurable with any $3$-dimensional hyperbolic reflection orbifold.  Moreover, at most finitely many $M(n)$ occupy any single commensurability class.  \end{Lobellthm}

The \crush s of the links of Theorem \ref{Lobell thm} are the famous L\"obell polyhedra.  We believe that the behavior recorded in the theorem is generic among augmented links, but these are particularly amenable to analysis.

While this work was in preparation, we became aware of \cite{BHW} and \cite{B}, which provide other examples of virtually special hyperbolic manifolds.  The former mostly concerns arithmetic lattices, while the latter deals with finite-sheeted covers of the 3-sphere that branch over the figure-eight knot.

\subsection*{Acknowledgements}

The authors would like to thank Ian Agol, Dave Futer, Alan Reid and Matthew Stover for useful conversations.  Thanks also to Jack Button for confirming some of our Alexander polynomial computations, and to Jessica Purcell for a helpful reference to \cite{Purcell_cusps}.  Finally, we thank the referee for a careful reading and helpful comments.

%%%%%%%%%%%%%%%%%%%%%
\section{Preliminaries}
%%%%%%%%%%%%%%%

%%%%%%%%%%%%%%%%%%%%%%%%%%%%%%
\subsection{Cube complexes}  \label{cube cx}

In this subsection we review relevant notions about cube complexes following the treatment of Haglund--Wise \cite{HW}.  Another helpful reference is \cite{BrH}, particularly Chapters I.7 and II.5.

\begin{dfn}(\cite[Definition 2.1]{HW})  Let $I = [-1,1] \subset \mathbb{R}$.  A \textit{cube complex} $X$ is a $CW$-complex such that each $k$-cell has a homeomorphism to $I^k \subset \mathbb{R}^k$ with the property that the restriction of the attaching map to each $(k-1)$-face of $\partial I^k$ to $X^{k-1}$ is an isometry onto $I^{k-1}$ followed by the inclusion of a $(k-1)$-cell.  A map $f \co X \to Y$ between cube complexes is \textit{combinatorial} if for each $k$-cell $\phi \co I^{k} \to X$, the map $f \circ \phi$ is a $k$-cell of $Y$ following an isometry of $I^k$.  A \textit{square complex} is a $2$-dimensional cube complex, and we will refer by \textit{vertex}, \textit{edge}, or \textit{square} to the image in $X$ of a $0$-, $1$- or $2$-cell, respectively.  \end{dfn}

Now let $X$ be a square complex.   We will take the \textit{link} of the vertex $(1,1) \in I^2$ to be the line segment in $I^2$ joining $(0,1)$ to $(1,0)$ (the midpoints of the edges abutting $(1,1)$), and the link of another vertex $v$ to be the image of the link of $(1,1)$ under the symmetry taking it to $v$.  The link of a vertex $v \in X$ is the $1$-complex obtained by joining the links of $v$ in the squares of $X$ attaching to it.  We say $X$ is \textit{simple} if for each vertex $v$ there is a combinatorial map from the link of $v$ to a simplicial graph.  In particular, if $X$ is simple then no two squares meet along consecutive edges.

We will say a square complex $X$ is \textit{nonpositively curved} if for each vertex $v$ in $X$, the link of $v$ does not contain any simple cycle with fewer than four edges.  (We are taking Gromov's link condition as a definition; see eg, \cite[Ch. II.5]{BrH} for a discussion.)  In particular, $X$ is simple.  If $X$ is simply connected and nonpositively curved, we will say $X$ is $\cat(0)$.  For a more general discussion, see \cite{BrH}, in particular Chapter II.5.

The notion of a \textit{hyperplane} is very important in defining ``special'' cube complexes.  Here we will specialize the definition in \cite{HW} to square complexes.

\begin{dfn}(\cite[Definition 2.2]{HW})  The \textit{midlines} of $I^2$ are the subsets $I \times \{0\}$ and $\{0\} \times I$, each parallel to two edges of $X$.  The \textit{center} of a square $\phi\co I^2 \to X$ is $\phi(0,0)$, and the \textit{midpoint} of an edge $\phi \co I \to X$ is $\phi(0)$.  A midline of $I^2$ meets its two \textit{dual} edges perpendicularly at their midpoints.

Given a square complex $X$, we define a graph $Y$, the associated \textit{midline complex}, as follows.  The $0$-cells of $Y$ are the midpoints of the edges of $X$, and the $1$-cells of $Y$ are midlines of squares of $X$, attached by the restrictions of the corresponding attaching maps.  A \textit{hyperplane} of $X$ is a component of the associated midline complex $Y$.  \end{dfn}

By the definition of the midline complex, each hyperplane $Y$ has an immersion into $X$, taking an edge to the midline of the square that contains it.  Definition 3.1 of \cite{HW} describes the following pathologies of hyperplane immersions: \textit{self-intersection}, \textit{one-sidedness}, \textit{direct} or \textit{indirect self-osculation}, or \textit{inter-osculation}.  If the hyperplanes of $X$ do not have any such pathologies, and its one-skeleton is bipartite, we will say that $X$ is $C$-\textit{special}.

The following theorem of Haglund--Wise is our main concern.

\begin{thm}[\cite{HW}, Lemma 4.3]\label{t: Haglund--Wise}  Let $X$ be a $C$-special square complex.  Then there exists a right-angled Coxeter group $W$, an injective homomorphism $\pi_1 X\hookrightarrow W$ and a $\pi_1 X$-equivariant, combinatorial, isometric embedding from the universal cover of $X$ into the Davis--Moussong complex of $W$.  In particular, $\pi_1X$ is isomorphic to a word-quasiconvex subgroup of $W$ (with respect to the standard generating set).\end{thm}

The Davis--Moussong complex of a right-angled Coxeter group $W$ is a certain CAT(0) cube complex on which $W$ acts naturally.  The reader is referred to \cite{HW} for the definition.  A square complex $X$ is called \emph{virtually special} if $X$ has a $C$-special finite-sheeted covering space. To prove Theorem \ref{rt ang cox}, we will prove that $\pi_1 M$ is isomorphic to the fundamental group of a virtually special square complex.

We will find the notion of a regular neighborhood of a hyperplane from \cite{HW} useful.

\begin{dfn}
Let $Y\to X$ be a hyperplane of a square complex $X$.  A \emph{(closed) regular neighborhood} for $Y$ is a cellular $I$-bundle $p:N\to Y$ equipped with a combinatorial immersion $j:N\to X$ such that the diagram
\[\xymatrix{
      N\ar@{>}[d]^{p}\ar@{>}[rd]^{j}& \\
      Y\ar@{>}[r] & X
}\]
commutes.  (Here the $I$-bundle $N$ is given the obvious square-complex structure: the preimage of a vertex is an edge and the preimage of an edge is a square.)
\end{dfn}

Every hyperplane of a non-positively curved square complex has a regular neighborhood \cite[Lemma 8.2]{HW}.  The $I$-bundle $p \co N \to Y$ has a section taking each $e \in Y^{(1)}$ to a midline of the square $p^{-1}(e)$.  We refer to $Y \subset N$ as embedded by this section.  In \cite[Definition 8.7]{HW}, the \textit{hyperplane subgroup} $\pi_1 Y < \pi_1 X$ is defined as the image of $j_*$ after an appropriate choice of basepoint.

%%%%%%%%%%%%%%%%%%%%%%%%%%%%%%%%%%
\subsection{A standard square complex}  \label{std square}

In this subsection we will take $M$ to be a complete hyperbolic $3$-manifold of finite volume, with an ideal polyhedral decomposition $\{\p_i\}$.  For a pair of faces $f$ and $f'$ of polyhedra $\p_i$ and $\p_j$ such that $f' = \phi_f(f)$, we say that $f$ and $f'$ represent a face of the decomposition.  Similarly, let $\{e_j\}_{j=1}^n$ be a sequence of edges of polyhedra $\p_{i_j}$ with the property that for each $j <n$, there is a face $f_j$ of $\p_{i_j}$ containing $e_j$ such that $\phi_{f_j}(e_j) = e_{j+1}$.  Then we say the edges $e_j$ represent an edge of the decomposition.

For each $i$, let $\olp_i$ be the union of $\p_i$ with its ideal vertices.  (In the Poincar\'e ball model for $\mathbb{H}^3$, the ideal vertices of $\p_i$ are its accumulation points on $\partial B^3$.)  Each face pairing isometry $\phi_f$ induces a homeomorphism from $\bar{f}$, the union of $f$ with its ideal vertices, to $\bar{f}'$, where $f' = \phi_f(f)$.

The extended face pairings determine a cell complex $\c$ such that $M$ is homeomorphic to $\c - \c^{(0)}$.  The $0$-cells of $\c$ are equivalence classes of ideal vertices under the equivalence relation generated by $v \sim \phi_f(v)$ for ideal vertices $v$ of faces $f$.  The $1$- and $2$- cells of $\c$ are equivalence classes of edges and faces of the $\p_i$ under the analogous equivalence relation, and the $3$-cells are the $\olp_i$.

Let $\c'$ be the barycentric subdivision of the cell complex $\c$ associated to an ideal polyhedral decomposition. If $v$ is a vertex of a cell $\olp$ of $\c'$, the open star of $v$ in $\olp$ is the union of the interiors of the faces of $\olp$ containing $v$.  The open star $\mathfrak{st}(v)$ of $v$ in $\c'$ is the union of the open stars of $v$ in the cells of $\c'$ containing it.   Take $\mathfrak{st}(\c^{(0)})$ to be the disjoint union of the open stars in $\c'$ of the vertices of $\c$.  Then $\s_0 \doteq \c' - \mathfrak{st}(\c^{(0)})$ is the unique subcomplex of $\c'$, maximal with respect to inclusion, with the property that $\s_0^{(0)} = (\c')^{(0)} - \c^{(0)}$.

A simplex of $\s_0$ is determined by its vertex set, which consists of barycenters of cells of $\c$.  We will thus refer to each simplex of $\s_0$ by the tuple of cells of $\c$ whose barycenters are its vertices, in order of increasing dimension.  For example, a simplex of maximal dimension is a triangle of the form $(\bar{e},\bar{f},\overline{\p}_i)$, where $e$ is an edge and $f$ a face of some ideal polyhedron $\p_i$ in the decomposition of $M$, with $e \subset f$.

\begin{lemma} \label{spine}
There is a cellular deformation retraction $\Phi$ taking $M$ to $|\s_0|$.
\end{lemma}

\proof  Let $v$ be an ideal vertex of $\olp_i$, and let $U$ be the open star in $\c'$ of the equivalence class of $v$ in $\c^{(0)}$.  Let $U_0$ be the component of $U \cap \olp_i$ containing $v$.  Then $\overline{U}_0$ is homeomorphic to the cone to $v$ of its frontier in $\olp_i$, a union of triangles of $\s_0$.  Hence there is a ``straight line'' deformation retraction of $\overline{U}_0 - \{v\}$ to its frontier.  These may be adjusted to match up along faces of the $\p_i$, determining $\Phi$.
\endproof

The standard square complex is obtained by taking a union of faces of $\s_0$.

\begin{dfn}  Let $M$ be a complete hyperbolic $3$-manifold with a decomposition into ideal polyhedra $\{\p_i\}$, with associated cell complex $\c$ such that $M \cong \c - \c^{(0)}$, and let $\s_0 = \c' - \mathfrak{st}(\c^{(0)})$, where $\c'$ is the first barycentric subdivision of $\c$.  Define the \textit{standard square complex} $\s$ associated to $\{\p_i\}$, with underlying topological space $|\s| = |\s_0|$, as follows: $\s^{(0)} = \s_0^{(0)}$, $\s^{(1)} = \s_0^{(1)} - \{(\bar{e},\overline{\p}_i)\,|\, e \subset \p_i\}$, and $\s^{(2)} = \{ (\bar{e},\bar{f},\overline{\p}_i) \cup (\bar{e},\bar{g},\overline{\p}_i)\,|\, f,g \subset \p_i\  \mbox{and}\ f\cap g = e\}$.  \end{dfn}

Since each $2$-dimensional face $(\bar{e},\bar{f},\olp_i) \cup (\bar{e},\bar{g},\olp_i)$ of $\s$ is the union of two triangles of $\s_0$ which meet along the edge $(\bar{e},\olp_i)$, it may be naturally identified with a  square.  Furthermore, since it is exactly the set of edges of the form $(\bar{e},\olp_i)$ which are in $\s_0^{(1)} - \s^{(1)}$, $\s$ has the structure of a cell complex.  %Figure \ref{fig:std square} depicts the standard square complex associated to an example of the sort of polyhedron with which we are concerned: the right-angled ideal cuboctahedron.

%\begin{figure}%[h]
%\centering
%\includegraphics[width=3.75in]{cubo.pdf}
%\caption{The standard square complex associated to a right-angled ideal cuboctahedron.}
%\label{fig:std square}
%\end{figure}

\begin{lemma}\label{bipartite}  Let $\s$ be the standard square complex associated to an ideal polyhedral decomposition $\{\calp_i\}$.  Then $\s^{(1)}$ is bipartite.  \end{lemma}

\begin{proof}  By definition, the vertices of $\s$ are barycenters of cells of the cell complex $\c$ associated to $\{\p_i\}$.  We divide them into two classes by parity of dimension.  An edge of $\s$ is of the form $(\bar{f},\olp_i)$ for some $i$, where $f$ is a face of $\calp_i$, or $(\bar{e},\bar{f})$, where $e$ is an edge and $f$ a face of some polyhedron.  In either case, the endpoints belong to different classes.  \end{proof}

Say a cell of $\s$ is \textit{external} if it is contained in $\s \cap \c^{(2)}$, and \textit{internal} otherwise. Each square of $\s$ has two adjacent external edges, of the form $(\bar{e},\bar{f})$ and $(\bar{e},\bar{f}')$ in the notation above, and two internal edges $(\bar{f},\olp_i)$ and $(\bar{f}',\olp_i)$.  In particular, each external edge of each square is opposite an internal edge, and vice-versa.

\begin{lemma} \label{external_spine}  As one-subcomplexes, $\s \cap \c^{(2)} = (\c^{(2)})' - \mathfrak{st}(\c^{(0)})$, where $(\c^{(2)})'$ is the barycentric subdivision of $\c^{(2)}$.  In particular, $\Phi$ restricts to a deformation retraction from $\bigcup_i \partial \p_i$ to $|\s \cap \c^{(2)}|$.  \end{lemma}

\proof  By definition $\s \cap \c^{(2)} = \s_0 \cap \c^{(2)}$, whence the first claim of the lemma follows.  The second claim now holds because $\Phi$ is cellular.  \endproof

\begin{lemma} \label{two-sided}  Suppose $H$ is a hyperplane of the standard square complex associated to an ideal polyhedral decomposition $\{\p_i\}$ of a complete hyperbolic $3$-manifold $M$, and let $p \co N \to H$ be the regular neighborhood of $H$.  $N$ has boundary components $\partial_e N$ and $\partial_i N$, mapped by $j \co N \to M$ to a union of external and internal edges, respectively.  \end{lemma}

\begin{proof}  Let $s$ be a square of $\s$.  The vertices of $s$ are the barycenters of $\bar{e}$, $\bar{f}$, $\bar{g}$, and $\olp_i$, where $\p_i$ is a polyhedron in the decomposition of $M$, $e$ is an edge of $\p_i$, and $f$ and $g$ are the faces of $\p_i$ intersecting in $e$.  One midline of $s$ has vertices on the midpoints of the opposite  edges $(\bar{e},\bar{f})$ and $(\bar{g},\olp_i)$ of $s$, and the other has vertices on the midpoints of $(\bar{f},\olp_i)$ and $(\bar{e},\bar{g})$.  Take $H$ to be the hyperplane containing the midline $m$ with vertices on $(\bar{e},\bar{f})$ and $(\bar{g},\olp_i)$.

Let $s_0 = p^{-1}(m) \subset N$; then $s_0$ is a square which $j$ maps homeomorphically to $s$.  The edges of $s_0 \cap \partial N$ are mapped by $j$ to the edges of $s$ parallel to $m$.  These are $(\bar{f},\olp_i)$, which is internal, and $(\bar{e},\bar{g})$, which is external.  Let $b_i$ be the edge mapped to $(\bar{f},\olp_i)$ by $j$, let $b_e$ be mapped to $(\bar{e},\bar{g})$, and let $\partial_i N$ and $\partial_e N$ be the components of $\partial N$ containing $b_i$ and $b_e$, respectively.  It is \textit{a priori} possible that $\partial_i N = \partial_e N$, but we will show that $\partial_i N$ (respectively, $\partial_e N$) is characterized by the fact that its edges map to internal (resp, external) edges of $\s$.

Let $s_1$ be a square of $N$ adjacent to $s_0$.  Then the edge $m_1 \doteq p(s_1)$ of $H$ is the midline of the square $s' = j(s_1)$ adjacent to $s$.  Suppose first that $s'$ meets $s$ along the external edge $(\bar{e},\bar{f})$.  Then there is a polyhedron $\p_j$ of the decomposition with a face $f'$ and edge $e' \subset f$ with $\phi_f(f) = f'$ and $\phi_f(e)=e'$ (ie, $f$ and $f'$ represent the same face of the decomposition of $M$, and $e$ and $e'$ the same edge), such that the vertices of $s'$ are the barycenters of $\bar{e}'$, $\bar{f}'$, $\bar{g}_1$, and $\olp_j$.  Here $g_1$ is the other face of $\olp_j$ containing $e'$.

Since $m_1$ meets $m$, it has an endpoint at the midpoint of $(\bar{e}',\bar{f}')$, which is identified with $(\bar{e},\bar{f})$ in $M$.  Then the other endpoint of $m_1$ is on the opposite edge $(\bar{g}_1,\olp_j)$ of $s'$.  The external edge $(\bar{e}',\bar{g}_1)$ of $s'$ which is parallel to $m_1$ meets the external edge $(\bar{e},\bar{g})$ of $s$ at the barycenter of the edge of the decomposition represented by $e$ and $e'$.  It follows that $j$ maps the edge of $s_1 \cap \partial N$ adjacent to $b_e$ to $(\bar{e},\bar{g})$.  Likewise, the edge of $s_1 \cap \partial N$ adjacent to $b_i$ is mapped to the internal edge $(\bar{f}',\olp_j)$ of $s'$.

Now suppose $s'$ meets $s$ along the internal edge $(\bar{g},\olp_i)$.  Then there is an edge $e_1$ of $g$ such that the vertices of $s'$ are the barycenters of $\bar{e}_1$, $\bar{g}$, $\bar{f}_1$, and $\olp_i$.  Here $f_1$ is the other face of $\p_i$ containing $e_1$.  Then $m_1$ meets $m$ at the midpoint of $(\bar{g},\olp_i)$.  Since $b_e$ is mapped by $j$ to $(\bar{e},\bar{g})$, the edge of $s_1 \cap \partial N$ adjacent to it is mapped to the external edge $(\bar{e}_1, \bar{g})$.  It follows that the other edge of $s_1 \cap \partial N$ is mapped to the internal edge $(\bar{f}_1,\olp_i)$ of $s'$ parallel to $m_1$.

The above establishes that the union of the set of edges of $\partial_i N$ mapped to internal edges of $\s$ is open and nonempty in $\partial_i N$.  Since it is clearly also closed, it is all of $\partial_i N$.  An analogous statement holds for $\partial_e N$, establishing the lemma.
\end{proof}

It is occasionally useful to think of the standard square complex associated to an ideal polyhedral decomposition as a subdivision of the ``dual two-complex''.  If $\c$ is the cell complex associated to the ideal polyhedral decomposition $\{\calp_i\}$, let $D\c$ be the two-complex with a vertex at the barycenter of each $3$-cell of $\c$, for each $f \in \c^{(2)}$ an edge $Df$ crossing $f$, and for each $e \in \c^{(1)}$ a face $De$ crossed by $e$.  The standard square complex $\s$ is obtained from $D\c$ by dividing each face along its intersections with the $2$-cells of $\c$ which meet at the edge.

\begin{lemma} \label{nonpos}  Suppose $\{\p_i\}$ is a decomposition of $M$ into right-angled ideal polyhedra.  The standard square complex $\s$ associated to $\{\p_i\}$ is non-positively curved.
\end{lemma}

\proof  Recall that $\s$ is non-positively curved if and only if in the link of any vertex, each simple cycle has length at least $4$.  If $v$ is a vertex of $\s$, a simple cycle of length $k$ in the link of $v$ is a sequence of squares $s_0,s_1, \hdots,s_{k-1}$ with the following properties: for each $i$ there is an edge $e_i$ with $v \subset e_i \subset s_i \cap s_{i+1}$ (taking $i+1$ modulo $k$), and $s_i \neq s_j$ and $e_i \neq e_j$ when $i \neq j$.

Since the decomposition $\{\p_i\}$ is into right-angled polyhedra, the dual two-complex $D\c$ described above the lemma is a square complex.  This follows from the fact that each edge of $\c$ is contained in four faces of $\c$.  We will show that $D\c$ is non-positively curved; since $\s$ is a subdivision of $D\c$, it will follow that $\s$ is non-positively curved.

Suppose $v$ is a vertex of $D\c$, and let $\{De_0,\hdots,De_{k-1}\}$ be a simple cycle in the link of $v$ in $D\c$.  The associated sequence of edges $\{ Df_0, \hdots,Df_{k-1}\}$ determines a sequence of distinct faces $\{f_0,\hdots, f_{k-1}\}$ of the polyhedron $\p_i$ containing $v$, each meeting the next in an edge.   It follows immediately from the necessary conditions of Andreev's theorem \cite{Andreev} and the fact that $\p_i$ is right-angled that every such cycle has length at least four.  The conclusion of the lemma follows.
\endproof

%%%%%%%%%%%%%%%%%%%%%%%%%%%%%%%%%%%%%%%%%%%
\section{Totally geodesic hyperplane groups}  \label{sec:geodesic hyperplanes}
%%%%%%%%%%%%%%%%%%%%%%%%%%%%%%%%%%%%%%%%%%

Fix an orientable, complete hyperbolic manifold $M = \mathbb{H}^3/\Gamma$ of finite volume, equipped with a decomposition $\{\p_i\}$ into right-angled ideal polyhedra.  Here we have identified $M$ with the quotient of $\mathbb{H}^3$ by a discrete group of isometries $\Gamma$, thus identifying $\pi_1M$ with $\Gamma$.  Let $\s$ be the standard square complex associated to the polyhedral decomposition as in Section \ref{std square}.  The goal of this section is, for each hyperplane $H \to X$, to identify a totally geodesic surface immersed in $M$ which ``carries'' $H$.

Since each $\p_i$ is right-angled and the angle in $M$ around each edge is $2\pi$, the equivalence class of each edge has four members.  If $f$ represents a face of the decomposition and $e$ an edge of $f$, define the {\it flat $e$-neighbor of $f$} to be the face of the decomposition that meets $f$ at angle $\pi$ along $e$ in $M$.

If $\p_i$ is the polyhedron containing $f$, let $g$ be the other face of $\p_i$ containing $e$.  Let $g' = \phi_{g}(g)$, a face of some polyhedron $\p_j$, and let $e' = \phi_{g}(e)$.  Then $e$ and $e'$ represent the same edge of the decomposition, and the flat $e$-neighbor of $f$ is represented by the face $f_1$ of $\p_j$ which intersects $g'$ along $e'$.  Let $\Sigma_f$ be the collection of faces of the decomposition, minimal with respect to inclusion, satisfying the properties below.
\begin{enumerate}
\item $f \in \Sigma_f$, and
\item if $g\in \Sigma_f$ and $e$ is an edge of $g$, then every flat $e$-neighbor of $g$ is in $\Sigma_f$.
\end{enumerate}
Note that if $g \subset \Sigma_f$ is a 2-cell then $\Sigma_f = \Sigma_g$. Furthermore, there is a sequence $\{f = f_0, f_1, \hdots, f_n = g\}$ such that for each $i>0$ there is an edge $e_i$ with $f_i$ a flat $e_i$-neighbor of $f_{i-1}$.  Call such a sequence a {\it path of flat neighbors.}

Now let $\widehat{\Sigma}_f$ be the quotient of $\Sigma_f$ by the following edge pairings: if $g$ represents an element of $\Sigma_f$ and $e$ is an edge of $g$, glue $g$ to its flat $e$-neighbor $g'$ by the restriction of the face pairing isometry $\phi_{g}$ described above.  Since each face of the decomposition has a unique flat $e$-neighbor along each of its edges, $\widehat{\Sigma}_f$ is topologically a surface without boundary.  It is connected, since any two faces in $\Sigma_f$ are connected by a path of flat neighbors, and it inherits a hyperbolic structure from its faces, since the edge gluing maps are isometries.

The inclusion maps of faces $\{ g \hookrightarrow \p_i\, |\, g \subset \p_i,\ g \in \Sigma_f\}$ determine an immersion from $\widehat{\Sigma}_f$ to $\bigsqcup_i \p_i/ \sim$.  This is not necessarily an embedding because the preimage of an edge may consist of two edges of $\widehat{\Sigma}_f$, each mapped homeomorphically.  However, by construction it is a local isometry.

\begin{lemma}  \label{lem:geodesic_immersion}  Let $i \co \widehat{\Sigma}_f \to M$ be the composition of the inclusion-induced map to $\bigsqcup_i \p_i/\sim$ with the isometry to $M$.  Then $i$ is a proper  immersion which maps onto its image with degree one.
\end{lemma}

\begin{proof}  If $g$ is a face of $\p_i$, the inclusion $g \hookrightarrow \p_i$ is proper by definition.  Since the collection $\{\p_i\}$ is finite, it follows that $i$ is proper.  By construction, the interior of each face in $\Sigma_f$ is mapped homeomorphically by $i$, thus it has degree one onto its image.  \end{proof}

Since the map $i \co \widehat{\Sigma}_f \to M$ is a proper local isometry and $M$ is complete, the hyperbolic structure on $\widehat{\Sigma}_f$ is complete.  Since it is contained in the union of finitely many polygons of finite area, $\widehat{\Sigma}_f$ has finite area.  Choosing an isometric embedding of $f$ in $\mathbb{H}^2$ thus determines a developing map identifying the universal cover of $\widehat{\Sigma}_f$ with $\mathbb{H}^2$, and identifying $\pi_1 \widehat{\Sigma}_f$ with a subgroup $\Gamma_f$ of $\mathrm{Isom}(\mathbb{H}^2)$.

Now fix a component $\tilde{f}$ of the preimage of $i(f)$ under the universal cover $\mathbb{H}^3 \rightarrow M$.  This choice determines a lift $\tilde{\imath} \co \mathbb{H}^2 \to \mathbb{H}^3$ of $i \co \widehat{\Sigma}_f \to M$, equivariant with respect to the actions of $\Gamma_f$ on $\mathbb{H}^2$ and $i_* (\pi_1 \widehat{\Sigma})$ on $\mathbb{H}^3$.

\begin{lemma} \label{stabilizer}  Let $\mathcal{H}$ be the geodesic hyperplane of $\mathbb{H}^3$ containing $\tilde{f}$.  Then $\tilde{\imath}$ maps $\mathbb{H}^2$ isometrically onto $\calh$, and $i_*$ takes $\pi_1 \widehat{\Sigma}_f$ isomorphically onto $\mathrm{Stab}_{\Gamma}(\calh)$.  \end{lemma}

\begin{proof}  Since $i$ is a local isometry, $\tilde{\imath}$ maps $\mathbb{H}^2$ isometrically onto the geodesic hyperplane in $\mathbb{H}^3$ containing $\tilde{\imath}(f) = \tilde{f}$, hence $\calh$.  Since $\pi_1 \widehat{\Sigma}_f$ acts faithfully on $\mathbb{H}^2$ by isometries, its action on $\calh$, and hence all of $\mathbb{H}^3$ is also faithful.  If $i_*(\pi_1 \widehat{\Sigma}_f)$ were properly contained in $\mathrm{Stab}_{\Gamma}(\calh)$, the embedding $i$ would factor through the covering map $\calh/i_*(\pi_1 \widehat{\Sigma}_f) \to \calh/\mathrm{Stab}_{\Gamma}(\calh)$, contradicting the fact that $i$ maps onto its image with degree one.
\end{proof}

Let us now take $\Gamma_f = i_* (\pi_1\widehat{\Sigma}_f)$ and $\widehat{M}_f = \mathbb{H}^3/\Gamma_f$.   By Lemma \ref{stabilizer}, $i \co \widehat{\Sigma}_f \to M$ lifts to an embedding $\hat{\imath}$ to $\widehat{M}_f$, such that $\widehat{M}_f$ is homeomorphic to $\hat{\imath}(\widehat{\Sigma}_f) \times \mathbb{R}$.  We thus obtain the following diagram.

\[ \xymatrix{  \mathbb{H}^2 \ar[d] \ar[r]^{\tilde{\imath}} & \mathbb{H}^3 \ar[d] \\
\widehat{\Sigma}_f  \cong \mathbb{H}^2/\Gamma_f  \ar[dr]_{i} \ar[r]^{\hat{\imath}} & \widehat{M}_f := \mathbb{H}^3/\Gamma_f \ar[d] \\
& M}
 \]
Below we will refer by $\widehat{\Sigma}_f \subset \widehat{M}_f$ to the image of $\hat{\imath}$.

\begin{dfn}  Let $M$ be a complete, orientable, hyperbolic $3$-manifold of finite volume equipped with a decomposition $\{\p_i\}$ into right-angled ideal polyhedra, and suppose $H$ is a hyperplane of the associated square complex, with regular neighborhood $(N,p,j)$.  Choose a midline $m$ of $H$, let $s = p^{-1}(m)$, and let $\p_i$ contain $j(s)$.  There is a unique face $f$ of $\p_i$ containing the external edge of $j(s\cap \partial N)$, and we define $\widehat{\Sigma}(H) = \widehat{\Sigma}_f$, $\Gamma(H) = \Gamma_f$, and $\widehat{M}(H) = \widehat{M}_f$.  \end{dfn}

\begin{lemma}  \label{lem:nbhd_lift}  Using notation from the definition above, let $\widehat{\s}$ be the standard square complex associated to the decomposition $\widehat{M}(H)$ inherits from $\{\p_i\}$.  Then $j \co N \to \s$ lifts to an immersion $\hat{\jmath}$ to $\widehat{\s}$, taking $\partial_e N$ to a spine of $\widehat{\Sigma}(H)$, such that $\hat{\jmath}|_{\partial_e N}$ is an embedding if $\widehat{\Sigma}(H)$ is orientable, and a two-to-one cover if not.  \end{lemma}

\begin{cor} \label{geod sfces carry}  If $\Sigma(H)$ is orientable, $\pi_1 H = \Gamma(H)$; otherwise $\pi_1 H$ is the index-two orientation-preserving subgroup of $\Gamma(H)$.  \end{cor}

\begin{proof}[Proof of Lemma \ref{lem:nbhd_lift}]  Suppose $m_0$ and $m_1$ are two adjacent midlines of $H$, and let $s_0 = p^{-1}(m_0)$ and $s_1 = p^{-1}(m_1)$ in $N$.  Take $\p_{i_0}$ and $\p_{i_1}$ to be the polyhedra containing $j(s_0)$ and $j(s_1)$, respectively, and let $f_0$ be the face of $\p_{i_0}$ and $f_1$ the face of $\p_{i_1}$ containing $j(s_0 \cap \partial_e N)$ and $j(s_1 \cap \partial_e N)$.  If $m_0$ meets $m_1$ at the midpoint of an internal edge of $\s$, it is clear that $\p_{i_0} = \p_{i_1}$ and $f_0 = f_1$.

If $m_0$ meets $m_1$ in an external edge of $\s$, then $\p_{i_0}$ and $\p_{i_1}$ abut in $M$ along a face of the decomposition.  Let $g \subset \p_{i_0}$ represent this face of the decomposition.  Then $g$ and $f_0$ meet along an edge $e$, and $g' = \phi_g(g) \subset \p_{i_1}$ and $f_1$ meet along $e' = \phi_g(e)$.  Hence if $m_0$ meets $m_1$ in an external edge of $\s$, there is an edge $e$ of the decomposition of $M$ such that $f_0$ and $f_1$ represent flat $e$-neighbors.  It follows that a sequence of edges $m_0,m_1,\hdots,m_k$ of $H$, with the property that $m_i$ is adjacent to $m_{i-1}$ for each $i>0$, determines a path of flat neighbors in $\Sigma(H)$.  Therefore $j$ maps $\partial_e N$ into $i(\widehat{\Sigma}(H))$.

Now let $f$ be a face of some polyhedron $\p_i$ representing a face of $\Sigma(H)$.  The cover $\widehat{M}(H)$ inherits a polyhedral decomposition from that of $M$, and since the covering map is injective on a neighborhood of $\hat{\imath}(f)$, there is a unique polyhedron $\widehat{\p}_i$ of this decomposition with the property that $\widehat{\p}_i$ projects to $\p_i$ and contains $\hat{\imath}(f)$.
For a square $s$ of $N$, we thus define $\hat{\jmath}(s)$ to be the component of the preimage of $j(s)$ contained in $\widehat{\p}_i$, where $\p_i$ is the polyhedron containing $j(s)$.

Suppose $\p_{i_0}$ and $\p_{i_1}$ contain faces $f_0$ and $f_1$, respectively, each representing a face of $\Sigma(H)$, which are flat $e$-neighbors for some edge $e$.  Let $g \subset \p_{i_0}$ satisfy $g \cap f_0 = e$ and $\phi_g(e) = \phi_g(g) \cap f_1 \subset \p_{i_1}$.  Since $\hat{\imath}(f_0)$ and $\hat{\imath}(f_1)$ meet in $\widehat{M}(H)$ along the preimage of $e$, $\widehat{\p}_{i_0}$ and $\widehat{\p}_{i_1}$ meet along the face represented by the preimage of $g$.  For adjacent squares $s_0$ and $s_1$ in $N$, it follows that if $j(s_0)$ and $j(s_1)$ meet along an external edge of $\s$, then $\hat{\jmath}(s_0)$ and $\hat{\jmath}(s_1)$ meet along an external edge of $\widehat{\s}$.

If $s_0$ and $s_1$ are adjacent squares of $N$ such that $j(s_0)$ meets $j(s_1)$ in an internal edge of $\s$ contained in a polyhedron $\p_i$, then $\hat{\jmath}(s_0)$ meets $\hat{\jmath}(s_1)$ in $\widehat{\p}_i$.   Thus $\hat{\jmath}$ is continuous.  Since $j$ is an immersion, $\hat{\jmath}$ is an immersion as well.  We claim $\hat{\jmath}$ maps $\partial_e N$ onto $\widehat{\s} \cap \widehat{\Sigma}(H)$.

Since $\hat{\jmath}$ is continuous, the image of $\partial_e N$ is closed in $\widehat{\s} \cap \widehat{\Sigma}(H)$.  Now suppose $e_0$ and $e_1$ are adjacent edges of $\widehat{\s} \cap \widehat{\Sigma}(H)$ such that $e_0 \subset \hat{\jmath}(\partial_e N)$.  Let $s_0 \subset N$ be a square such that $\hat{\jmath}(s_0)$ contains $e_0$, and let $m_0 = p(s_0)$ be a midline of $j(s_0)$.  There is a square $s$ of $\s$, containing the projection of $e_1$ to $M$, such that $s \cap j(s_0)$ is a union of edges containing the projection of $e_0 \cap e_1$.  Let $m_1$ be the midline of $s$ meeting $m_0$; then $m_1 \in H$, so by definition $s_1 = p^{-1}(m_1)$ is mapped by $j$ to $s$.  Now from the above it follows that $\hat{\jmath}(s_1)$ contains $e_1$.  This implies that $\hat{\jmath}(\partial_e N)$ is open in $\widehat{\s} \cap \widehat{\Sigma}(H)$ and proves the claim.

Lemma \ref{external_spine} implies that $\widehat{\s} \cap \widehat{\Sigma}(H)$ is a spine for $\widehat{\Sigma}(H)$, hence $\hat{\jmath}$ maps $\partial_e N$ onto a spine of $\widehat{\Sigma}(H)$.  Each square $s \subset N$ has the property that $s \cap \partial_e N$ is the unique edge of $s$ mapped by $\hat{\jmath}$ into $\widehat{\Sigma}(H)$.  For let $f\subset \Sigma(H)$ be the face of $\p_i$ containing $j(s \cap \partial_e N)$, where $\p_i$ contains $j(s)$, let $g$ be the face containing the other external edge of $j(s)$, and let $f_1$ be the flat $e$-neighbor of $f$, where $e = f \cap g$.  Then $\hat{\imath}(f)$ and $\hat{\imath}(f')$ are in $\widehat{\Sigma}(H)$.  If the face $\hat{g}$ of $\widehat{\p}_i$ adjacent to $\hat{\imath}(f)$ were also in $\hat{\imath}(\widehat{\Sigma}(H))$, $\hat{\imath}$ would not be an embedding.

Now suppose $\hat{\jmath}(s_0) = \hat{\jmath}(s_1)$ for squares $s_0$ and $s_1$ of $N$.  By the property above, there is an edge $e$ of $\widehat{\s} \cap \widehat{\Sigma}(H)$ such that $\hat{\jmath}(s_0 \cap \partial_e N) = e = \hat{\jmath}(s_1 \cap \partial_e N)$.  It follows that $j$ maps the external edge of each of $s_0$ and $s_1$ to the projection of $e$ in $M$.  By definition, $p(s_0)$ is the midline of $j(s_0)$ parallel to $j(s_0 \cap \partial_e N)$, and the same holds true for $s_1$.  Thus $p(s_0) = p(s_1)$, so $s_0 = s_1$.

The paragraph above implies that $\hat{\jmath}|_{\partial_e N}$ is at worst two-to-one, since each external edge of $\widehat{\s}$ is contained in exactly two squares.  Since $\widehat{M}(H)$ is orientable, if $\widehat{\Sigma}(H)$ is orientable as well, then it divides any sufficiently small regular neighborhood into two components.  Since $N$ is connected and $\hat{\jmath}$ is continuous, in this case its image is on one side of $\widehat{\Sigma}(H)$, so $\hat{\jmath}|_{\partial_e N}$ is an embedding.

If $\widehat{\Sigma}(H)$ is nonorientable, then a regular neighborhood is connected.  Thus in this case, for any edge $e$ of $\widehat{\s} \cap \widehat{\Sigma}(H)$, both squares containing $e$ are in the image of $\hat{\jmath}$, and the restriction to $\partial_e N$ maps two-to-one.
\end{proof}

The final result of this section characterizes some behaviors of hyperplanes of $\s$ in terms of the behavior of their associated totally geodesic surfaces.  Below we say distinct hyperplanes $H_1$ and $H_2$ are \textit{parallel} if $\Sigma(H_1) = \Sigma(H_2)$.

\begin{prop} \label{hyperplanes vs surfaces}  Let $M$ be a complete, orientable hyperbolic $3$-manifold equipped with a decomposition $\{\p_i\}$ into right-angled ideal polyhedra, with associated standard square complex $\s$, and let $H_1$ and $H_2$ be hyperplanes of $\s$.  If $H_1$ osculates $H_2$ along an external edge of $\s$, then either \begin{enumerate}
\item $H_1 = H_2$ and $\Sigma(H_1)$ is nonorientable; or
\item $H_1$ and $H_2$ are parallel and $\Sigma(H_1)=\Sigma(H_2)$ is orientable.  \end{enumerate}
$H_1$ intersects $H_2$ if and only if $i(\widehat{\Sigma}(H_1))$ intersects $i(\widehat{\Sigma}(H_2))$ at right angles.  \end{prop}

\begin{proof}  Suppose $H_1$ osculates $H_2$ along an external edge $e$.  Then there are squares $s_1$ and $s_2$ of $\s$ intersecting along $e$, such that the midline $m_1$ of $s_1$ parallel to $e$ is in $H_1$, and the midline $m_2 \subset s_2$ parallel to $e$ is in $H_2$.  If $f$ is the face of the decomposition containing $e$, then by definition $f \in \Sigma(H_1)$ and $f \in \Sigma(H_2)$.  Since $s_1$ and $s_2$ are on opposite sides of $f$ in $M$, Lemma \ref{lem:nbhd_lift} implies alternatives $\mathit{1}$ and $\mathit{2}$.

Suppose $H_1$ intersects $H_2$ in a square $s$ contained in some polyhedron $\p_i$, and for $j = 0,1$ let $m_j$ be the midline of $s$ in $H_j$.  For each $j$, there is a unique external edge $e_j$ of $s$ parallel to $m_j$.  By definition, the faces $f_1$ and $f_2$ of $\p_i$ containing $e_1$ and $e_2$ are contained in $\Sigma(H_1)$ and $\Sigma(H_2)$, respectively.  Since $\p_i$ is right-angled they meet at right angles, establishing the lemma.  \end{proof}

%%%%%%%%%%%%%%%%%%%%%%%%%%%%%%%%
\section{Embedding in Coxeter groups} \label{sec:separability}
%%%%%%%%%%%%%%%%%%%%%%%%%%%%%%%%

Let $M = \mathbb{H}^3/\Gamma$ be a complete, orientable hyperbolic $3$-manifold of finite volume, equipped with a decomposition $\{\p_i\}$ into right-angled ideal polyhedra.  In this section we describe separability properties of hyperplane subgroups which allow pathologies to be removed in finite covers of $M$.

If $H$ is a subgroup of a group $G$, we say $H$ is \textit{separable} in $G$ if for each $g \in G - H$ there is a subgroup $K$, of finite index in $G$, such that $H <K$ and $g \notin K$.  The separability result needed for the proof of Theorem 1.1 follows from \cite[Lemma 1]{Long} and extends its conclusion to a slightly more general class of subgroups.

\newcommand\Longlemma{Let $M = \mathbb{H}^3/\Gamma$ be a complete, orientable hyperbolic $3$-manifold with finite volume, and let $\calh \subset \mathbb{H}^3$ be a hyperplane such that $\mathrm{Stab}_{\Gamma}(\calh)$ acts on $\calh$ with finite covolume.  Then the subgroup of $\mathrm{Stab}_{\Gamma}(\calh)$ that acts preserving an orientation of $\calh$ is separable in $\Gamma$.}
\begin{lemma}[Cf.\ \cite{Long} Lemma 1] \label{Long} \Longlemma \end{lemma}
\theoremstyle{plain} \newtheorem*{Longlemma_lemma}{Lemma \ref{Long}}
\begin{proof}
It follows from \cite[Lemma 1]{Long} that $\mathrm{Stab}_\Gamma(\calh)$ is separable.  It remains to consider the case in which $\mathrm{Stab}_\Gamma(\calh)$ is orientation-reversing on $\calh$ and to show that the orientation-preserving subgroup is separable.

As in \cite[Theorem 1]{Long}, there is a finite-sheeted covering $M'\to M$ such that the immersed surface $\calh/\mathrm{Stab}_\Gamma(\calh)$ lifts to an embedded surface $\Sigma$ in $M'$.  Because $M'$ is orientable, the surface $\Sigma$ is one-sided.  Let $N$ be a closed regular neighbourhood of $\Sigma$ and let $M_0$ be the complement of the interior of $N$ in $M'$.  The boundary of $N$ is homeomorphic to $\widetilde\Sigma$, the orientable double cover of $\Sigma$.

The neighbourhood $N$ has the structure of a twisted interval bundle over $\Sigma$, so $\pi_1N\cong\pi_1\Sigma$.  The double cover $\widetilde N$ of $N$ obtained by pulling back the bundle structure along the covering map $\widetilde\Sigma\to\Sigma$ is an orientable interval bundle over $\widetilde\Sigma$ and hence homeomorphic to the product $\widetilde\Sigma\times [-1,+1]$.  This homeomorphism can be chosen so that $\widetilde\Sigma\times\{0\}$ double covers $\Sigma$.

The inclusion map $i:\partial N\hookrightarrow N$ has precisely two lifts to $\widetilde N$; let $i^\pm$ be the lift that identifies $\partial N$ with $\widetilde\Sigma\times\{\pm 1\}$.  Construct a new manifold $\widetilde M$ as follows: let $M_0^\pm$ be two copies of $M_0$ and let $\partial^\pm N$ be the corresponding copy of $\partial N$ in $M^\pm$; then $\widetilde M$ is obtained from
\[
M_0^+\sqcup \widetilde N\sqcup M_0^-
\]
by identifying $x\in \partial^\pm N$ with $i^\pm(x)$.  By construction, $\widetilde M$ is a double cover of $M'$ and so a finite-sheeted cover of $M$.  The image of $\calh$ in $\widetilde M$ is precisely the orientable double cover of $\Sigma$, so $\pi_1\widetilde M$ is a finite-index subgroup of $\Gamma$ that contains the orientation-preserving elements of $\mathrm{Stab}_\Gamma(\calh)$ but not the orientation-reversing ones, as required.
\end{proof}

If $H$ is a hyperplane of the standard square complex associated to the decomposition of $M$ into right-angled ideal polyhedra, Lemma \ref{stabilizer} and Corollary \ref{geod sfces carry} together describe a geodesic hyperplane $\calh$, such that $\mathrm{Stab}_{\Gamma}(\calh)$ acts on it with finite covolume and $\pi_1 H$ is the subgroup which preserves an orientation of $\calh$.  Thus:

\begin{cor} \label{Longsep}  Suppose $M = \mathbb{H}^3/\Gamma$ is a complete, orientable hyperbolic $3$-manifold of finite volume that admits a decomposition $\{\p_i\}$ into right-angled ideal polyhedra.  If $H$ is a hyperplane of the standard square complex associated to $\{\p_i\}$, then $\pi_1 H$ is separable in $\Gamma$.  \end{cor}

This implies, using \cite[Corollary 8.9]{HW}, that a hyperbolic manifold $M$ with a right-angled ideal polyhedral decomposition has a finite cover whose associated square complex lacks most pathologies forbidden in the definition of special complexes.

\begin{prop} \label{no self osculate}  Suppose $M = \mathbb{H}^3/\Gamma$ is a complete, orientable hyperbolic $3$-manifold with finite volume that admits a decomposition into right-angled ideal polyhedra $\{\p_i\}$.  There is a cover $M' \to M$ of finite degree such that hyperplanes of the standard square complex of $M'$ do not self-intersect or -osculate.  \end{prop}

\begin{proof}  Let $X$ be the standard square complex associated to $\{\p_i\}$.  Lemma \ref{spine} implies that the inclusion $X \hookrightarrow M$ induces an isomorphism $\pi_1 X \to \Gamma$.  By Corollary \ref{Longsep}, each hyperplane subgroup is separable in $\pi_1 X$, so by \cite[Corollary 8.9]{HW}, $X$ has a finite cover $X'$ such that hyperplanes of $X'$ do not self-intersect or -osculate.  Let $\Gamma'$ be the subgroup of $\pi_1 X$ corresponding to $X'$, and let $M' \to M$ be the cover corresponding to $\Gamma'$.  The decomposition $\{\p_i\}$ of $M$ lifts to a right-angled ideal decomposition of $M'$ with standard square complex $X'$, proving the proposition.
\end{proof}

Proposition \ref{no self osculate} already implies that a large class of hyperbolic $3$-manifolds is virtually special.  Below we will say that the decomposition $\{\p_i\}$ of $M$ is \textit{checkered} if the face pairing preserves a two-coloring --- an assignment of white or black to each face $f$ of each $\p_i$ such that if another face $f'$ of $\p_i$ intersects $f$ in an edge, it has the opposite color.  The decompositions of augmented link complements described in the appendix to \cite{LAD} are checkered, for example.

\begin{thm} \label{check cox}  Suppose $M$ is a complete hyperbolic $3$-manifold with finite volume that admits a checkered decomposition into right-angled ideal polyhedra.  Then $\pi_1M$ has a subgroup of finite index that is isomorphic to a word-quasiconvex subgroup of a right-angled Coxeter group.  \end{thm}

\begin{proof}  Let $M = \mathbb{H}^3/\Gamma$ be a complete hyperbolic $3$-manifold of finite volume with a decomposition $\{\p_i\}$ into right-angled polyhedra.  If the decomposition is checkered, and $f$ represents a face of the decomposition, it is easy to see that for each edge $e \subset f$, the flat $e$-neighbor of $f$ has the same color as $f$.  It follows that each face of the surface $\Sigma_f$ described in Section \ref{sec:geodesic hyperplanes} has the same color as $f$.  If $H$ is a hyperplane of the square complex $X$ associated to $\{\p_i\}$, we will say $H$ is white if all faces of $\Sigma(H)$ are white, and black if they are black.

By Proposition \ref{hyperplanes vs surfaces}, a hyperplane intersects only hyperplanes of the opposite color and osculates only hyperplanes of the same color along an external edge.  If hyperplanes $H_0$ and $H_1$ osculate along an internal edge, let $s_0$ and $s_1$ be squares of $\s$, meeting along an internal edge $e$, with parallel midlines $m_0 \in H_0$ and $m_1 \in H_1$.  Then $e$ is of the form $(\bar{g},\olp_i)$, where $\p_i$ is the polyhedron containing $s_0$ and $s_1$ and $g$ is a face of $\p_i$.  The edges of $s_0$ and $s_1$ opposite $e$ are contained in faces $f_0$ and $f_1$ of $\p_i$ in $\widehat{\Sigma}(H_0)$ and $\widehat{\Sigma}(H_1)$, respectively.  Then each of $f_0$ and $f_1$ intersects $g$, so the color of $f_0$ and $f_1$ is opposite that of $g$.  It follows that hyperplanes of $\s$ do not inter-osculate.

By Proposition \ref{no self osculate}, $M$ has a finite cover $M'$ such that hyperplanes of the square complex $X'$ associated to the lifted ideal polyhedral decomposition of $M'$ do not self-intersect or -osculate.  The lifted ideal polyhedral decomposition of $M'$ inherits the checkered property from that of $M$, so by the above, hyperplanes of $X'$ do not inter-osculate.  In addition, Lemma \ref{nonpos} implies that $X'$ is nonpositively curved, Lemma \ref{two-sided} implies that each hyperplane is two-sided, and Lemma \ref{bipartite} implies that $X'^{(1)}$ is bipartite.  Thus $X'$ is $C$-special, and by Theorem \ref{t: Haglund--Wise}, the subgroup $\Gamma' < \Gamma$ corresponding to $M'$ embeds as a word-quasiconvex subgroup of a right-angled Coxeter group.  \end{proof}

In fact, we will show below that every right-angled decomposition determines a twofold cover whose associated decomposition is checkered.  This uses the lemma below, which is a well known consequence of Andreev's theorem.

\begin{lemma}  Let $\calp \subset \mathbb{H}^3$ be a right-angled ideal polyhedron of finite volume.  There are exactly two checkerings of the faces of $\calp$.  \end{lemma}

Theorem \ref{rt ang cox} follows quickly from this lemma and Theorem \ref{check cox}.

\begin{proof}[Proof of Theorem \ref{rt ang cox}]  Suppose $\{ \calp_i \}_{i=1}^n$ is a right-angled ideal decomposition of $M$.  Let $\{ \calp_i^{(0)}, \calp_i^{(1)} \}_{i=1}^n$ be a collection of disjoint right-angled polyhedra such that for each $i$, $\calp_i^{(0)}$ and $\calp_i^{(1)}$ are each isometric to $\calp_i$, and the faces of $\calp_i^{(0)}$ have the opposite checkering of the faces of $\calp_i^{(1)}$.  Here we take for granted that we have fixed marking isometries $\calp_i^{(j)} \to \calp_i$ for each $j \in \{0,1\}$, so that each face $f$ of $\calp_i$ has fixed correspondents  $f^{(0)} \subset \calp_i^{(0)}$ and $f^{(1)} \subset \calp_i^{(1)}$.

For each $i$ and each face $f$ of $\calp_i$, we determine face pairing isometries $\phi_{f^{(0)}}$ and $\phi_{f^{(1)}}$ for $\{ \calp_i^{(0)}, \calp_i^{(1)} \}$ using the following requirements:  each $\phi_{f^{(j)}}$, $j \in \{0,1\}$ must commute with $\phi_f$ under the marking isometries, and each must preserve color.  Thus if $f' = \phi_f(f)$ and $f'^{(0)}$ has the same color as $f^{(0)}$, we take $\phi_{f^{(j)}}(f^{(j)}) = f'^{(j)}$ for each $j$; otherwise we take $\phi_{f^{(j)}}(f^{(j)}) = f'^{(1-j)}$

Let $\widetilde{M}$ be the quotient of $\{ \calp_i^{(0)}, \calp_i^{(1)} \}_{i=1}^n$ by the face pairing isometries described above.  By construction, $\widetilde{M}$ is a double cover of $M$, and it is easy to see that $\widetilde{M}$ is disconnected if and only if the original decomposition $\{ \calp_i \}$ admits a checkering.  If it did, Theorem \ref{check cox} would apply directly to $M$, so we may assume that it does not.  Then, by Theorem \ref{check cox}, the conclusion of Theorem \ref{rt ang cox} applies to $\widetilde{M}$; hence it applies as well to $M$.
\end{proof}

%%%%%%%%%%%%%%%%%%%%%%%%%%%%%%%%%%%%%
\section{Virtual retractions and quasiconvexity}  \label{sec:QCERF}
%%%%%%%%%%%%%%%%%%%%%%%%%%%%%%%%%%%%%

This section contains the proof of Theorem \ref{rel qc still sep}.  We will need to work with various different definitions of quasiconvexity for subgroups.  These definitions all coincide in the case of a Gromov-hyperbolic group because Gromov-hyperbolic metric spaces enjoy a property sometimes known as the Morse Property, which asserts that quasigeodesics are uniformly close to geodesics.  In our case, $M$ has cusps and therefore $\Gamma=\pi_1M$ is not Gromov hyperbolic but rather relatively hyperbolic.  One of the results we use to circumvent this difficulty, Proposition \ref{p: Fully qc implies combinatorially qc}, makes use of of \cite[Theorem 1.12]{druu_tree-graded_2005}, which the authors call the `Morse Property for Relatively Hyperbolic Groups'.

\begin{dfn}
Let $X$ be a geodesic metric space.  A subspace $Y$ is \emph{quasiconvex} if there exists a constant $\kappa$ such that any geodesic in $X$ between two points of $Y$ is contained in the $\kappa$-neighbourhood of $Y$.
\end{dfn}

We will apply this notion in two contexts.  If $U$ is a CAT(0) cube complex with base vertex $v$ and a group $G$ acts properly discontinuously by combinatorial isometries on $U$ then we consider the one-skeleton $X=U^{(1)}$ with the induced length metric (where each edge has length one).  We say that a subgroup $H$ is \emph{combinatorially quasiconvex} if $Hv$ is a quasiconvex subspace of $X$.  In fact, combinatorial quasiconvexity is independent of the choice of basepoint if the action of $G$ on $U$ is special \cite[Corollary 7.8]{HW}.

On the other hand, given a group $G$ with a generating set $S$ we can consider the Cayley graph $Cay_S(G)$.  A subgroup $H$ is \emph{word quasiconvex} if $H$ is a quasiconvex subspace of $Cay_S(G)$.

Let $W$ be a right-angled Coxeter group with standard generating set $S$ and let $U$ be the universal cover of the Davis--Moussong complex for $W$.  The one-skeleton of $U$ is very closely related to $Cay_S(W)$: the edges of the Cayley graph come in pairs; identifying these pairs gives $U^{(1)}$.  Furthermore, the image of the universal cover of a special cube complex under the isometry defined by Haglund and Wise to the Davis--Moussong complex of $W$ is a convex subcomplex \cite[Lemma 7.7]{HW}.  We therefore have the following relationship between combinatorial quasiconvexity and word quasiconvexity in special cube complexes.

\begin{remark}
Suppose that $G$ is the fundamental group of a C-special cube complex $\s$, so that $G$ is isomorphic to a word-quasiconvex subgroup of a right-angled Coxeter group $W$ \cite{HW}.  If $H$ is a subgroup of $G$, then $H$ is combinatorially quasiconvex in $G$ (with respect to the action of $G$ on the universal cover of $\s$) if and only if $H$ is word quasiconvex in $W$ (with respect to the standard generating set).
\end{remark}

The idea is to prove Theorem \ref{rel qc still sep} by applying the following theorem of Haglund.

\begin{thm}[\cite{Hag} Theorem A]\label{t: Haglund's Theorem}
Let $W$ be a right-angled Coxeter group with the standard generating set and let $H$ be a word-quasiconvex subgroup.  Then $H$ is a virtual retract of $W$.
\end{thm}

Theorem A of \cite{Hag} is not stated in this form.  Nevertheless, as observed in the paragraph following Theorem A, this is what is proved.

\begin{cor}[Cf.\ \cite{HW} Corollary 7.9]\label{c: Word-qc implies separable}
{If $G$ is the fundamental group of a compact, virtually special cube complex and $H$ is a combinatorially quasiconvex subgroup of $G$ then $H$ is a virtual retract of $G$.}
\end{cor}
\begin{proof}
Let $G'$ be a special subgroup of finite index in $G$.  It is clear that $H'=H\cap G'$ is combinatorially quasiconvex in $G'$.  By the above remark, $H'$ is word-quasiconvex in the right-angled Coxeter group $W$, so $H'$ is a virtual retract of $W$ and hence of $G'$ by Theorem \ref{t: Haglund's Theorem}.  By \cite[Theorem 4.4]{HW}, $G$ is linear.  We can now apply the argument of \cite[Theorem 2.10]{long_subgroup_2008} to deduce that $H$ is a virtual retract of $G$.
\end{proof}

The reader is referred to \cite{manning_separation_2008} and \cite{hruska_relative_2008} for definitions of \textit{relatively hyperbolic} groups and \textit{relatively quasiconvex} subgroups, which are the subject of Theorem \ref{rel qc still sep}.  (See Proposition \ref{p: Hruska's characterization} below for a characterization of relative quasiconvexity.)  In order to deduce Theorem \ref{rel qc still sep} from Corollary \ref{c: Word-qc implies separable}, it would be enough to show that every relatively quasiconvex subgroup of the relatively hyperbolic fundamental group of a C-special cube complex is combinatorially quasiconvex.  Unfortunately, this may be false.  For instance, the diagonal subgroup of $\mathbb{Z}^2$ with the standard generating set is not quasiconvex.  The next theorem, a minor modification of a result of \cite{manning_separation_2008}, gets round this difficulty.

\begin{dfn}
Suppose a group $G$ is hyperbolic relative to a finite set of subgroups $\p$.  Then a relatively quasiconvex subgroup is called \emph{fully relatively quasiconvex} if for every $P\in\p$ and every $g\in G$, either $H\cap gPg^{-1}$ is trivial or $H\cap gPg^{-1}$ has finite index in $gPg^{-1}$.
\end{dfn}

\begin{thm}[Cf.\ \cite{manning_separation_2008} Theorem 1.7]\label{t: Fully rel. qc subgroups}
Suppose that $G$ is hyperbolic relative to $\p$ and that every $P\in\p$ is finitely generated and abelian.  If $Q$ is a relatively quasiconvex subgroup of $G$ then $G$ has a fully relatively quasiconvex subgroup $H$ such that $Q$ is a retract of $H$.
\end{thm}
\begin{proof}
In the proof of \cite[Theorem 1.7]{manning_separation_2008}, the authors construct a sequence of relatively quasiconvex subgroups
\[
Q=Q_0\subseteq Q_1\subseteq\ldots\subseteq Q_n=H
\]
with $H$ fully relatively quasiconvex.  We recall a few details of the construction of $Q_k$ from $Q_{k-1}$.  We will modify this construction slightly so that $Q_{k-1}$ is a retract of $Q_k$ for each $k$.  For some maximal infinite parabolic subgroup $K_k$ of $Q_{k-1}$, there is $P_k\in\p$ and $f_k\in G$ such that $K_k\subseteq f_kP_kf_k^{-1}$.  Manning and Martinez-Pedroza find a finite-index subgroup $R_k$ of $f_kP_kf_k^{-1}$ that contains $K_k$ and excludes a certain finite set $F$.  We shall impose an extra condition on $R_k$ that is easily met when $P_k$ is abelian, namely that $K_k$ should be a direct factor of $R_k$.  Just as in \cite{manning_separation_2008}, the next subgroup in the sequence is now defined as $Q_k=\langle Q_{k-1}, R_k\rangle$, and just as in that setting it follows that $Q_k$ is relatively quasiconvex.

It remains only to show that $Q_{k-1}$ is a retract of $Q_k$.  By assertion (1)  of \cite[Theorem 3.6]{manning_separation_2008}, the natural map
\[
Q_{k-1}*_{K_k} R_k\to Q_k
\]
is an isomorphism.  But $K_k$ is a direct factor of $R_k$ and so there is a retraction $R_k\to K_k$, which extends to a retraction $Q_k\to Q_{k-1}$ as required.
\end{proof}

In light of Theorem \ref{t: Fully rel. qc subgroups}, to prove Theorem \ref{rel qc still sep} it will suffice to show that when $G$ is the relatively hyperbolic fundamental group of a non-positively curved cube complex, its fully relatively quasiconvex subgroups are combinatorially convex.  This is the content of Proposition \ref{p: Fully qc implies combinatorially qc} below.

Hruska has extensively investigated various equivalent definitions of relative hyperbolicity and relative quasiconvexity \cite{hruska_relative_2008}.  Corollary 8.16 of \cite{hruska_relative_2008} provides a characterization of relative quasiconvexity in terms of geodesics in the Cayley graph.  Unfortunately, to prove Theorem \ref{rel qc still sep} we need to work in the one-skeleton of the universal cover of a cube complex.  This is not actually a Cayley graph unless the cube complex in question has a unique vertex.  It is, however, quasi-isometric to the Cayley graph.  Therefore, we will need a quasigeodesic version of Hruska's Corollary 8.16.  Fortunately, we shall see that Hruska's proof goes through.

In what follows, $S$ is any choice of finite generating set for $G$ and $d$ is the usual length metric on $Cay_S(G)$.  For any $g\in G$ write $l(g)$ for $d(1,g)$, the word length of $g$ with respect to $S$.  For $x\in Cay_S(G)$ we denote by $B(x,R)$ the open ball of radius $R$ about $x$.  We define
\[
N_R(Y)=\bigcup_{y\in Y} B(y,R)
\]
for any subspace $Y\subseteq Cay_S(G)$ and any $R>0$.  To keep notation to a minimum we will work with $\tau$-quasigeodesics, which are more usually defined as $(\tau,\tau)$-quasigeodesics.  That is, a path $c$ is a $\tau$-quasigeodesic if
\[
\frac{1}{\tau}|s-t|-\tau\leq d(c(s),c(t))\leq \tau|s-t|+\tau
\]
for all suitable $s$ and $t$.  We will always assume that our quasigeodesics are continuous, which we can do by \cite[Lemma III.H.1.11]{BrH}.  The following definition is adapted from \cite{hruska_relative_2008}.

\begin{dfn}[Cf.\ \cite{hruska_relative_2008} Definition 8.9]
Let $H$ be a subgroup of $G$.  Let $c$ be (the image of) a quasigeodesic in $Cay_S(G)$.  If $x\in c$ is not within distance $R$ of the endpoints of $c$ and
\[
B(x,R)\cap c\subseteq N_\epsilon(gP)
\]
for some $g\in G$ and $P\in\p$ then $x$ is called \emph{$(\epsilon,R)$-deep} in $gP$.  If $x\in c$ is not $(\epsilon,R)$-deep in any such coset $gP$ then $x$ is called an \emph{$(\epsilon,R)$-transition point} of $c$.
\end{dfn}

The next proposition characterizes relatively quasiconvex subgroups in terms of quasigeodesics in the Cayley graph.  Roughly, it asserts that every point on a quasigeodesic between elements of $H$ is either close to $H$ or is close to some peripheral coset $gP$.

\begin{prop}[Cf.\ \cite{hruska_relative_2008} Corollary 8.16]\label{p: Hruska's characterization}
Suppose $G$ is hyperbolic relative to $\p$ and $H$ is a subgroup of $G$.  Then $H$ is relatively quasiconvex in $G$ if and only if for every $\tau$ there are constants $\epsilon, R, \kappa$ such that the following two properties hold.
\begin{enumerate}
\item For any continuous $\tau$-quasigeodesic $c$ in $Cay_S(G)$, any connected component $\bar{c}$ of the set of all $(\epsilon,R)$-deep points of $c$ is $(\epsilon,R)$-deep in a unique peripheral left coset $gP$; that is, there exists a unique $P\in\p$ and $gP\in G/P$ such that every $x\in\bar{c}$ is $(\epsilon,R)$-deep in $gP$ and  no $x\in\bar{c}$ is $(\epsilon,R)$-deep in any other peripheral left coset.
\item If the quasigeodesic $c$ joins two points of $H$ then the set of $(\epsilon,R)$-transition points of $c$ is contained in $N_\kappa(H)$.
\end{enumerate}
\end{prop}

The statement of \cite[Corollary 8.16]{hruska_relative_2008} only deals with the case when $c$ is a geodesic.  However, the necessary results of Section 8 of \cite{hruska_relative_2008} also hold in the quasigeodesic case.

The following proposition completes the proof of Theorem \ref{rel qc still sep}.

\begin{prop}\label{p: Fully qc implies combinatorially qc}
Let $G$ be finitely generated and relatively hyperbolic.  Suppose that $G$ acts properly discontinuously and cocompactly by isometries on a geodesic metric space $X$.   Fix a basepoint $v\in X$.  For any fully relatively quasiconvex subgroup $H\subseteq G$ there exists a constant $\nu$ such that any geodesic between two points of the orbit $Hv$ lies in the $\nu$-neighbourhood of $Hv$.  In particular, if $G$ is the fundamental group of a non-positively curved cube complex then, taking $X$ to be the one-skeleton of the universal cover, it follows that $H$ is combinatorially quasiconvex.
\end{prop}

Proposition \ref{p: Hruska's characterization} implies that, to prove Proposition \ref{p: Fully qc implies combinatorially qc}, it is enough to prove that deep points of quasigeodesics between points of $H$ lie in a bounded neighbourhood of $H$.  The key technical tool is the following lemma, which is nothing more than the Pigeonhole Principle.

\begin{lemma}\label{l: Pigeonhole Lemma}
Let $G$ be a finitely generated group.  Fix a choice of finite generating set and the corresponding word metric on $G$.  If $H,K$ are subgroups and $H\cap K=1$ then
\[
\#(H\cap N_r(K))<\infty
\]
for any $r>0$.
\end{lemma}
\begin{proof}
For a contradiction, suppose $h_i\in H\cap N_r(K)$ are distinct for all $i\in\mathbb{N}$.  For each $i$, there is $k_i\in K$ with $d(h_i,k_i)< r$. Let $g_i=h_i^{-1}k_i$, so $l(g_i)<r$.  The ball of radius $r$ in $G$ is finite, so $g_i=g_j$ for some $i\neq j$ by the Pigeonhole Principle.  But now
\[
h_ih_j^{-1}=h_ig_ig_j^{-1}h_j^{-1}=k_ik_j^{-1}
\]
is a non-trivial element of $H\cap K$, a contradiction.
\end{proof}

It follows that only short elements of $H$ can be close to parabolic left cosets for which $H$ intersects the stabilizer trivially.

\begin{lemma}\label{l: Only short elements are close to empty parabolics}
Suppose $G$ is hyperbolic relative to $\p$ and $H$ is any subgroup of $G$.   Let $g\in G$ and $P\in\p$ be such that $H\cap gPg^{-1}=1$.  For any $r>0$ there exists finite $\lambda=\lambda(r,gP)$ such that if $h\in N_r(gP)\cap H$ then $l(h)\leq\lambda$.
\end{lemma}
\begin{proof}
Choose $g$ of minimal word length in $gP$ and set $k=l(g)$.  For any $p\in P$, $d(gp,gpg^{-1})=k$ and it follows that
\[
N_r(gP)\subseteq N_{k+r}(gPg^{-1})
\]
by the triangle inequality.  Therefore, by Lemma \ref{l: Pigeonhole Lemma} with $K=gPg^{-1}$, $N_r(gP)\cap H$ is finite and so
\[
\lambda=\max\{l(h)\mid h\in N_r(gP)\cap H\}
\]
is as required.
\end{proof}

We are now ready to prove Proposition \ref{p: Fully qc implies combinatorially qc}.
\proof[Proof of Proposition \ref{p: Fully qc implies combinatorially qc}]
Consider a geodesic $b$ in $X$ joining two points of $Hv$.  We need to show that $b$ is contained in a uniformly bounded neighbourhood of $Hv$.

By the \v{S}varc--Milnor Lemma, $G$ has a finite generating set $S$ and $X$ is quasi-isometric to the Cayley graph $Cay_S(G)$.  The geodesic $b$ maps to some $\tau$-quasigeodesic in $Cay_S(G)$, which we denote $c$.  Furthermore, we can assume that $c$ is continuous by \cite[Lemma III.H.1.11]{BrH}.  It is therefore enough to show that $c$ is contained in a uniformly bounded neighbourhood of $H$ in the word metric $d$ on $Cay_S(G)$.

Let $\epsilon$, $R$ and $\kappa$ be as in Proposition \ref{p: Hruska's characterization}.  By assertion 2 of Proposition \ref{p: Hruska's characterization}, the $(\epsilon,R)$-transition points of $c$ are contained in the $\kappa$-neighbourhood of $H$.  Therefore, it remains to show that the $(\epsilon,R)$-deep points of $c$ are contained in a uniformly bounded neighbourhood of $H$.

Let $\bar{c}$ be a connected component of the set of all $(\epsilon,R)$-deep points of $c$.  By definition, every $x\in\bar{c}$ is in the $\epsilon$-neighbourhood of some peripheral left coset $gP$. By assertion 1 of Proposition \ref{p: Hruska's characterization}, the component $\bar{c}$ is contained between two $(\epsilon,R)$-transition points of $c$, which we shall denote $y_1$ and $y_2$.  We can take these points to be arbitrarily close to $\bar{c}$, and hence we can assume that $d(y_i,gP)\leq\epsilon$ for $i=1,2$.  On the other hand, by assertion 2 of Proposition \ref{p: Hruska's characterization}, there exist $h_1,h_2\in H$ such that $d(h_i,y_i)<\kappa$ for $i=1,2$.  Therefore, $h_i\in N_{\epsilon+\kappa}(gP)$ for $i=1,2$.

Let $h_0=h_1^{-1}h_2$ and let $g_0=h_1^{-1}g$, so $h_0\in N_{\epsilon+\kappa}(g_0P)$ and, without loss of generality, $l(g_0)\leq \epsilon+\kappa$.  There are two cases to consider, depending on whether $h_0$ is long or short.  Let
\[
\lambda_{\max}=\max\{\lambda(\epsilon+\kappa,gP)\mid P\in\p,l(g)\leq\epsilon+\kappa, H\cap gPg^{-1}=1\}
\]
where $\lambda(\epsilon+\kappa,gP)$ is provided by Lemma \ref{l: Only short elements are close to empty parabolics}.  In the first case, $l(h_0)\leq\lambda_{\max}$ so $d(h_1,h_2)\leq\lambda_{\max}$ and therefore $d(y_1,y_2)<\lambda_{\max}+2\kappa$.  Because $c$ is a $\tau$-quasigeodesic it follows that for every $x\in\bar{c}$, for some $i=1,2$, we have that
\[
d(x,y_i)<\lambda'=\frac{\tau^2}{2}(\lambda_{\max}+2\kappa+\tau)+\tau
\]
and so $d(x,h_i)<\lambda'+\kappa$.

In the second case, $l(h_0)>\lambda_{\max}$ and so $H\cap g_0Pg_0^{-1}\neq 1$ by Lemma \ref{l: Only short elements are close to empty parabolics}.  Therefore $H\cap g_0Pg_0^{-1}$ has finite index in $g_0Pg_0^{-1}$ because $H$ is fully relatively quasiconvex.  For each $g\in G$ and $P\in\p$ for which $H\cap gPg^{-1}$ has finite index in $P$, let $\mu=\mu(gP)$ be a number such that $gPg^{-1}\subseteq N_\mu(H\cap gPg^{-1})$.  Set
\[
\mu_{\max}=\max\{\mu(gP)\mid P\in\p, l(g)\leq\epsilon+\kappa, |gPg^{-1}:H\cap gPg^{-1}|<\infty\}.
\]
Therefore
\[
g_0Pg_0^{-1}\subseteq N_{\mu_{\max}}(H)
\]
and so
\[
g_0P\subseteq N_{\mu_{\max}+\epsilon+\kappa}(H)
\]
because $l(g_0)\leq\epsilon+\kappa$.  For each $x\in\bar{c}$ we have $h_1^{-1}x\in N_\epsilon(g_0P)$ and so $h_1^{-1}x\in N_{\mu_{\max}+2\epsilon+\kappa}(H)$. Therefore $x\in N_{\mu_{\max}+2\epsilon+\kappa}(H)$.

In summary, we have shown the following: the $(\epsilon,R)$-transition points of the geodesic $c$ are contained in the $\kappa$-neighbourhood of $H$; the short $(\epsilon,R)$-deep components of $c$ are contained in the $(\lambda'+\kappa)$-neighbourhood of $H$; and the long $(\epsilon,R)$-deep components of $c$ are contained in the $(\mu_{\max}+2\epsilon+\kappa)$-neighbourhood of $H$.  Therefore, $c$ is completely contained in the $\nu$-neighbourhood of $H$, where
\[
\nu=\max\{\kappa,\lambda'+\kappa,\mu_{\max}+2\epsilon+\kappa\}
\]
This completes the proof.
\endproof

We have assembled all the tools necessary to prove Theorem \ref{rel qc still sep}.

\proof[Proof of Theorem \ref{rel qc still sep}]
Let $Q$ be a relatively quasiconvex subgroup of $G=\pi_1\s$.  By Theorem \ref{t: Fully rel. qc subgroups}, there exists a fully relatively quasiconvex subgroup $H$ of $G$ such that $Q$ is a retract of $H$.  Let $X$ be the one-skeleton of the universal cover of $\s$, equipped with the induced length metric.  By Proposition \ref{p: Fully qc implies combinatorially qc}, for any basepoint $v$ the orbit $Hv$ is quasiconvex in $X$; that is, $H$ is a combinatorially quasiconvex subgroup of $G$.  Therefore, by Corollary \ref{c: Word-qc implies separable}, $H$ is a virtual retract of $G$ and so $Q$ is also a virtual retract of $G$, as required.
\endproof

Corollary \ref{c: Corollary 2} now follows easily.

\proof[Proof of Corollary \ref{c: Corollary 2}]
Let $\Gamma=\pi_1M$.  As pointed out in \cite{canary_mardens_2008}, to prove that $\Gamma$ is LERF it is enough to prove that $\Gamma$ is GFERF --- that is, that the geometrically finite subgroups are separable.  Furthermore, by \cite[Proposition 3.28]{Hag}, it is enough to prove that the geometrically finite subgroups of $G$ are virtual retracts.

First, suppose that $M$ is orientable.  Let $Q$ be a geometrically finite subgroup of $\Gamma$.  By \cite[Theorem 1.3]{manning_separation_2008}, for instance, $\Gamma$ is hyperbolic relative to its maximal parabolic subgroups and $Q$ is a relatively quasiconvex subgroup of $\Gamma$.  The maximal parabolic subgroups of $\Gamma$ are isomorphic to $\mathbb{Z}^2$.  By Theorem \ref{rt ang cox}, $\Gamma$ is the fundamental group of a virtually special cube complex $\s$, so $Q$ is a virtual retract of $\Gamma$ by Theorem \ref{rel qc still sep}.

If $M$ is nonorientable then we can pass to a degree-two orientable cover $M'$ with fundamental group $\Gamma'$.  As above, we see that for every geometrically finite subgroup $Q$ of $\Gamma$, the intersection $Q'=Q\cap \Gamma'$ is a virtual retract of $\Gamma'$.  Now, by the proof of \cite[Theorem 2.10]{long_subgroup_2008}, it follows that $H$ is a virtual retract of $\Gamma$.
\endproof

{
We take this opportunity to note that the combination of Proposition \ref{p: Fully qc implies combinatorially qc} and {Corollary \ref{c: Word-qc implies separable} }shows that many subgroups of virtually special relatively hyperbolic groups are virtual retracts, even without any hypotheses on the parabolic subgroups.  Indeed, we have the following.

\begin{thm}\label{t: Retracts with no hypotheses}
Let $\s$ be a compact, virtually special cube complex and suppose that $\pi_1\s$ is relatively hyperbolic.  Then every fully relatively quasiconvex subgroup of $\pi_1\s$ is a virtual retract.  
\end{thm}

Recall that an element $\gamma$ of a relatively hyperbolic group is called \emph{hyperbolic} if it is not conjugate into a parabolic subgroup.  Denis Osin has shown that cyclic subgroups generated by hyperbolic elements are strongly relatively quasiconvex \cite[Theorem 4.19]{osin_relatively_2006}.  In the torsion-free case this implies \emph{a fortiori} that such subgroups are fully relatively quasiconvex.

\begin{cor}
Let $\s$ be a compact, virtually special cube complex and suppose that $\pi_1\s$ is relatively hyperbolic. For any hyperbolic element $\gamma\in\pi_1\s$, the cyclic subgroup $\langle \gamma\rangle$ is a virtual retract of $\pi_1\s$.
\end{cor}

Combining Theorem \ref{t: Retracts with no hypotheses} with \cite[Theorem 1.7]{manning_separation_2008}, we obtain a slightly weaker version of Theorem \ref{rel qc still sep} that holds when the peripheral subgroups are only assumed to be LERF and slender.  (A group is \emph{slender} if each subgroup is finitely generated.)

\begin{cor}\label{c: Virtual retracts with more general parabolics}
Let $\s$ be a compact, virtually special cube complex and suppose that $\pi_1\s$ is hyperbolic relative to a collection of slender, LERF subgroups.  Then every relatively quasiconvex subgroup of $\pi_1\s$ is separable and every fully relatively quasiconvex subgroup of $\pi_1\s$ is a virtual retract.  \end{cor}

This result would apply if $\pi_1\s$ were the fundamental group of a finite-volume negatively curved manifold of dimension greater than three, in which case the parabolic subgroups would be non-abelian but nilpotent.  Note that the full conclusion of Theorem \ref{rel qc still sep} does not hold in this case: nilpotent groups that are not virtually abelian contain cyclic subgroups that are not virtual retracts.}

%%%%%%%%%%%%%%%%%%%%%
%%%%%%%%%%%%%%%%%
\section{Examples}\label{sec:examples}
%%%%%%%%%%%%%%%%%

In this section we describe many hyperbolic $3$-manifolds that decompose into right-angled ideal polyhedra.  Our aim is to display the large variety of situations in which Theorem \ref{rt ang cox} applies, and to explore the question of when a manifold that decomposes into right-angled ideal polyhedra is commensurable with a right-angled reflection orbifold.  When this is the case, the results of this paper follow from previous work, notably that of Agol--Long--Reid \cite{ALR}.  The theme of this section is that this occurs among examples of lowest complexity, but that one should not expect it to in general.

Lemma \ref{symm comm} describes when one should expect a manifold $M$ that decomposes into right-angled ideal polyhedra to be commensurable with a right-angled reflection orbifold.  This is the case when all of the polyhedra decomposing $M$ are isometric to a single right-angled ideal polyhedron $\calp$, which furthermore is highly symmetric.  A prominent example which satisfies this is the Whitehead link complement, which is commensurable with the reflection orbifold in the regular ideal octahedron.  

The octahedron (also known as the $3$-antiprism, see Figure \ref{antiprisms}) is the simplest right-angled ideal polyhedron, as measured by the number of ideal vertices.  Propositions \ref{P_1 comm} and \ref{P_2 comm} imply that any manifold that decomposes into isometric copies of the right-angled ideal octahedron or, respectively, the $4$-antiprism, is commensurable with the corresponding reflection orbifold.  On the other hand, in Section \ref{sec:One cusp} we will describe an infinite family of ``hybrid'' hyperbolic $3$-manifolds $N_n$, each built from both the $3$- and $4$-antiprisms, that are not commensurable with any $3$-dimensional hyperbolic reflection orbifold.  We use work of Goodman-Heard-Hodgson \cite{GHH} here to expicitly identify the commensurator quotients for the $N_n$.

%Although the manifolds $N_n$ of Section \ref{sec:One cusp} are not commensurable with $3$-dimensional hyperbolic right-angled reflection orbifolds, their hybrid nature suggests that there may exist quasi-isometric embeddings of the groups $\pi_1 N_n$ into a $4$-dimensional right-angled reflection group.  Indeed, the group $SO(4,1;\mathbb{Z})$ is commensurable to a right-angled reflection group and contains finite-index subgroups of $\Gamma_{\calp_1}$ and $\Gamma_{\calp_2}$ as stabilizers of intersecting hyperplanes (see \cite{ALR}, Theorem 2.3 and Lemma 3.2).  One may consequentially describe many representations $\Gamma \to SO(4,1;\mathbb{Z})$ for subgroups $\Gamma < \pi_1(N_n)$, and \cite[Theorem 9.1]{BHW} suggests that some of these may be quasiconvex embeddings.

%%%%%%%%%%%%%%%%%%%%%%%%%%%%%%%
\subsection{The simplest examples.}\label{sec: rt ang comm}
%%%%%%%%%%%%%%%%%%%%%%%%%%%%%%%

It may initially seem that a manifold that decomposes into right-angled polyhedra should be commensurable with the right-angled reflection orbifold in one or a collection of the polyhedra.  This is not the case in general; however, the technical lemma below implies that it holds if all of the polyhedra are isometric and sufficiently symmetric.  

\begin{lemma}\label{symm comm}  Let $M$ be a complete hyperbolic $3$-manifold with a decomposition $\{\calp_i\}$ into right-angled ideal polyhedra.  For a face $f \in \calp_i$, let $\gamma_f$ be reflection in the hyperplane containing $f$.  If for each such face, $\phi_f \circ \gamma_f$ is an isometry to the polyhedron $\calp_j$ containing $\phi_f(f)$, then $\pi_1 M$ is contained in $\Gamma \rtimes \mathrm{Sym}(\calp_1)$, where $\Gamma$ is the reflection group in $\calp_1$ and $\mathrm{Sym}(\calp_1)$ is its symmetry group.  \end{lemma}

\begin{proof}  Let $M$ be a hyperbolic $3$-manifold satisfying the hypotheses of the lemma.  There is a ``dual graph'' to the polyhedral decomposition $\{\calp_i\}$ with a vertex for each $i$, such that the vertex corresponding to $\calp_i$ is connected by an edge to that corresponding to $\calp_j$ for every face $f$ of $\calp_i$ such that $\phi_f(f)$ is a face of $\calp_j$.  Let $\calt$ be the tiling of $\mathbb{H}^3$ by $\Gamma$-translates of $\calp_1$.  A maximal tree $T$ in the dual graph determines isometries taking the $\calp_i$ into $\calt$ as follows.

Suppose $f$ is a face of $\calp_1$ that corresponds to an edge of $T$.  Then by hypothesis $\phi_f^{-1}(\calp_i) = \gamma_f(\calp_1)$, where $\calp_i$ contains $\phi_f(f)$.  For arbitrary $i$, let $\alpha$ be an embedded edge path in $T$ from the vertex corresponding to $\calp_1$ to that of $\calp_i$, and suppose $\calp_{i_0}$ corresponds to the vertex with distance one on $\alpha$ from that of $\calp_i$.  We inductively assume that there exists an isometry $\phi_{i_0}$ such that $\phi_{i_0}(\calp_{i_0})$ is a $\Gamma$-translate of $\calp_1$.  Let $f$ be the face of $\calp_{i_0}$ corresponding to the edge of $T$ between $\calp_{i_0}$ and $\calp_i$.  Then $\phi_{i_0} \gamma_f \phi_{i_0}^{-1} = \gamma_{\phi_{i_0}(f)}\in \Gamma$, so by hypothesis,
$$  \phi_{i_0} \circ \phi_f^{-1}(\calp_i) = \phi_{i_0} \gamma_f(\calp_{i_0}) = (\phi_{i_0}\gamma_f\phi_{i_0}^{-1})(\phi_{i_0}(\calp_{i_0}))$$ 
is a $\Gamma$-translate of $\calp_1$.

Now for each $i$, after replacing $\calp_i$ by $\phi_i(\calp_i)$ we may assume that there is some $\gamma_i \in \Gamma$ such that $\calp_i = \gamma_i(\calp_1)$.  For a face $f$ of $\calp_i$, let $\calp_j$ be the polyhedron containing $\phi_f(f)$.  Then by hypothesis $\gamma_j^{-1} \phi_f \gamma_f \gamma_i \in \mathrm{Sym}(\calp_1)$.  Therefore $\phi_f \in \Gamma \rtimes \mathrm{Sym}(\calp_1)$; thus the lemma follows from the Poincar\'e polyhedron theorem.
\end{proof}

\begin{figure}
\begin{center}
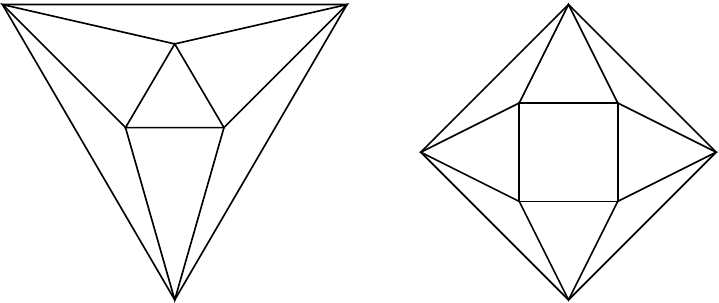
\end{center}
\caption{The $3$- and $4$-antiprisms.}
\label{antiprisms}
\end{figure}

A natural measure of the complexity of a right-angled ideal polyhedron is its number of ideal vertices.  By this measure, the two simplest right-angled ideal polyhedra are the $3$- and $4$-antiprisms, pictured in Figure \ref{antiprisms}.  (The general definition of a $k$-antiprism, $k \geq 5$ should be evident from the figure.)  

\begin{lemma}\label{simple right}  The only right-angled ideal polyhedra with fewer than ten vertices are the $3$- and $4$- antiprisms.  \end{lemma}

\begin{proof}  By a \textit{polyhedron} we mean a $3$-complex with a single $3$-cell whose underlying topological space is the $3$-dimensional ball, such that no two faces that share an edge $e$ have vertices in common other than the endpoints of $e$.  By Andreev's theorem, there is a right-angled ideal polyhedron in $\mathbb{H}^3$ with the combinatorial type of a given polyhedron if and only if each vertex has valence $4$, there are no prismatic $3$- or $4$-circuits, and the following criterion holds: given faces $f_0$, $f_1$, and $f_2$ such that $f_0$ and $f_2$ each share an edge with $f_1$, $f_0$ and $f_2$ have no vertices in common with each other but not $f_1$.  (A prismatic $k$-circuit is a sequence of $k$ faces $f_0, f_1, \hdots f_{k-1}$ such that no three faces have a common vertex but for each $i$, $f_i$ shares an edge with $f_{i-1}$ and $f_{i+1}$, taking indices modulo $k$.)

If $f$ is a $k$-gon face of a right-angled ideal polyhedron $\calp$, the final criterion above implies that $\calp$ has at least $2k$ ideal vertices, since each face that abuts $f$ contributes at least one unique vertex to $\calp$.  Thus any right-angled ideal polyhedron with fewer than $10$ ideal vertices has only triangular and quadrilateral faces.  Let $v$, $e$, and $f$ be the number of vertices, edges and faces of $\calp$, respectively.  Since each vertex has valence $4$, we have $4v = 2e$.  If $\calp$ has only triangular faces, then $2e = 3f$, and an Euler characteristic calculation yields
$$ v - e + f = \frac{3f}{4} - \frac{3f}{2} + f = 2.  $$
Therefore in this case $f=8$, and it is easy to see that $\calp$ must be the $3$-antiprism.

If $\calp$ has a quadrilateral face $f$ and only $8$ vertices, then by the final criterion of the first paragraph all faces adjacent to it are triangles.  The union of $f$ with the triangular faces adjacent to it is thus a subcomplex that is homeomorphic to a disk and contains all vertices of $\calp$.  It follows that $\calp$ is the $4$-prism.  Since each vertex of a right-angled ideal polyhedron is $4$-valent, the number of vertices is even, and the lemma follows.  \end{proof}

It is well known that the $3$-antiprism $\calp$, better known as the octahedron, is \textit{regular}: there is  a symmetry exchanging any two ordered triples $(v,e,f)$ where $v \subset e \subset f$ are faces of dimension $0$, $1$, and $2$, respectively.  Now suppose $M$ is a manifold with a decomposition into polyhedra $\{\calp_i\}$ such that for each $i$, there is an isometry $\gamma_i \co \calp \to \calp_i$.  If $\calp_i$ and $\calp_j$ are polyhedra in this decomposition, containing faces $f$ and $f'$, respectively, such that $\phi_f (f) = f'$, then $\gamma_j^{-1} \phi_f \gamma_i$ takes one face of $\calp$ isometrically to another; hence it is realized by a symmetry $\sigma$ of $\calp$.  It follows that $\gamma_{f'} \circ \gamma_j \sigma \gamma_i^{-1} = \phi_f$.  Thus Lemma \ref{symm comm} implies:

\begin{prop}\label{P_1 comm}  Let $\Gamma_1$ be the group generated by reflections in the sides of the octahedron $\calp$, and let $\Sigma_1$ be its symmetry group.  If $M$ is a complete hyperbolic manifold that decomposes into copies of $\calp$, then $\pi_1 M < \Gamma_1 \rtimes \Sigma_1$.  In particular, $\pi_1 M$ is commensurable to $\Gamma_1$.  \end{prop}

The $4$-antiprism does not have quite enough symmetry to directly apply Lemma \ref{symm comm}, but its double across a square face is the cuboctahedron, the semi-regular polyhedron pictured on the right-hand side of Figure \ref{P_1 and P_2}.  The cuboctahedron has a symmetry exchanging any two square or triangular faces, and each symmetry of each face extends over the cuboctahedron.

\begin{figure}
\begin{center}
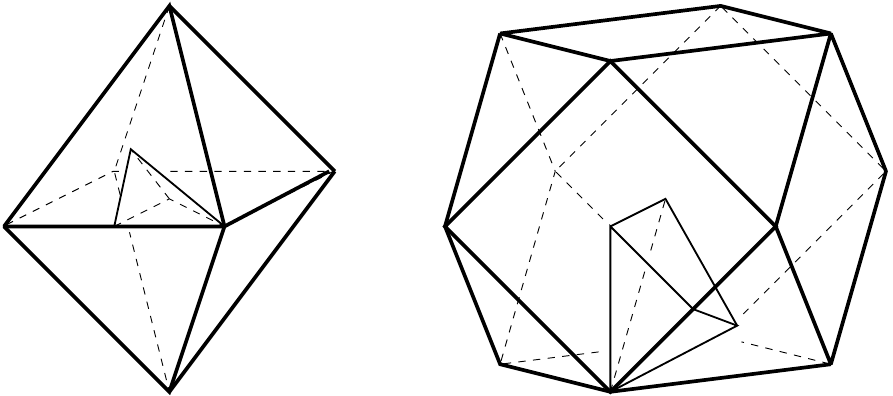
\end{center}
\caption{The ideal octahedron $\calp_1$ and cuboctahedron $\calp_2$.}
\label{P_1 and P_2}
\end{figure}

\begin{prop}\label{P_2 comm}  Let $\Gamma_2$ be the group generated by reflections in the sides of the cuboctahedron, and let $\Sigma_2$ be its group of symetries.  If $M$ is a complete hyperbolic $3$-manifold that decomposes into copies of the cuboctahedron, then $\pi_1(M)<\Gamma_2 \rtimes \Sigma_2$.  If $M$ decomposes into $4$-antiprisms, then $\pi_1(M)$ has an index-$2$ subgroup contained in $\Gamma_2 \rtimes \Sigma_2$.   \end{prop}

\begin{proof}  Since face pairing isometries must in particular preserve combinatorial type, it follows from Lemma \ref{symm comm} as argued above Proposition \ref{P_1 comm} that if $M$ decomposes into copies of the cuboctahedron, then $\pi_1(M) < \Gamma_2 \rtimes \Sigma_2$.

Opposite square faces of the $4$-antiprism inherit opposite colors from any checkering.  Thus if a hyperbolic $3$-manifold $M$ has a checkered decomposition into right-angled ideal $4$-antiprisms, they may be identified in pairs along, say, dark square faces, yielding a decomposition into right-angled ideal cuboctahedra.  The proof of Theorem \ref{rt ang cox} shows that if the decomposition of $M$ is not checkered, there is a twofold cover $\widetilde{M} \to M$ that inherits a checkered decomposition.  Hence if $M$ decomposes into $4$-antiprisms, $\widetilde{M}$ decomposes into copies of the cuboctahedron.  The final claim of the proposition follows.
\end{proof}

The results of \cite{Hatcher} imply that for $j=1,2$, $\Gamma_j \rtimes \Sigma_j$ is isomorphic to the arithmetic group $\mathrm{PGL}_2(\calo_j)$, where $\calo_j$ is the ring of integers of $\mathbb{Q}(\sqrt{-j})$. 

The fundamental domain for $\mathrm{Sym}(\calp_1)$ pictured in Figure \ref{P_1 and P_2} intersects $\partial \calp_1$ in a $(2,3,\infty)$ triangle.  We refer by $\Lambda$ to the group generated by reflections in the sides of this triangle.  The fundamental domain for $\mathrm{Sym}(\calp_2)$ intersects a triangular face in a $(2,3,\infty)$ triangle as well; thus $\Lambda$ embeds in $\Gamma_j \rtimes \mathrm{Sym}(\calp_j)$ for $j=1$ and $2$.  The lemma below records an observation we will find useful in the following sections.

\begin{lemma}\label{transitive}  For $j=1,2$, let $\calt_j$ be the tiling of $\mathbb{H}^3$ by $\Gamma_j$-conjugates of $\calp_j$.  The action of $\Gamma_j \rtimes \mathrm{Sym}(\calp_j)$ is transitive on the set of all geodesic planes that contain a triangular face of a tile of $\calt_j$.  \end{lemma}

This lemma follows from the fact, evident by inspection of the fundamental domains in Figure \ref{P_1 and P_2}, that $\mathrm{Sym}(\calp_j)$ acts transitively on triangular faces of $\calp_j$.

%%%%%%%%%%%%%%
\subsection{A family of one-cusped manifolds}\label{sec:One cusp}
%%%%%%%%%%%%%%

In this section, we exhibit an infinite family $\{N_n\}$ of pairwise incommensurable manifolds that are not commensurable to \emph{any} 3-dimensional reflection group.  Each of these manifolds has a single cusp, and they are constructed using an explicit right-angled ideal polyhedral decomposition.

%%%%%%%%%%%%%%%%%%%%%%
%\subsection{The examples} \label{setup}
%%%%%%%%%%%%%%%%%%%%%%

\begin{dfn}  For $n\geq 2$, let $\{ \p_i \}_{i=1}^{n+2}$ be a collection of right-angled ideal polyhedra embedded in $\mathbb{H}^3$ with the following properties.
\begin{enumerate}
\item $\p_i$ is an octahedron if  $i \in \{1,n+2\}$, and a cuboctahedron otherwise.
\item There is an ideal vertex $\hat{v}$ shared by all the polyhedra.
\item $\p_i \cap \p_j$ if and only if $i = j\pm1$.
\item If $\p_i$ and $\p_j$ meet, then they share a triangular face.
\end{enumerate}  
Define $\mathcal{D}_n = \bigcup_{i=1}^{n+2} \, \p_i$.  \end{dfn}

An isometric copy in $\mathbb{H}^3$ of such a collection is determined by an embedding of $\calp_1$, a choice of $\hat{v}$, and a choice of triangular face $\calp_1 \cap \calp_2$.  If we use the upper half space model for $\mathbb{H}^3$ then $\mathrm{Isom}^+(\mathbb{H}^3)$ is identified with $\mathrm{PSL}_2(\mathbb{C})$, by isometrically extending the action by M\"obius transformations on $\partial \mathbb{H}^3 = \mathbb{C}\cup\{\infty\}$.  Using this model, we apply an isometry so that $\hat{v} = \infty$ we can project the faces of the $\p_i$'s to $\bound \mathbb{H}^3$ to get a cell decomposition of $\C$.  This decomposition is pictured for $n=2$ in Figure \ref{one}.  

Each 2-cell in the figure corresponds to a face of some $\p_i$ which is not shared by any other $\p_j$.  Shade half of the faces of $\p_1$ and $\p_{n+2}$ gray and label them $A, B, C, D, E, F, G$, and $H$ as indicated in the figure.  Label the square face of $\p_2$ which shares an edge with $B$ (respectively $A$, $D$) as $X_1$ (respectively $Y_1$, $Z_1$).  Label the square face opposite $X_1$ as $X_1'$ and so on.  Now use the parabolic translation ${\sf c}$ that takes $\p_2$ to $\p_3$ to translate the labeling to the other cuboctahedra, adding one to the subscript every time we apply ${\sf c}$.

\begin{figure}[h]

\setlength{\unitlength}{.1in}

\begin{picture}(40,17)
\put(-3,0) { \includegraphics[width=4.5in]{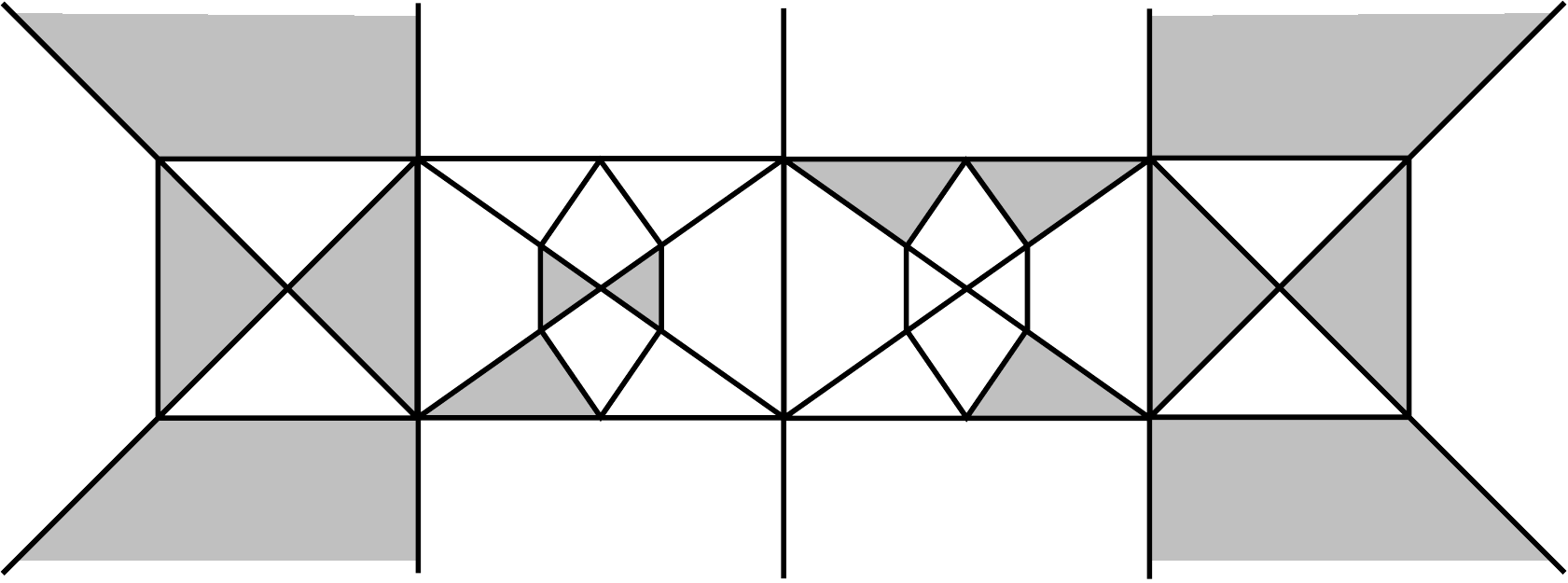} }
\put(7.2,8){$A$}
\put(5,14){$B$}
\put(2.75,8){$C$}
\put(5,2){$D$}

\put(13.8,14){$X_1$}
\put(13.8,2){$Z_1$}
\put(10.5,8){$Y_1$}
\put(17.5,8){$Y_1'$}
\put(13.8,9.6){$Z_1'$}
\put(13.8,6.5){$X_1'$}

\put(24.3,14){$X_2$}
\put(24.3,2){$Z_2$}
\put(21,8){$Y_2$}
\put(28,8){$Y_2'$}
\put(24.3,9.6){$Z_2'$}
\put(24.3,6.5){$X_2'$}

\put(36,8){$F$}
\put(33.8,14){$G$}
\put(31.5,8){$H$}
\put(33.8,2){$E$}

\end{picture}
   \caption{$\mathcal{D}_2$.}
   \label{one}
\end{figure}
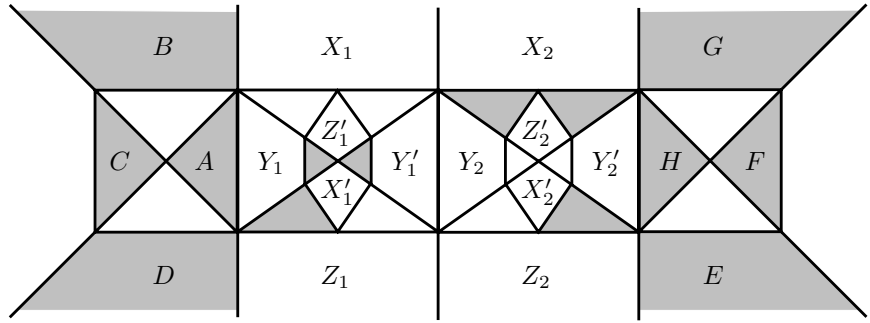

Define the isometries ${\sf a, b, f, g, x, y, z} \in \text{Isom}^+(\mathbb{H}^3)$ as follows.  The isometry taking $A$ to $B$ so that their shared vertex is taken to the vertex shared by $B$ and $C$ is ${\sf a}$.  The isometry taking $C$ to $D$ so that their shared vertex is taken to the vertex shared by $B$ and $D$ is ${\sf b}$.  The isometry taking $E$ to $F$ so that their shared vertex is taken to the vertex shared by $F$ and $G$ is ${\sf f}$.  The isometry taking $G$ to $H$ so that their shared vertex is taken to the vertex shared by $H$ and $F$ is ${\sf g}$.  The isometry taking $Y_1'$ to $X_1$ so that their shared vertex is taken to the vertex shared by $X_1$ and $Z_1'$ is ${\sf x}$.  The isometry taking $Z_1'$ to $Z_1$ so that the vertex shared by $Z_1'$ and $Y_1'$ is taken to the vertex shared by $X_1'$ and $Z_1$ is ${\sf y}$.  The isometry taking $X_1'$ to $Y_1$ so that their shared vertex is taken to the vertex shared by $Y_1$ and $Z_1$ is ${\sf z}$.

The set $S_n$ defined below is a collection of face pairings for $\{ \p_i \}_1^{n+2}$.  Here we take $\sf x^c = c x c^{-1}$.
\[ S_n \ = \ \left\{ {\sf a, b, f, g, x, y, z, x^c, y^c, z^c,} \ldots, {\sf x}^{{\sf c}^{n-1}}, {\sf y}^{{\sf c}^{n-1}}, {\sf z}^{{\sf c}^{n-1}}\right\}\]
By examining the combinatorics of these face pairings, one deduces that the quotient by these side pairings is a complete hyperbolic manifold $N_n$ with finite volume and a single cusp.  (See, for instance,  \cite[Theorem 11.1.6]{Ratcliffe}.)  By Poincar\'e's polyhedron theorem \cite[Theorem 11.2.2]{Ratcliffe}, $\Delta_n = \langle S_n \rangle$ is discrete and $\mathcal{D}_n$ is a fundamental domain for $\Delta_n$.  Furthermore, in the manner of \cite{CD},  one can write down explicit matrices in $\text{PSL}_2(\C)$ which represent these isometries and see that the trace field for $\Delta_n$ is $\Q(i, \sqrt{2})$.  Hence, $N_n \cong \mathbb{H}^3/\Delta_n$ is non-arithmetic.

\begin{dfn}  The \textit{commensurator} of $\Gamma < \mathrm{Isom}(\mathbb{H}^3)$ is defined as
$$  \mathrm{Comm}(\Gamma) \doteq \{ \sfg \in \mathrm{Isom}(\mathbb{H}^3)\,|\,[\Gamma:\sfg\Gamma\sfg^{-1} \cap \Gamma]<\infty\}.  $$  \end{dfn}

It is easy to see that every group commensurable with $\Gamma$ is contained in $\mathrm{Comm}(\Gamma)$.  A well known theorem of Margulis asserts that if $\Gamma$ is discrete and acts with finite covolume, then $\mathrm{Comm}(\Gamma)$ is itself discrete if and only if $\Gamma$ is not arithmetic (see \cite[(1) Theorem]{Margulis}).  

Let $G_n = \mathrm{Comm}(\Delta_n)$ and $O_n = \mathbb{H}^3/G_n$.  Since $\Delta_n$ is a non-arithmetic Kleinian group, $G_n$ is discrete and $O_n$ is an orbifold.  We will use the techniques of Goodman--Hodgson--Heard \cite{GHH} to prove the following proposition.

\begin{prop}\label{OP comm}
Every element of $G_n$ is orientation preserving.  Hence, $\Delta_n$ is not commensurable to any 3-dimensional reflection group.
\end{prop}

Theorem \ref{t: One-cusp examples} will follow immediately from the proposition above upon observing that the $N_n$ are pairwise incommensurable.  This follows most easily from a \textit{Bloch invariant} computation.  The Bloch invariant of a hyperbolic $3$-manifold $M$ is a sum of parameters, each an element of $\mathbb{C}^*$, of a tetrahedral decomposition of $M$, considered as an element of $\calp(\mathbb{C})$.  For a field $k$, the \textit{Pre-Bloch group} $\calp(k)$ is the quotient of the free $\mathbb{Z}$-module on $k-\{0,1\}$ by a ``five-term relation'' that can be geometrically interpreted as relating different decompositions of the union of two tetrahedra.  The Bloch group $\calb(k)$ is a subgroup of $\calp(k)$; see eg.~\cite{Neumanns_bloch}.

We will use the decomposition of $N_n$ into a collection of $2$ right-angled ideal octahedra and $n$ cuboctahedra.  These may each be divided into tetrahedra yielding a decomposition of $N_n$.  The parameters of the tetrahedra contained in the octahedron sum to an element $\beta_1\in \calb(\mathbb{Q}(i))$, and those of the cuboctahedron sum to an element $\beta_2\in \calb(\mathbb{Q}(i\sqrt{2}))$.  It can be showed that $\beta_1$ and $\beta_2$ are linearly independent in $\calb(\mathbb{Q}(i,\sqrt{2}))$, and this in turn implies that the invariants $2\cdot \beta_1 + n\cdot\beta_2$ of the $N_n$ are pairwise linearly independent.  Hence the $N_n$ are pairwise incommensurable; see \cite[Prop.~4.5]{CD} for an analogous proof.

%%%%%%%%%%%%%%%%%%%%%%%
%\subsection{Commensurators}
%%%%%%%%%%%%%%%%%%%%%%%
In proving Proposition \ref{OP comm}, we give a partial description of the commensurator $G_n$.  We use the algorithm of \cite{GHH} to perform such computations here and in Section \ref{sec:lobel}, so we briefly introduce the set-up below.  The \textit{Lorentz inner product} on $\R^4$ is the degenerate bilinear pairing 
\[ \langle {\bf v, w} \rangle \ = \ v_1 w_1+v_2w_2+v_3w_3-v_4w_4.\]
The \textit{hyperboloid model} of $\mathbb{H}^3$ is the set $\{ \mathbf{v} \, | \,  \langle {\bf v, v} \rangle=-1, \, v_4>0 \}$ equipped with the Riemannian metric on tangent spaces determined by the Lorentz inner product.  The \textit{positive light cone} is the set $L^+ = \{  {\bf v} \, | \,  \langle {\bf v,v} \rangle=0, \, v_4\geq0 \}$.  The \textit{ideal boundary} $\partial \mathbb{H}^3$ is identified with the set $PL^+$ of equivalence classes of $\bv \in L^+$, where $\bv \sim \bw$ if $\bw = \lambda \bv$ for $\lambda \in \mathbb{R}^+$.

Given a vector ${\bf v} \in L^+$, we say the set $H_{\bf v} = \{ {\bf w} \in \mathbb{H}^3 \, | \, \langle {\bf v, w} \rangle =-1\}$ is a \textit{horosphere centered at $v = [\bv]$}.  If $\alpha \in \R^+$ the horosphere $H_{\alpha {\bf v}}$ is a horosphere centered at the same ideal point as $H_{\bf v}$ and if $\alpha \leq 1$ then $H_{\bf v}$ is contained in the horoball determined by $\alpha {\bf v}$.  This correspondence between vectors in $L^+$ and horospheres in $\mathbb{H}^3$ is a bijection.  Hence, we call the vectors in $L^+$ \emph{horospherical vectors}.  

The group $\text{Isom}(\mathbb{H}^3)$ is the subgroup $\text{O}_0(3,1) \subset \text{GL}_4(\R)$ (acting by matrix multiplication) which preserves the Lorentz inner product and the sign of the last coordinate of each vector in $\R^4$.

Suppose $M=\mathbb{H}/\Lambda$ is a complete finite volume hyperbolic orbifold with $k$ cusps.  For each cusp $c_i$ of $M$, choose a horospherical vector ${\bf v}_i$ for which $H_{{\bf v}_i}$ projects to a cross section of $c_i$ under the covering map $\mathbb{H}^3 \rightarrow M$.  Then $V = \Lambda \cdot \{ {\bf v}_i \}_1^k$ is $\Lambda$-invariant and determines a $\Lambda$-invariant set of horospheres.  The convex hull $C$ of $V$ in $\R^4$ is called the \emph{Epstein--Penner convex hull}.  Epstein and Penner show that $\bound C$ consists of a countable set of 3-dimensional faces $F_i$, where each $F_i$ is a finite sided Euclidean polyhedron in $\R^4$.  Furthermore, this decomposition of $\bound C$ projects to a $\Lambda$--invariant tiling $\mathcal{T}$ of $\mathbb{H}^3$ \cite[Prop 3.5 and Theorem 3.6]{epstein_euclidean_1988}.  If $M$ is a manifold then the quotient of this tiling by $\Lambda$ gives a cell decomposition of $M$.  We refer to the tiling as a {\em canonical tiling} for $M$ and to the cell decomposition as a {\em canonical cell decomposition} of $M$.  If we make a different choice for $\{ {\bf v}_i\}_1^k$ by multiplying each vector by a common positive scalar then the resulting Epstein--Penner convex hull differs from $C$ by multiplication by this scalar.  The combinatorics of the boundary of this scaled convex hull is identical to that of $C$ and projects exactly to the tiling $\mathcal{T}$.  Hence, we obtain all possibilities for canonical tilings using initial sets of the form $\{ {\bf v}_1, \alpha_2 {\bf v}_2, \ldots, \alpha_k {\bf v}_k \}$.

Consider the group of symmetries $\text{Sym}(\mathcal{T}) \subset \text{Isom}(\mathbb{H}^3)$.  Since $\mathcal{T}$ is $\Lambda$-invariant $\Lambda \subset \text{Sym}(\mathcal{T})$.  On the other hand, $\text{Sym}(\mathcal{T})$ acts on the set $V$ of horospherical vectors.  It follows that $\text{Sym}(\mathcal{T})$ is discrete \cite[Lemma 2.1]{GHH} and therefore $\mathbb{H}^3/\Lambda \rightarrow \mathbb{H}^3/\text{Sym}(\mathcal{T})$ is a finite cover between orbifolds.

Suppose that $\Lambda$ is non-arithmetic.  Since $\text{Comm}(\Lambda)$ is the unique maximal discrete group that contains $\Lambda$, then $\text{Sym}(\mathcal{T}) \subset \text{Comm}(\Lambda)$ for every canonical tiling $\mathcal{T}$.  Futhermore, every canonical tiling for $\text{Comm}(\Lambda)$ is also a canonical tiling for $\Lambda$, hence $\text{Comm}(\Lambda) = \text{Sym}(\mathcal{T})$ for some canonical tiling $\mathcal{T}$ for $\Lambda$.

We say that a set $\{ \calp_i \}$ of ideal polyhedra $\Lambda$-{\em generate} the tiling $\calt$ if every tile of $\calt$ is of the form $\gamma \calp_i$ for some $\gamma \in \Lambda$ and some $i$.  The canonical tilings can be determined using elementary linear algebra.  According to \cite[Lemma 3.1]{GHH}, a set $\{ \calp_i \}$ of ideal polyhedra $\Lambda$-generates the canonical tiling associated to the set $V$ if 
\begin{enumerate}
\item $\Lambda \cdot \{ \calp_i\}$ is a tiling of $\mathbb{H}^3$,
\item  given any vertex of any $\calp_i$ there is a horospherical vector ${\bf v} \in V$ so that the vertex lies at the center of the horosphere $H_{\bf v}$,
\item\label{coplanar} the set of horospherical vectors corresponding to the vertices of any given $\calp_i$ lie on a single plane in $\R^4$,
\item\label{positive tilt} if $\calp_i$ and $\gamma \calp_j$ are two tiles that meet in a common face then the Euclidean planes in $\R^4$ determined by the two tiles meet convexly.
\end{enumerate}
The last two conditions can be re-phrased using linear algebra.  If $\{ {\bf v}_1,\hdots, {\bf v}_s \}$ are the horospherical vectors for $\calp_i$ and ${\bf w}$ is a horospherical vector for a neighboring tile which is not shared by $\calp_i$ then there exists a {\em normal} vector for $\calp_i$, ${\bf n} \in \R^4$ such that
\begin{enumerate}
\item[(3)] (coplanar) ${\bf n} \cdot {\bf v}_i =1$ for every $i=1, \ldots s$, and
\item[(4)] (positive tilt) ${\bf n} \cdot {\bf w} >1$,
\end{enumerate}
where $\cdot$ denotes the standard Euclidean inner product.  Observe that these conditions are invariant under $\text{Isom}(\mathbb{H}^3)$, for if ${\bf n} \cdot \bv=\alpha$ and $A \in \text{Isom}(\mathbb{H}^3)$ then $({\bf n} A^{-1}) \cdot A\bv= \alpha$.

\begin{prop} \label{tiling}  Let $\Delta_n < \mathsf{O}_0(3,1)$ be determined by the following embedding of the $\calp_i$ in $\mathbb{H}^3$: the isometry group of $\calp_2$ fixes $(0,0,0,1)^T$, the ideal vertex $\hat{v}$ shared by the $\calp_i$ is $[\hat\bv]$, where $\hat\bv = (2,0,0,2)^T$, and $\calp_1 \cap \calp_2$ has ideal vertices $[\hat\bv],[\bv_9], [\bv_4]$, where $\bv_4 = (1,1,-\sqrt{2},2)^T$ and $\bv_9 = (1,-1,-\sqrt{2},2)^T$.  Let $\calt_n$ be the tiling of $\mathbb{H}^3$ determined by $V_n =  \Delta_n \cdot \{\hat\bv\}$.  The tiles of $\mathcal{T}_n$ are the $\Delta_n$-orbits of the $\p_i$.
\end{prop}

\proof
If $X$ is a $4 \times n$ matrix we denote the $i^{\text{th}}$ column of $X$ by $x_i$.  When the columns of $X$ lie in $L^+$ and the convex hull of the corresponding ideal points is an ideal polyhedron we call the polyhedron $\p_X$.  Consider the matrices
\[M \ = \ \left(
\begin{array}{llllllllllll}
 2 & 1 & 0 & 1 & 0 & -1 & -2 & -1 & 1 & -1 & -1 & 1 \\
 0 & 1 & 2 & 1 & -2 & -1 & 0 & -1 & -1 & 1 & 1 & -1 \\
 0 & \sqrt{2} & 0 & -\sqrt{2} & 0 & \sqrt{2} & 0 & -\sqrt{2} & -\sqrt{2} & -\sqrt{2} & \sqrt{2} & \sqrt{2} \\
 2 & 2 & 2 & 2 & 2 & 2 & 2 & 2 & 2 & 2 & 2 & 2
\end{array}
\right).\]
\[N \ = \ \left(
\begin{array}{llllll}
 \sqrt{2} & 0 & 0 & -\sqrt{2} & 0 & 0 \\
 0 & \sqrt{2} & 0 & 0 & -\sqrt{2} & 0 \\
 0 & 0 & \sqrt{2} & 0 & 0 & -\sqrt{2} \\
 \sqrt{2} & \sqrt{2} & \sqrt{2} & \sqrt{2} & \sqrt{2} & \sqrt{2}
\end{array}
\right)\]
The columns of $M$ and $N$ are horospherical vectors and represent horospheres centered about the ideal vertices of a regular ideal cuboctahedron and octahedron respectively.  These matrices are chosen so that, for $X=M,N$, the isometries in $\text{Isom}(\p_X)$ all fix $(0,0,0,1)^T \in \mathbb{H}^3$ and the columns of $X$ are $\text{Isom}(\p_X)$--invariant.  Furthermore, if ${\sf h}$ is the orientation preserving hyperbolic isometry that takes the triangular face $(n_1, n_2, n_3)$ of $\p_N$ to the triangular face $(m_1, m_9, m_4)$ of $\p_M$ so that ${\sf h}(\p_N) \cap \p_M$ is exactly this face, then our choice of horospheres agree on this intersection.  That is, ${\sf h}(n_1, n_2, n_3)=(m_1, m_9, m_4)$.

Let $\p_1={\sf h}(\p_N)$ and $\p_2=\p_M$.  Embed the remaining polyhedra in $\{ \p_i \}_1^{n+2}$, as described above, so that the common ideal vertex is the center of the $m_1$ horosphere.  Choose horospherical vectors for the $\p_i$'s so that they are $\text{Isom}(\p_i)$--invariant and to coincide with the horospherical vectors of $\p_{i\pm1}$ wherever ideal vertices are shared.

Notice that the face pairings of $\p_i$ in $S_n$ are all compositions of elements of $\text{Isom}(\p_i)$ with parabolics that fix an ideal vertex of $\p_i$.  Since we have chosen our horospherical vectors to be $\text{Isom}(\p_i)$--invariant, it follows that our choice of horospheres is compatible with the face pairings in $S_n$.  Hence, the choice of horospheres descends to a choice of horospherical torus in $N_n$ and therefore determines a canonical cell decomposition of $N_n$ and a canonical tiling of $\mathbb{H}^3$ whose symmetry group is $G_n$.   To prove the proposition, we need to show that this tiling is $\mathcal{T}_n$.

Take ${\bf n} = (0,0,0,1/2)^T$.  Then ${\bf n} \cdot m_i =1$ for $i=1, \ldots 12$ and $\sqrt{2} {\bf n} \cdot n_i =1$ for $i=1,\ldots, 6$.  Therefore by Goodman--Hodgson--Heard's criterion (\ref{coplanar}), the horospherical vertices of ${\sf k} (\p_i)$ are coplanar for every ${\sf k} \in \Delta_n$.  It remains only to show that condition (\ref{positive tilt}) holds for adjacent pair of cuboctahedra that meet along a triangular face, an adjacent pair of cuboctahedra that meet along a square face, and an octahedron adjacent to a cuboctahedron.

If $Q$ is a cuboctahedron adjacent to $\p_M$ sharing the triangular face $(m_1, m_9, m_4)$ with $\text{Isom}(Q)$--invariant horospherical vectors which agree with $(m_1, m_9, m_4)$ then $w=(7, 1, -5 \sqrt{2}, 10)^T$ is a horospherical vector for $Q$ which is not shared by $\p_M$.  We have ${\bf n} \cdot w = 5 >1$.  If $Q$ is a cuboctahedron adjacent to $\p_M$ sharing the square face $(m_1, m_2, m_3, m_4)$ with $\text{Isom}(Q)$--invariant horospherical vectors which agree with $(m_1, m_2, m_3, m_4)$ then $w=(3, 5, - \sqrt{2}, 6)^T$ is a horospherical vector for $Q$ which is not shared by $\p_M$.  We have ${\bf n} \cdot w = 3 >1$.  The octahedron ${\sf h}(\p_N)$ is adjacent to $\p_M$ sharing the face $(m_1, m_9, m_4)$.  Its vectors are invariant under the isometry group of ${\sf h}(\p_N)$ and they agree with those of $\p_M$ along the shared face.  The vector $w=(2+2\sqrt{2}, 0, -2-2\sqrt{2}, 4+4\sqrt{2})^T$ is a horospherical vector for ${\sf h}(\p_N)$ which is not shared by $\p_M$.  We have ${\bf n} \cdot w = 2+\sqrt{2}>1$.
\endproof

For $i=2, n+1$, shade each face of $\p_i$ gray if it is identified with a face of an octahedron in the quotient.  For the other cuboctahedra $\p_i$, color each triangular face red if it is identified with a face of $\p_{i-1}$.  (In Figure \ref{one}, every white triangular face of $\p_3$ should be colored red.)

The tiles of $\mathcal{T}_n$ inherit a coloring from the coloring of the $\p_i$'s.   We can further classify the triangular faces in cuboctahedral tiles of $\mathcal{T}$ into \emph{type I} and \emph{type II} triangles.  A face of a cuboctahedral tile $T$ is type I if it has exactly one ideal vertex that is shared by a triangular face of $T$ of the opposite color.  Triangular faces of cuboctahedra that are not type I are type II.

\proof[Proof of Proposition \ref{OP comm}]    
Suppose ${\sf h} \in G_n - \Delta_n$.  By \cite{GHH}, ${\sf h}$ is a symmetry for the tiling $\mathcal{T}_n$.  The polyhedron $\mathcal{D}_n$ is a fundamental domain for $\Delta_n$, so by composing ${\sf h}$ with some element of $\Delta_n$, we may assume that ${\sf h}(\p_2) \in \{ \p_i \}_1^{n+2}$.  It is clear that ${\sf h}$ must preserve the set of gray faces in the tiling, hence ${\sf h}(\p_2)$ is either $\p_2$ or $\p_{n+1}$.

The isometry ${\sf h}$ must also preserve the types of the triangular faces of cuboctahedra.  By examining the combinatorics of the face pairings in $S_n$, we see that every cuboctahedron in the tiling has exactly two vertices that are shared by a pair of type I triangles.  There is one such vertex for each of the two triangular colors on the tile.  Let $v$ be the vertex of $\p_2$ which is shared by the two gray type I triangles of $\p_2$ and $w$ the vertex shared by the two white type I triangles.   If ${\sf h}(\p_2)=\p_2$ then, by considering the coloring of $\p_2$ we see that ${\sf h}$ must be the order-2 elliptic fixing $v$ and $w$.  If, on the other hand, we have ${\sf h}(\p_2)=\p_{n+1}$ then ${\sf h}(v)$ must be the vertex shared by the two gray type I triangles of $\p_{n+1}$ and ${\sf h}(w)$ must be the vertex shared by the two red type I triangles of $\p_{n+1}$.  The gray pattern on $\p_{n+1}$ forces ${\sf h}$ to be orientable.
\endproof

%%%%%%%%%%%%%%%%%%%%%%%%%%%%
\section{Augmented links}\label{sec: augmented}
%%%%%%%%%%%%%%%%%%%%%%%%%%%%

A rich class of examples that satisfy the hypotheses of Theorem \ref{rt ang cox} is that of the \textit{augmented links}.  These were introduced by Adams \cite{Adams} and further studied in e.g.~\cite{LAD}, \cite{Purcell}, and \cite{Purcell_cusps}.  In this section we will describe their construction and, in Section \ref{sec:low cx augmented}, classify up to scissors congruence the complements of augmented links with at most $5$ twist regions.  We will discuss when an augmented link complement is commensurable with a right-angled reflection orbifold, and in Section \ref{sec:lobel} describe an infinite family of augmented link complements that do not have this property.

A link $L$ in $S^3$ with hyperbolic complement determines (not necessarily uniquely) an augmented link using a projection of $L$ which is \textit{prime} and \textit{twist-reduced}.  We will regard a projection of $L$ as a $4$-valent graph in the plane, together with crossing information at each vertex, and use the term \textit{twist region} to denote either a maximal collection of bigon regions of the complement arranged end-to-end or an isolated crossing that is not adjacent to any bigon.

A projection is prime if there is no simple closed curve $\gamma$ in the projection plane intersecting it in exactly two points, with the property that each component of the complement of $\gamma$ contains a crossing.  A projection is twist-reduced if for every simple closed curve $\gamma$ in the projection plane which intersects it in four points, such that two points of intersection are adjacent to one crossing and the other two are adjacent to another, there is a single twist region containing all crossings in one component of the complement of $\gamma$.  

\begin{figure}
\includegraphics{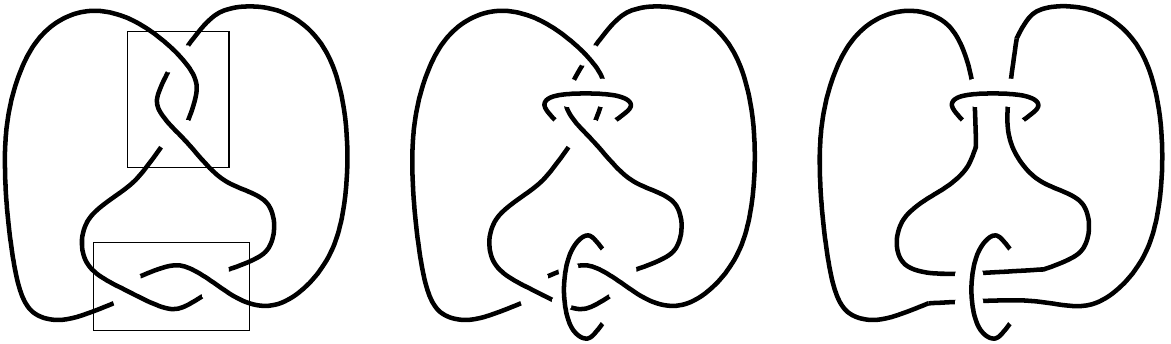}
\caption{Augmenting the figure-$8$ knot.}
\label{augmentfig8}
\end{figure}

An augmented link is obtained from a prime, twist reduced projection by encircling each twist region with a single unknotted component, which we call a \textit{clasp}.  This process is illustrated in Figure \ref{augmentfig8} for the figure-$8$ knot, pictured on the left-hand side with its twist regions in boxes.  The augmented link that it determines is pictured in the middle of the figure.  Each link with hyperbolic complement admits a prime, twist reduced diagram, and the augmented link obtained from such a diagram also has hyperbolic complement (a direct proof of this fact is given in Theorem 6.1 of \cite{Purcell_cusps}).  Thus every hyperbolic link complement in $S^3$ is obtained by Dehn surgery on some cusps of the complement of an augmented link.  

Each clasp of an augmented link $L$ bounds a disk that has two points of transverse intersection with $L$.  Given such a disk $D$, a family of homeomorphisms of $S^3-L$ is determined by cutting along the twice-punctured open disk $D-L$ and re-gluing by a rotation of angle $n\cdot 2\pi$, where $n \in \mathbb{Z}$.  This adds or subtracts $2n$ crossings to the twist region of $L$ encircled by the clasp bounding $D$.  It follows that the link on the right-hand side of Figure \ref{augmentfig8} has a complement homeomorphic to that of the link in the middle.  The complements of two augmented links that differ by only a single crossing in a twist region are not necessarily homeomorphic; however, we will see below that they are scissors congruent.   We also have:

\begin{lemma}\label{projection reflection}  Let $L$ be an augmented link.  Reflection through the projection plane determines an automorphism of $S^3 - L$.  \end{lemma}

This is because while such a reflection changes the sign of each crossing, it does not change the parity of the number of crossings per twist region.

Given an augmented link projection, the appendix to \cite{LAD} describes a decomposition of its complement into two isometric ideal polyhedra.  These polyhedra may be checkered so that each white face lies in the projection plane and each dark face is an ideal triangle in a ``vertical'' twice-punctured disk.  This is illustrated in Figure \ref{augmentpoly} for an augmented link with two twist regions.  

\begin{figure}
\includegraphics{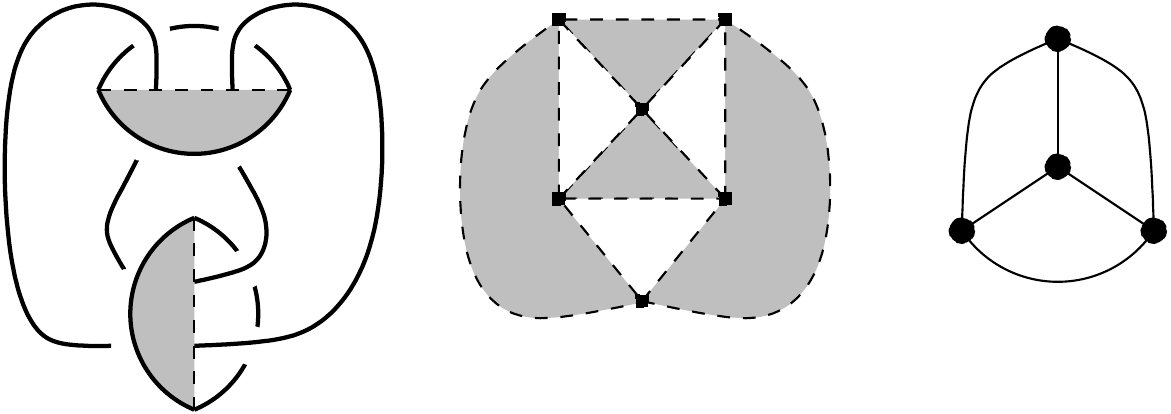}
\caption{An augmented link, the associated polyhedron, and its \crush.}
\label{augmentpoly}
\end{figure}

On the left-hand side of the figure, the dotted lines divide each twice-punctured clasp disk into the union of two ideal triangles.  We arrange for these disks to meet the projection plane transversely in the dotted lines, so the darkened ideal triangles lie above the projection plane and the others below it.  Cutting the link complement along the clasp disks and the projection plane yields two ideal polyhedra, one above and one below the projection plane, with edges coming from the dotted arcs.  These are  isomorphic by reflection through the projection plane.  Flattening the two-skeleton of the polyhedron above it onto the plane yields the polyhedron in the middle of the figure, an ideal octahedron, where each of the darkened half-disks on the left-hand side gives rise to two ideal triangles and the link itself has been shrunken to darkened rectangles at the vertices.  (See also \cite[Figure 3]{Purcell}.)

If $L$ is an augmented link, after removing all crossings in each twist region, we call the polyhedron produced by cutting along the projection plane and clasp disks the \textit{ideal polyhedron associated to $L$}.  This polyhedron may be checkered by coloring black the triangular faces that lie in clasp disks and white the faces that lie in the projection plane.  Note also that each black triangular face has a unique ideal vertex corresponding to a clasp.  The following lemma summarizes the construction of the appendix to \cite{LAD}, in our language.

\begin{lemma}\label{LAD decomp}  If $L$ is an augmented link with hyperbolic complement, there is a right-angled checkered ideal polyhedron $\calp$ in $\mathbb{H}^3$ combinatorially isomorphic to the ideal polyhedron associated to $L$.  For a face $f$ of $\calp$, let $\rho_f$ denote reflection in the plane containing $f$.  Fix a white face $f_0$ of $\calp$, and let $\overline{\calp} = \rho_{f_0}(\calp)$, $\bar{f} = \rho_{f_0}(f)$ for each face $f$ of $\calp$, and $\bar{v} = \rho_{f_0}(v)$ for each ideal vertex.  Then the quotient of $\calp \cup \overline{\calp}$ by the following face pairing gives a right-angled ideal decomposition of $S^3-L$.  \begin{enumerate}
\item  If $f \neq f_0$ is a white face of $\calp$, let $\phi_f = \rho_{f_0} \circ \rho_f$, taking $f$ to $\bar{f} \subset \overline{\calp}$.
\item  If $f$ is a black triangular face of $\calp$, let $f'$ be the black face of $\calp$ that shares the ideal vertex $v$ of $f$ corresponding to a clasp.  \begin{enumerate}
   \item\label{even twist}  If the corresponding twist region has an even number of crossings, let $\phi_f$ be the unique orientation-preserving isometry with $\phi_f(f) = f'$, $\phi_f(v) = v$, and $\phi_f(\calp) \cap \calp= f'$.
   \item\label{odd twist}  If the corresponding twist region has an odd number of crossings, let $\phi_f$ be the unique orientation-preserving isometry with $\phi_f(f) = \bar{f}'$, $\phi_f(v) = \bar{v}$, and $\phi_f(\calp)\cap\overline{\calp} = \bar{f}'$.  \end{enumerate}  \end{enumerate} 
Furthermore, $\rho_{f_0}$ induces the isometry of $S^3 - L$ supplied by Lemma \ref{projection reflection}.  In particular, $\phi_{\bar{f}} = \rho_{f_0} \circ \phi_f \circ \rho_{f_0}$ for each face $f$ of $\calp$. \end{lemma}

For another discussion of the content of Lemma \ref{LAD decomp}, see \cite[\S 2.3]{Purcell}.  In particular, Figure 4 there clarifies the different gluings producing twist regions with even vs.~odd numbers of crossings.  The last sentence of the lemma is not covered in \cite{LAD}; however it follows easily from the discussion above.

On the right-hand side of Figure \ref{augmentpoly} is the compact polyhedron obtained from the checkered ideal octahedron by the following rule: it has a vertex corresponding to every dark face and an edge joining each pair of vertices that correspond to dark faces which share ideal vertices.  We will call this the \textit{\crush} of $L$, since it may be regarded as obtained by crushing the darkened faces of the associated right-angled polyhedron to points.  We note that each vertex of the \crush\ has valence 3, since each dark face is an ideal triangle.  The right-angled ideal polyhedron associated to $L$ is recovered by truncation from its \crush.

For an alternative perspective on obtaining the \crush\ and a connection with Andreev's theorem, we refer the reader to Section 6 of \cite{Purcell_cusps}, in particular page 487.  The one-skeleton of the \crush\ is the graph $\Gamma$ dual to the nerve $\gamma$ of the circle packing defined there.  We thank Jessica Purcell for pointing this out.

Figure \ref{6prismlinks} illustrates two augmented links with the same underlying polyhedron,  each depicted draped over the one-skeleton of its \crush, the $6$-prism.  (More generally, for $k \geq 3$ we will call the $k$-\textit{prism} the polyhedron combinatorially isomorphic to the cartesian product of a $k$-gon with an interval.)  Since the associated right-angled ideal polyhedron is obtained by truncating vertices of the \crush, its ideal vertices occur at midpoints of edges.  Each triangular face resulting from truncation is paired with one of its neighbors across an ideal vertex producing a clasp; thus for each vertex of the \crush, exactly one edge which abuts it is encircled by a clasp.  Each other edge carries a single strand of the ``horizontal'' component of the augmented link.  

\begin{figure}
\includegraphics{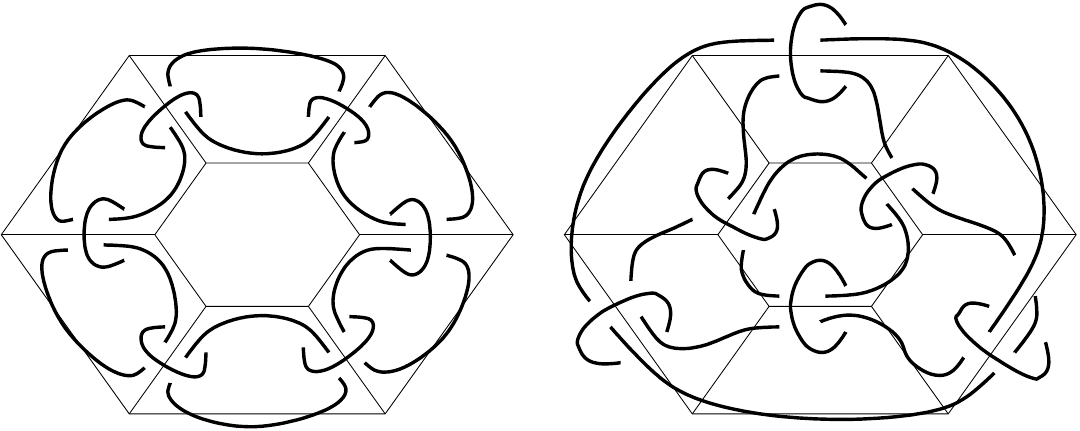}
\caption{Two augmented links with \crush\ the $6$-prism.}
\label{6prismlinks}
\end{figure}

Since the ideal polyhedron $\calp$ associated to an augmented link is canonically obtained from its \crush, each symmetry of the \crush\ determines a combinatorial symmetry of $\calp$.  Together with Mostow rigidity, this implies:

\begin{lemma}\label{symmetric crush}  Let $L$ be an augmented link, $\calp$ the associated right-angled ideal polyhedron in $\mathbb{H}^3$, and $\calc$ its \crush.  There is a canonical injection $\mathrm{Sym}(\calc) \to \mathrm{Sym}(\calp)$.  \end{lemma}

Lemma \ref{symmetric crush} implies that the complement of an augmented link with a highly symmetric \crush\ may be commensurable with the reflection group in the associated right-angled polyhedron.

\begin{lemma}\label{symmetric links}  Let $L$ be an augmented link, $\calp$ the associated right-angled polyhedron, and $\calc$ its \crush, and suppose $\calc$ has the property that for each clasp component $K$ of $L$, corresponding to an edge $e$ of $\calc$ with vertices $v$ and $v'$, \begin{enumerate}
\item  if $K$ encloses a twist region with an even number of crossings, there is a reflective involution of $\calc$ preserving $e$ and exchanging $v$ with $v'$.  
\item  if $K$ encloses a twist region with an odd number of crossings, there is a rotational involution of $\calc$ preserving $e$ and exchanging $v$ with $v'$.  \end{enumerate}
Then $\pi_1(S^3-L) < \Gamma_{\calp} \rtimes \mathrm{Sym}(\calp)$, where $\Gamma_{\calp}$ is the group generated by reflections in $\calp$ and $\mathrm{Sym}(\calp)$ is the group of symmetries of $\calp$.\end{lemma}

\begin{proof}  Lemma \ref{symmetric crush} implies that for each edge $e$ of $\calc$ corresponding to a clasp $K$ of $L$, there is an involution $\iota_e$ of $\calp$ that exchanges the triangular faces $f$ and $f'$ corresponding to $v$ and $v'$, and fixes the ideal vertex that they share.   This involution is  a reflection or $180$-degree rotation in case (1) or (2) above, respectively.

We now use the notation of Lemma \ref{LAD decomp}, and record that case (\ref{even twist}) there is the same as case (1) above.  In this case, $\iota_e \circ \rho_f$ realizes the orientation-preserving isometry $\phi_f$ there.  In case (2) above, corresponding to case (\ref{odd twist}) of Lemma \ref{LAD decomp}, the required isometry $\phi_f$ is realized by $\rho\circ\iota_e\circ\rho_f$.  \end{proof}

Lemma \ref{symmetric links} implies for instance that the link on the left-hand side of Figure \ref{6prismlinks} is commensurable with the reflection group in the corresponding right-angled polyhedron, but it does not apply to the link on the right-hand side on account of the twist region with a single crossing.  On the other hand, the commensurability classes of some links are entirely determined by their \crush s.

\begin{cor} \label{rt ang prism}  Suppose $L$ is an augmented link such that the \crush\ of $L$ is a regular polyhedron.  Then $\pi(S^3-L)$ is commensurable with the reflection group in the sides of the corresponding right-angled polyhedron.  \end{cor}

In some cases the \crush\ of an augmented link may not have much symmetry, but it may be built from highly symmetric polyhedra.  In such cases the link may have hidden symmetries.  We will say a \crush\ is \textit{decomposable} if it contains a prismatic $3$-cycle --- that is, a sequence of three faces so that any two intersect along an edge but all three do not share a common vertex --- and \textit{indecomposable} otherwise.  

\begin{figure}
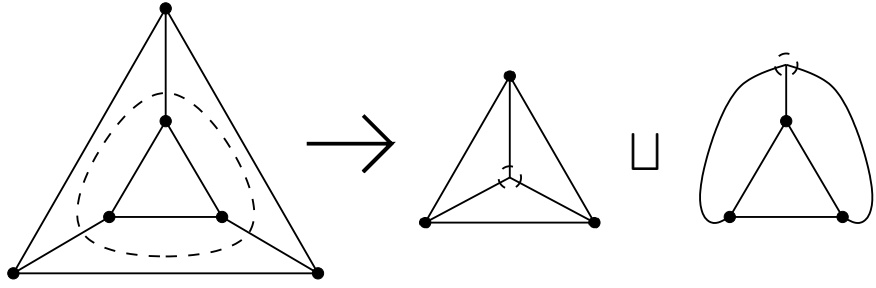
\caption{Decomposing the $3$-prism into two tetrahedra}
\label{decompose3prism}
\end{figure}

If $\calc$ is a decomposable \crush, we decompose along a prismatic $3$-cycle by selecting a simple closed curve $\gamma$ which lies in the union of the faces of the cycle and intersects each of the edges of the cycle once, and using the following procedure: cut along $\gamma$, separate the components that result, and complete each by replacing $\gamma$ with a single vertex containing the endpoints of all the three edges intersecting it.  This is illustrated for the triangular prism in Figure \ref{decompose3prism}, with the dotted curve on the left-hand side representing $\gamma$.  Decomposing results in a disjoint union of two tetrahedra.

Suppose $L$ is a link with a decomposable \crush\ $\calc$, and let $f_0$, $f_1$, and $f_2$ determine a prismatic $3$-cycle of $\calc$.  Then the corresponding faces in the associated right-angled ideal polyhedron $\calp$, obtained by truncating vertices of $\calc$, do not pairwise intersect but each two share an ideal vertex.  It is an elementary fact of hyperbolic geometry that there is a single hyperplane $\calh$ which perpendicularly intersects the hyperplanes containing each of $f_0$, $f_1$ and $f_2$.  Cutting $\calp$ along $\calh$ decomposes it into two new right-angled ideal polyhedra, each with an ideal triangular face contained in $\calh$.  Their \crush s are obtained by decomposing $\calc$ along the prismatic cycle determined by $f_0$, $f_1$, and $f_2$.

\begin{lemma}\label{hidden symmetries}  Suppose $L$ is an augmented link such that the \crush\ of $L$ decomposes into a disjoint union of copies of $\calc$, where $\calc$ is a regular polyhedron.  Then $\pi_1(S^3-L)$ is contained in $\Gamma_{\calp} \rtimes \mathrm{Sym}(\calp)$, where $\calp$ is the right-angled ideal polyhedron obtained from $\calc$ by truncating vertices.  \end{lemma}

\begin{proof}  There is a tiling $\calt$ of $\mathbb{H}^3$ consisting of $\Gamma_{\calp}$-translates of $\calp$.  If $\calc_0$ is the \crush\ of $L$, then the hypothesis and the description above the lemma establish that the associated right-angled polyhedron $\calp_0$ is a union of tiles of $\calt$.  Checkering $\calp_0$ so that dark faces are triangles obtained by truncating vertices of $\calc_0$, we claim that for each pair of dark faces $f$ and $f'$ which share an ideal vertex $v$, there exist in $\Gamma_{\calp} \rtimes \mathrm{Sym}(\calp)$ both a reflective and a rotational involution of $\mathbb{H}^3$ exchanging $f$ and $f'$ and fixing $v$.  We will prove the claim by induction on the number of tiles comprising $\calp_0$.  The case of one tile, $\calc_0 = \calc$, follows as in the proof of Lemma \ref{symmetric links} from the fact that $\calc$ is regular.  

Suppose that $\calp_0$ is the union of more than one tile, and let $\gamma(\calp)$ be a $\Gamma_{\calp}$-translate of $\calp$ such that $\calp_0$ is the union of $\gamma(\calp)$ and a polyhedron $\calp_1 \subset \calt$ across a face $f$ which is an ideal triangle.  The checkering of $\calp_0$ determines checkerings of each of $\calp_1$ and $\gamma(\calp)$ by declaring $f$ to be dark.  The claim holds for $\calp_1$ by induction and for $\gamma(\calp)$ by the base case.  Thus it only remains to verify the claim for dark faces of $\calp_0$ sharing an ideal vertex, one of which lies in $\calp_1$ and one in $\gamma(\calp)$.

Suppose $f_0$ and $f_1$ are dark faces, of $\gamma(\calp)$ and $\calp_1$ respectively, which share an ideal vertex $v$ in $\calp_0$.  Then each of $f_0$ and $f_1$ shares $v$ with $f$.  Let $\rho_0$ (respectively, $\rho_1$) be a reflective involution in $\Gamma_{\calp} \rtimes \mathrm{Sym}(\calp)$ fixing $v$ and exchanging $f_0$ (resp. $f_1$) with $f$, and let $\iota_0$ and $\iota_1$ be rotational involutions satisfying the same description.  Then $\rho_1 \circ \rho_0$ and $\iota_1 \circ \rho_0$ are isometries of infinite order taking $f_0$ to $f_1$.  This can be discerned by considering their actions on a horosphere centered at $v$, intersected by $\gamma(\calp)$ and $\calp_1$ in adjacent rectangles.  The first acts on this cross section as a translation and the second as a glide reflection.  If $\rho_{f_1}$ is reflection in the hyperplane containing $f_1$, it follows that $\rho_{f_1} \circ \rho_1 \circ \rho_0$ and $\rho_{f_1} \circ \iota_1 \circ \rho_0$ satisfy the conclusion of the claim.

The conclusion of the lemma now follows from Lemma \ref{LAD decomp}.  \end{proof}

%%%%%%%%%%%%%%%%%%%%%
\subsection{Examples with low complexity}\label{sec:low cx augmented}
%%%%%%%%%%%%%%%%%%%%%

The most natural measure of complexity of an augmented link is the number of twist regions, which is equal to half the number of dark faces of the associated right-angled polyhedron, or half the number of vertices of its \crush.  Here we will classify the augmented link complements with up to five twist regions up to \textit{scissors congruence}.  We will say that finite-volume hyperbolic 3-manifolds are scissors congruent if they can be cut into identical collections of ideal polyhedra.  It is natural for us to use this invariant because many different augmented links may be produced by different choices of face pairing on the same underlying right-angled polyhedron.  

\begin{lemma}\label{indecomp up to ten}  The indecomposable \crush s with at most ten vertices are the tetrahedron, the cube (or $4$-prism), and the $5$-prism.  \end{lemma}

\begin{proof}  The only indecomposable \crush\ with a triangular face is the tetrahedron, since the family of faces adjacent to a triangular face determines a prismatic $3$-cycle unless they share a common vertex.  On the other hand, if a \crush\ $\calc$ with at most ten vertices has a face which is a $k$-gon for $k \geq 6$, then two edges which emanate from distinct vertices of this face must share a common endpoint.  That $\calc$ is decomposable follows from the claim below.

\begin{claim}  Suppose $f$ is a face of a \crush\ $\calc$, and $e_0$ and $e_1$ are distinct edges of $\calc$, each with one endpoint on $f$, which share a vertex $v$.  Then $e_0$ and $e_1$ bound a triangle face of $\calc$ together with an edge of $f$.  \end{claim}

\begin{proof}[Proof of claim]  The set $f \cup e_0 \cup e_1$ cuts $\partial \calc$ into two disks.  Let $D$ be the closure of the disk that does not intersect the edge $e_2 \neq e_0, e_1$ with an endpoint at $v$.  There is a face $f' \subset D$ of $\calc$ which has $v$ as a vertex and $e_0$ and $e_1$ as edges.  Then $f'$ intersects $f$ along an edge $e_0'$ with an endpoint at $e_0 \cap f$ and also along an edge $e_1'$ with an endpoint at $e_1\cap f$.  But since $f$ and $f'$ cannot meet along more than one edge, we must have $e_0' = e_1'$.  Thus since $e_0 \cup e_1 \cup e_0'$ forms a simple closed curve in the boundary of $f'$, $f' = D$ is a triangle.  \end{proof}

Thus if $\calc$ is indecomposable and not a tetrahedron, with at most ten vertices, then every face of $\calc$ is a quadrilateral or pentagon.  Let $j$ be the number of quadrilateral faces and $k$ the number of pentagon faces, and let $v$ and $e$ be the number of vertices and edges, respectively.  Since each vertex is $3$-valent, we have $3v=2e$, and since each edge bounds two faces we have $2e = 4j+5k$.  Computing the Euler characteristic thus yields:
$$ v - e + (j+k) = \frac{4j+5k}{3} - \frac{4j+5k}{2} + (j+k) = \frac{j}{3}+\frac{k}{6}=2.  $$
Using the equation above we find that $j+k/2 = 6$.  Since we require that $\calc$ have at most ten vertices, the vertex and edge equations yield $4j+5k \leq 30$.  Thus using the fact that $j$ and $k$ are non-negative integers, we find that either $j=6$ and $k=0$ (and hence $v=8$) or $j=5$ and $k=2$ (and $v= 10$).  The cube and the $5$-prism respectively realize these possibilities.  It remains to show that these are the unique \crush s with the prescribed numbers of quadrilateral and pentagon faces.

In general, if a \crush\ $\calc$ has a $k$-gon face which is adjacent to only quadrilaterals, then $\calc$ is the $k$-prism.  This immediately implies that the only \crush\ with six quadrilateral faces and no pentagons is a cube.  Similarly, if $\calc$ is an indecomposable \crush\ with two pentagonal faces and five quadrilaterals, then $\calc$ is a $5$-prism unless the pentagonal faces are adjacent.  In the latter case, we note that the union of the pentagonal faces has eight vertices, and by the claim above and indecomposability, the three ``free'' edges emanating from one of them have distinct vertices.  Hence $\calc$ has at least eleven vertices, a contradiction.  Therefore the $5$-prism is the only indecomposable \crush\ with five quadrilateral faces and two pentagons.  \end{proof}

\begin{lemma}\label{decomp up to ten}  If $\calc$ is a decomposable \crush\ with at most ten vertices, a maximal sequence of decompositions yields a disjoint union of up to four tetrahedra or of a single tetrahedron and a single cube.  \end{lemma}

\begin{proof}  Suppose $\calc$ is a decomposable \crush, and let $\calc_0$ and $\calc_1$ be obtained by decomposing $\calc$ along a prismatic $3$-cycle.  If $v$, $v_0$, and $v_1$ are the numbers of vertices of $\calc$, $\calc_0$, and $\calc_1$, respectively, then from the description of decomposition one finds that 
$$v+2 = v_0 + v_1.$$
It is easy to see that each \crush\ has at least four vertices, and that the tetrahedron is the unique such with exactly four.  Thus by the equation above, any \crush\ with six vertices decomposes into two tetrahedra.  (By the classification of indecomposable \crush s, every \crush\ with six vertices is decomposable.)  If $\calc$ is a decomposable \crush\ with eight vertices, we thus find that a sequence of two decompositions yields a disjoint union of three tetrahedra.

Finally, suppose that $\calc$ is a decomposable \crush\ with ten vertices, and decompose it along a prismatic $3$-cycle into \crush s $\calc_0$ and $\calc_1$ with $v_0 \leq v_1$ vertices, respectively.  Then either $v_0 = v_1 = 6$ or $v_0 = 4$ and $v_1=8$.  In the former case, the above implies that neither $\calc_0$ nor $\calc_1$ is indecomposable; hence each decomposes into a disjoint union of two tetrahedra.  In the case $v_0 = 4$ and $v_1 = 8$, $\calc_0$ is a tetrahedron.  If $\calc_1$ is indecomposable, it is a cube; othewise, a sequence of two decompositions cuts it into a disjoint union of three tetrahedra.  \end{proof}

The scissors congruence classification of augmented links with up to five twist regions is now readily obtained.  Below let $L$ be an augmented link.  \begin{itemize}
\item  If the \crush\ of $L$ decomposes into a disjoint union of tetrahedra, then $S^3-L$ is a union of right-angled ideal octahedra.  It thus follows from Lemma \ref{hidden symmetries} and the results of \cite{Hatcher} that $\pi_1(S^3-L) < \mathrm{PGL}_2(\calo_1)$.  This holds in particular for all augmented links with at most three twist regions, or for any with four twist regions and a decomposable \crush.
\item  If $L$ has four twist regions and an indecomposable \crush, then $S^3-L$ is a union of two right-angled ideal cuboctahedra, and by Corollary \ref{rt ang prism} and the results of \cite{Hatcher}, $\pi_1(S^3-L) < \mathrm{PGL}_2(\calo_2)$.  \end{itemize}

\begin{figure}
\begin{center}
\includegraphics{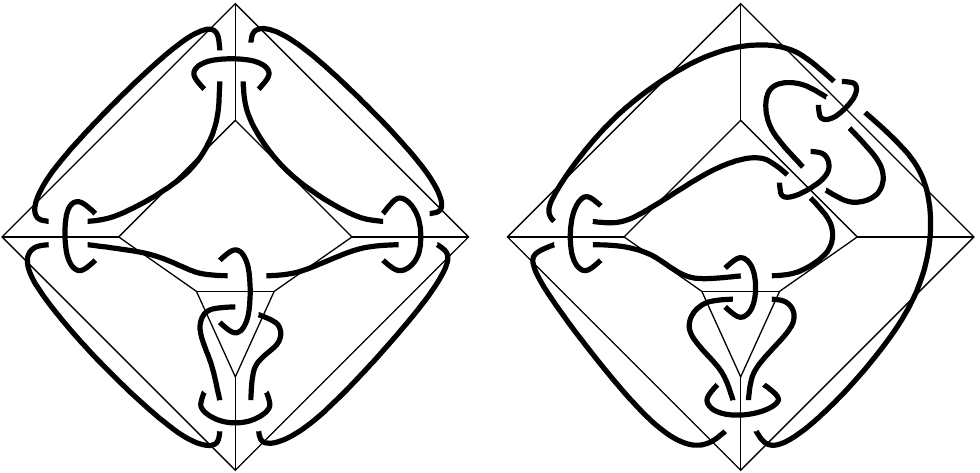}
\end{center}
\caption{Augmented links with $5$ twist regions and a decomposable \crush.}
\label{decomp5twist}
\end{figure}

In particular, the commensurability class of an augmented link with at most four twist regions is determined by its \crush, and each such link falls into one of two commensurability classes.  The augmented links with five twist regions display more variability.  \begin{itemize}
\item  If $L$ has five twist regions and an indecomposable \crush\ $\calc$, then $\calc$ is the $5$-prism.  In most cases, we have $\pi_1(S^3-L) < \Gamma_{\calp} \rtimes \mathrm{Sym}(\calp)$, where $\calp$ is the associated right-angled polyhedron, the double of the $5$-antiprism across one of its pentagon faces.  This holds by Lemma \ref{symmetric links}, unless $L$ has a twist region with an odd number of crossings that corresponds to an edge of a pentagon face of $\calc$. 
\item  If $L$ has five twist regions and a decomposable \crush\ that does not decompose into tetrahedra, then $S^3-L$ is a union of two right-angled octahedra and two cuboctahedra.  Two such links are pictured in Figure \ref{decomp5twist}.  Using the techniques of \cite[\S 4.3]{CD}, one can show that the horizontal component that runs across all vertices of the \crush\ on the right-hand side has cusp parameter that is $\mathrm{PGL}_2(\mathbb{Q})$-inequivalent to the parameters of all cusps of the left-hand link.  Hence their complements are incommensurable.  \end{itemize}

From the classification above, we find that an augmented link with at most five twist regions is almost determined up to commensurability by its \crush.  This is primarily because the indecomposable \crush s with at most ten vertices have so much symmetry.  Already among those with twelve vertices, we find an example with less symmetry.  This is pictured on the left-hand side of Figure \ref{indecomp}.  On the right-hand side is an augmented link that has this polyhedron as a \crush.

\begin{figure}[ht]
\begin{center}
\includegraphics{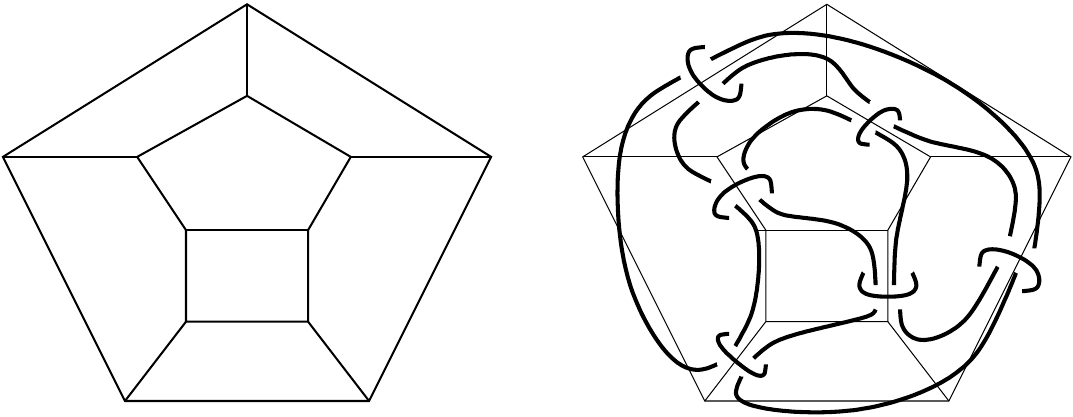}
\end{center}
\caption{An indecomposable \crush\ with $12$ vertices, and an augmented link built on it.}
\label{indecomp}
\end{figure}

\begin{lemma}\label{indecomp twelve}  The indecomposable \crush s with twelve vertices are the $6$-prism and the polyhedron on the left-hand side of Figure \ref{indecomp}.  \end{lemma}

\begin{proof}  Reasoning as in the proof of Lemma \ref{indecomp up to ten}, we find that a \crush\ with twelve vertices and a face which is a $k$-gon for $k > 6$ is decomposable, and that such a \crush\ with a hexagonal face is the $6$-prism.  Thus as in the proof of that lemma, we are left to consider \crush s with all quadrilateral and pentagon faces.  If $j$ is the number of quadrilateral and $k$ the number of pentagonal faces, an Euler characteristic calculation again yields $j +k/2 = 6$.  Counting vertices in this case yields $4j + 5k = 36$, and solving these two equations yields $j = 4$ and $k = 4$.

Let $\calc$ be an indecomposable \crush\ with twelve vertices and $4$ each of quadrilateral and pentagon faces.  Then every pentagon face of $\calc$ is adjacent to at least one other pentagon face.  

\begin{claim}  No vertex of $\calc$ is shared by three pentagon faces.  \end{claim}
\begin{proof}[Proof of claim]  Suppose $v$ is a vertex with this property, and let $v_0$, $v_1$, and $v_2$ be the vertices adjacent to $v$ in the one-skeleton of $\calc$.  Then for $i \in \{0,1,2\}$, let $f_i$ be the face of $\calc$ which contains $v_i$ but not $v$.  We may assume without loss of generality that $f_0$ and $f_1$ are quadrilaterals (at least two must be).

Consider the subcomplex of $\partial \calc$ which is the union of $f_0$, $f_1$, and the pentagon faces containing $v$.  If any edges on the boundary of this subcomplex were identified in $\partial \calc$, then it would have a prismatic $k$-cycle for $k \leq 3$; hence this subcomplex is a disk embedded in $\partial \calc$.  It contains all twelve vertices, and sixteen out of the eighteen edges of $\calc$.  But it is easy to see that any way of joining the four ``free'' vertices by two edges in the complement yields a triangular face, contradicting indecomposability.
\end{proof}

One may also rule out the possibility of a quadrilateral face which meets only pentagonal faces ---  the union of these faces would be an embedded disk containing all twelve vertices but only fourteen edges --- and to establish that each pentagonal face meets at least two other pentagonal faces.  
Thus the pentagonal faces form a prismatic $4$-cycle of $\calc$, neither of whose complementary regions can be occupied by a single quadrilateral.  It follows that $\calc$ is as pictured in Figure \ref{indecomp}.
\end{proof}

%%%%%%%%%%%%%%%%%%%%%%%
\subsection{L\"obell links}  \label{sec:lobel}
%%%%%%%%%%%%%%%%%%%%%%%

\begin{figure}[ht]
   \centering
   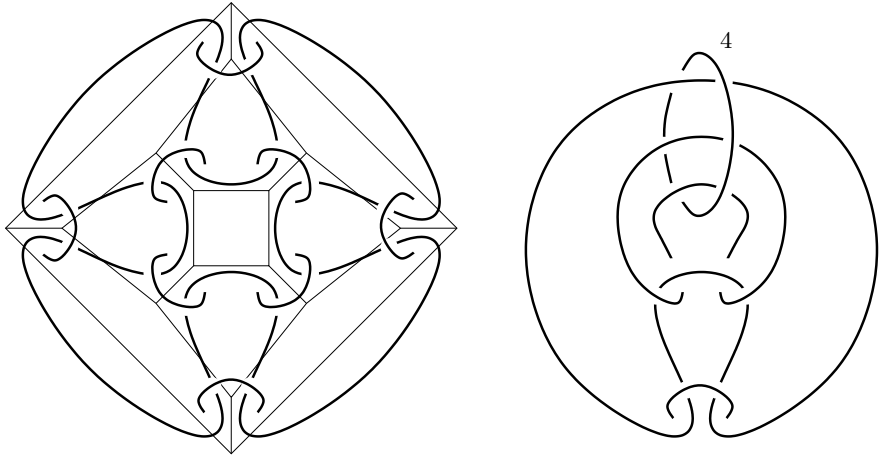
   \caption{The L\"obell link $L(4)$ and its $4$-fold cyclic quotient.}
   \label{fig:L4 and quot}
\end{figure}

For $n \geq 3$, we will denote by $\call(n)$ the $n$th L\"obell polyhedron.  This is the unique polyhedron with vertices of valence $3$ and faces consisting of $n$-gons $F$ and $F'$, and $2n$ pentagons, such that $F$ has distance $3$ from $F'$ in the dual graph.  The L\"obell polyhedron $\call(4)$ is pictured on the left-hand side of Figure \ref{fig:L4 and quot}, under a link that has it as a \crush.  We denote this link $L(4)$.  There is an evident rotational symmetry of $(S^3,L(4))$, with order $4$ and quotient the link on the right-hand side of Figure \ref{fig:L4 and quot}.  An additional component, the fixed axis of this rotation, has been added to the diagram and labeled with $4$.  For arbitrary $n \geq 3$, we define $L(n)$ to be the link with \crush\ $\call(n)$ that $n$-fold branched covers the diagram on the right-hand side.  The main result of this section is:

\begin{thm}\label{Lobell thm}  For all but finitely many $n \geq 4$, $M(n) \doteq S^3 - L(n)$ is not arithmetic nor commensurable with any $3$-dimensional hyperbolic reflection orbifold.  Moreover, at most finitely many $M(n)$ occupy any commensurability class.  \end{thm}

\begin{remark}  Since $\call(5)$ is the dodecahedron, $L(5)$ falls under the purview of Corollary \ref{rt ang prism} and so is commensurable with a right-angled reflection orbifold.  Therefore the stipulation ``all but finitely many'' above is necessary.  We do not know of any $M(n)$ that is arithmetic, however.  We note also that $\call(3)$ decomposes into two tetrahedra and a cube, whereas $\call(n)$ is indecomposable for $n > 3$.  \end{remark}

Proving the theorem requires identifying the commensurator quotient of $M(n)$.  We begin by identifying the symmetry group of $\call(n)$.  

\begin{fact}  For $n \neq 5$, the symmetry group of $\call(n)$ has presentation
$$ \Sigma(n) = \langle\ \sfa, \sfb_n, \sfs\,|\, (\sfb_n)^n = \sfs^2 = \sfa^2 = 1, \sfs\sfb_n \sfs = (\sfb_n)^{-1}, \sfa\sfb_n\sfa = (\sfb_n)^{-1}, \sfa \sfs \sfa = \sfb_n\sfs\ \rangle.  $$
The subgroup $\langle \sfa, \sfb_n\rangle$ preserves orientation, and $\sfs$ reverses it.  The subgroup $\langle \sfb_n,\sfs \rangle$ preserves each $n$-gon face, and $\sfa$ exchanges them.  \end{fact}

\begin{figure}
\begin{center}
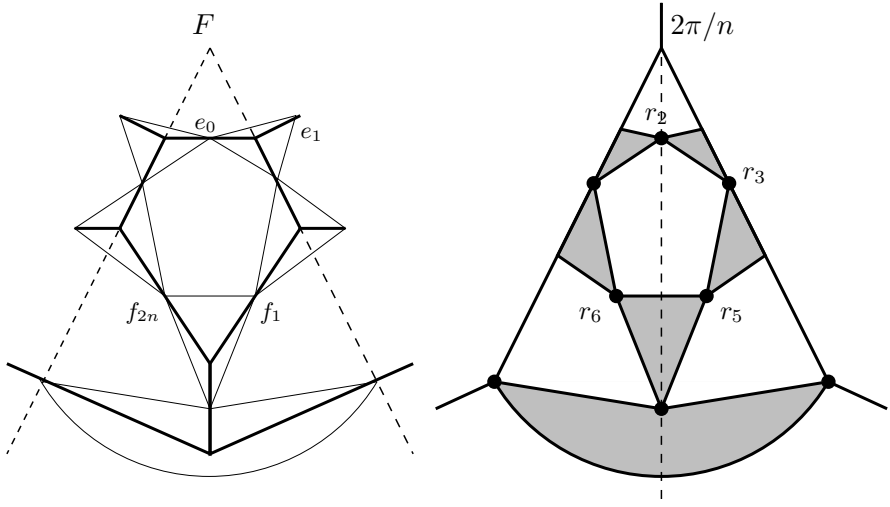
\end{center}
\caption{A fundamental domain for the action of $\langle \sfb_n\rangle$ on $\call(n)$, and the corresponding sub-polyhedron $\calo(n)$ of $\calp(n)$.}
\label{Lnsymm}
\end{figure}

\begin{proof}  Since $n \neq 5$, $\call(n)$ has exactly two $n$-gon faces $F$ and $F'$.  Let $e_0,e_1,\hdots,e_{n-1}$ be a cyclic ordering of the edges of $F$; ie, for each $i$, $e_i$ shares a vertex with $e_{i+1}$, where $i+1$ is taken modulo $n$.  The union of $F$ with the pentagonal faces of $\call(n)$ that abut it is a disk $D$ embedded in $\partial \call(n)$, with boundary consisting of $2n$ edges that can be cyclically ordered $f_1,f_2,\hdots,f_{2n}$ as follows: for $0 \leq i < n$, let $F_i$ be the pentagonal face of $\call(n)$ containing $e_i$ and let $f_{2i+1} \subset F_i \cap \partial D$ and $f_{2(i+1)} \subset F_{i +1} \cap \partial D$ be the unique pair of edges that share a vertex (with $i+1$ taken modulo $n$).

We now let $\sfb_n$ be the rotational symmetry of $F$ taking $e_i$ to $e_{i+1}$ for each $i$, and take $\sfs$ to be the reflection of $F$ preserving $e_0$ and exchanging $e_i$ with $e_{n-i}$ for $0 < i<n$.  It is easy to see that these extend to a rotation and reflection of $\call(n)$, respectively, yielding the subgroup $\langle \sfb_n,\sfs \rangle$ described above (we refer to the extensions by the same name).  

There is a symmetry $\sfa$ of the embedded circle $f_1 \cup f_2 \cup \hdots f_{2n}$ that preserves $f_1$ and $f_{n+1}$, exchanging endpoints of each, and exchanges $f_i$ with $f_{2n+2-i}$ for $1<i\leq n$.  This extends to a rotational symmetry of $\call(n)$ taking $F$ to $F'$.  In particular, for $0 \leq i < n$, we can take $F_i'$ to be the pentagonal face adjacent to $F'$ that contains $f_{2i+1}$ and $f_{2(i+1)}$.  Then $\sfa$ takes $F_i$ to $F'_{n-i}$.

The relations on $\sfb_n$, $\sfs$, and $\sfa$ follow by considering their actions on $F$.  Since every automorphism of $\call$ either exchanges $F$ and $F'$ or preserves each, there is a map to $\mathbb{Z}/2\mathbb{Z} = \{\pm 1\}$ taking such an element to $-1$ or $1$, respectively.  The subgroup $\langle \sfb_n,\sfs\rangle$ is contained in the kernel of this map; since it is the entire symmetry group of $F$, it is the entire kernel.  Hence the entire symmetry group of $\call(n)$ is generated by $\langle \sfb_n,\sfs\rangle$ and $\sfa$, which maps to $-1$.
\end{proof}

A fundamental domain for the action on $\call(n)$ of the cyclic group $\langle \sfb_n \rangle$ is depicted on the left-hand side of Figure \ref{Lnsymm}, cut out by the dotted line segments.  These should be interpreted as meeting at the point at infinity, in addition to the center of $F$.  The segment that runs through the edge joining endpoints of $e_0$ and $f_{2n}$ is fixed by the reflection $\sfs\sfb_n$, and the other is fixed by $\sfb_n\sfs$.  

Recall that by Lemma \ref{symmetric crush}, each symmetry of $\call(n)$ determines a symmetry of the right-angled ideal polyhedron $\calp(n)$ obtained by truncating vertices of $\call(n)$.  In particular, $\sfs\sfb_n$ and $\sfb_n\sfs$ determine reflective symmetries of $\calp(n)$.  Cutting along the mirrors of these reflections yields the polyhedron $\calo(n)$ pictured on the right-hand side of the figure.  The three edges with ``free'' ends should again be interpreted as meeting at the point at infinity.  The darkened vertices of $\calo(n)$ are ideal; the remaining vertices, each the midpoint of an edge of $\calp(n)$, are not.

The intersection of the mirror of $\sfs$ with $\partial \calo(n)$ is the dotted axis on the right-hand side of Figure \ref{Lnsymm}.  Clearly, $\sfs$ restricts to an isometry of $\calo(n)$.  Although $\sfa$ does not preserve $\calo(n)$, it does preserve the sub-polyhedron, obtained by cutting along the mirror of $\sfs$, that contains the ideal vertex labeled $r_5$.  Indeed, it acts on this polyhedron as a $180$-degree rotation fixing $r_5$ and the midpoint of the edge labeled $2\pi/n$, exchanging each of $r_3$ and $r_2$ with an unlabeled ideal vertex.

Since $\calp(n)$ is right-angled, each edge of $\calo(n)$ that is contained in one of $\calp(n)$ has dihedral angle $\pi/2$.  Since the mirrors of $\sfs\sfb_n$ and $\sfb_n\sfs$ meet each edge of $\calp(n)$ transversely, each edge of $\calo(n)$ that is the intersection of $\partial\calp(n)$ with a mirror of one of these reflections has dihedral angle $\pi/2$ as well.  Thus the only edge of $\calo(n)$ with a dihedral angle different than $\pi/2$ is the intersection of the mirrors of $\sfs\sfb_n$ and $\sfb_n\sfs$, labeled $2\pi/n$ at the top of the figure.  That this is the dihedral angle follows from the fact that the product of these reflections is the rotation $(\sfb_n)^2$, through an angle of $2\cdot 2\pi/n$.

Each symmetry of $\call(n)$, $n \neq 5$, exchanges edges enclosed by clasps of $L(n)$; hence the corresponding isometry of $\calp(n)$ induces one of $M(n) = S^3-L(n)$.  Since $\calo(n)$ is a fundamental domain for the action of the rotation group $\langle \sfb_n\rangle$ on $\call(n)$, Lemma \ref{LAD decomp} implies $\calo(n) \cup \overline{\calo}(n)$ is a fundamental domain for the action on $\mathbb{H}^3$ of the orbifold fundamental group of $O(n) = M(n)/\langle\sfb_n\rangle$.  Here $\overline{\calo}(n) \doteq \sfd_1(\calo(n))$, where $\sfd_1$ is the reflection through the white face of $\calo(n)$ whose sole ideal vertex is $r_2$.  Using the further symmetries $\sfa$ and $\sfs$ of $\calp(n)$, we thus obtain the lemma below.

\begin{lemma}\label{O(n) group}  Let $\sfd_2$ be the reflection through the white face of $\calo(n)$ with ideal vertices $r_2$, $r_3$, $\sfs(r_3)$, $r_5$, $r_6$, and let $\sfc$ be the parabolic isometry fixing $r_3$ and taking $r_2$ to $r_5$.  Then $O(n)$ is isometric to $\mathbb{H}^3/\Gamma(n)$, where $$
\Gamma(n) = \langle\ \sfd_1\sfd_2, \sfd_1\sfd_2^{\sfa}, \sfd_1\sfd_2^{\sfs\sfa}, \sfd_1\sfd_1^{\sfa}, \sfb_n, \sfc, \sfc^{\sfa}, \sfc^{\sfs}, \sfb_n^{\sfd_1}, \sfc^{\sfd_1}, \sfc^{\sfd_1\sfa}, \sfc^{\sfd_1\sfs}\ \rangle.  $$
Furthermore, the isometry of $O(n)$ visible on the right-hand side of Figure \ref{fig:L4 and quot} as reflection  through the projection plane is induced by $\sfd_1$.  \end{lemma}

Let $L$ be the link in $S^3$ that is the union of the fixed locus of $O(n)$ with the other components pictured on the right-hand side of Figure \ref{fig:L4 and quot}.  Then $O(n)$ is obtained from $S^3-L$ by $(n,0)$-Dehn filling on the added component, where the meridian here is chosen to lie in the projection plane and the longitude bounds a $3$-punctured disk.  Because the singular locus of $O(n)$ is the image of the edge $e$ of $\calo(n)$ with dihedral angle $2\pi/n$, $S^3 - L$ is obtained from $\calo(n) - e$ by the restriction of the face pairings described in Lemma \ref{O(n) group}.  Thus Poincar\'e's polyhedron theorem implies:

\begin{lemma}\label{L group}  Let $\calo$ be the all-right polyhedron in $\mathbb{H}^3$ homeomorphic to $O(n) - e$, and let $\sfa$, $\sfb$, $\sfc$, $\sfd_1$ and $\sfd_2$ have the same combinatorial descriptions as the correspondingly-named isometries determined by $\calo(n)$.  Let
$$ \Gamma_L = \langle  \sfd_1\sfd_2, \sfd_1\sfd_2^{\sfa}, \sfd_1\sfd_2^{\sfs\sfa}, \sfd_1\sfd_1^{\sfa}, \sfb, \sfc, \sfc^{\sfa}, \sfc^{\sfs}, \sfb^{\sfd_1}, \sfc^{\sfd_1}, \sfc^{\sfd_1\sfa}, \sfc^{\sfd_1\sfs} \rangle.  $$
Then $S^3 - L$ is homeomorphic to $\mathbb{H}^3/\Gamma_L$.  \end{lemma}

The only aspect of this lemma that requires comment is that Andreev's theorem implies that there is a right-angled polyhedron $\calo$ with the requisite combinatorial description.  An ideal vertex of $\calo$ replaces the edge of $\calo(n)$ with dihedral angle $2\pi/n$.  Thus $\sfb$ is parabolic, rather than elliptic like $\sfb_n$.

Denote by $r_7$ the ideal vertex of $\calo$ fixed by $\sfb$; that is, $r_7$ replaces the edge of $\calo(n)$ with dihedral angle $2\pi/n$.  The polyhedron obtained by cutting along the mirror of $\sfs$, that has $r_5$ as an ideal vertex, has $180$-degree rotational symmetry $\sfa$ fixing $r_5$ and $r_7$.  Therefore a single geodesic plane contains the ideal vertices $r_2$, $r_5$, $r_7$, and $\sfa(r_2)$.  Let $\calq_0$ be the polyhedron with $r_3$ as an ideal vertex that is obtained by cutting along this plane.

\begin{figure}
\begin{center}
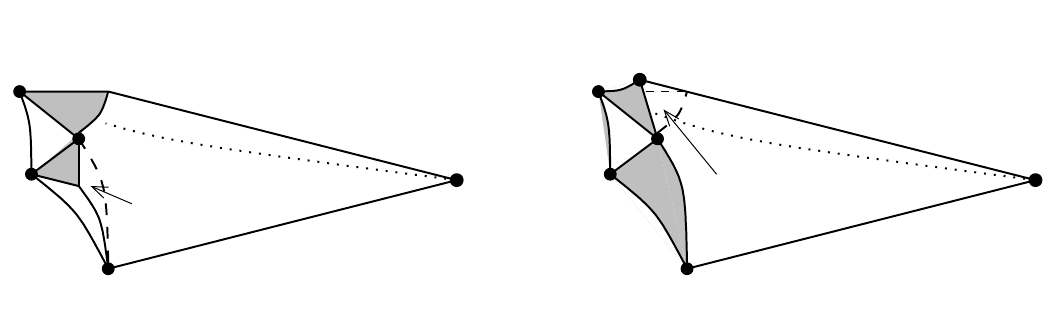
\end{center}
\caption{$\calq_0$ and $\calq$}
\label{calq_0}
\end{figure}

An ideal polyhedron $\calq$ may be obtained from $\calq_0$ as follows.  The geodesic plane $[r_5,r_3,\sfa(r_2)]$ containing $r_5$, $r_3$, and $\sfa(r_2)$ cuts off a tetrahedron $\calt$, with a finite vertex opposite this plane, from the remainder of $\calq_0$.  Let $\calq = (\overline{\calq_0 - \calt})\cup \sfc^{-1}(\calt)$.  Since all edges abutting each finite vertex of $\calq_0$ have dihedral angle $\pi/2$, the finite vertices of $\calq_0$, which are identified in $\calq$, lie in the interior of an edge of $\calq$.  We have depicted $\calq_0$ and $\calq$ on the left- and right-hand sides of Figure \ref{calq_0}, respectively, coloring black the face of $\calq$ in $[r_5,r_3,\sfa(r_2)]$ and its image under $\sfc^{-1}$.

The lemma below follows from Poincar\'e's polyhedron theorem and the descriptions from Lemma \ref{L group} of face pairing isometries on $\calo \cup \sfd_1(\calo)$ yielding $S^3 - L$.

\begin{lemma}\label{commensurator}  Let $\Gamma = \langle \sfa, \sfc,\sfd_1,\sfd_2,\sfd_3 \doteq \sfa\sfs\sfa \rangle$ be generated by face pairings for $\calq$.  Then $\mathbb{H}^3/\Gamma$ is a three-cusped hyperbolic $3$-orbifold, and $\Gamma_L \lhd \Gamma$ with index $8$.  \end{lemma}

The isometry $\sfd_3$ defined in Lemma \ref{commensurator} acts as reflection in the face of $\calq$ containing $r_7$, $\sfa(r_2)$, $r_3$, and $\sfc^{-1}\sfa(r_2)$, since $\sfa$ takes this face into the mirror of $\sfs$.  That the other generators act as face pairings follows from previous observations.  The index computation uses the fact that $\calo$ is the union of $4$ isometric copies of $\calq$; namely, $\calo = \calq \cup \sfa(\calq) \cup \sfs(\calq \cup \sfa(\calq))$.  In verifying that each generator for $\Gamma_L$ lies in $\Gamma$, it is helpful to note that $\sfb = \sfd_1 \sfs \in \Gamma$.

The key result in the proof of Theorem \ref{Lobell thm} is the proposition below.

\begin{prop}\label{prop:comm}  $\Gamma$ is its own commensurator.  \end{prop}

We defer the proof of Proposition \ref{prop:comm} for now, and first apply it.

\begin{proof}[Proof of Theorem \ref{Lobell thm}]   Since the orbifold fundamental group $\Gamma(n)$ of $O(n)$ contains the elliptic element $\sfb_n$, with order $n$, its invariant trace field $k\Gamma(n)$ contains the trace of $(\sfb_n)^2$ and thus $\mathbb{Q}(\cos (2\pi i\frac{2}{n}))$ (cf.~\cite[\S 3.3]{MaR} for the definition and properties of the invariant trace field).  This is a degree-two subfield of the cyclotomic field $\mathbb{Q}(\zeta_k)$, where $k = n$ if $n$ is odd and $k=n/2$ otherwise.  Thus $\liminf_{n \to \infty} [k\Gamma(n):\mathbb{Q}]$ is infinite.  It follows that at most finitely many $O(n)$ belong to any one commensurability class.  Furthermore, at most finitely many are arithmetic, since non-compact arithmetic hyperbolic $3$-manifolds have quadratic invariant trace fields.

Throwing away the arithmetic $\Gamma(n)$, Margulis' theorem implies that $\mathrm{Comm}(\Gamma(n))$ is a finite extension of $\Gamma(n)$ for the remaining $n$.  We remarked above Lemma \ref{O(n) group} that each symmetry of $\call(n)$ determines an isometry of $M(n) = S^3 - L(n)$.  In particular, there are isometries determined by $\sfa$ and $\sfs$, and since $\langle \sfb_n \rangle \lhd \langle \sfa,\sfb_n,\sfs \rangle$, these generate a group of isometries of $O(n)= M(n)/\langle \sfb_n \rangle$ with order $4$.  By Lemma \ref{O(n) group}, $\sfd_1$ determines an additional isometry of $O(n)$, that can easily be seen to commute with $\langle \sfa,\sfs\rangle$.  Thus $\mathrm{Comm}(\Gamma(n))$ contains the degree-$8$ extension $\langle \Gamma(n),\sfa,\sfs,\sfd_1\rangle$ of $\Gamma(n)$.

As the right-hand side of Figure \ref{fig:L4 and quot} makes clear, $O(n)$ is obtained from $S^3-L$ by $(n,0)$-Dehn filling on a fixed component.  Therefore the hyperbolic Dehn surgery theorem implies that the $O(n)$ converge geometrically to the hyperbolic structure on $S^3 - L$, and in particular, their volumes approach its from below.  (See eg.~\cite[\S E.5]{BenPet} for background on the hyperbolic Dehn surgery theorem.)  Furthermore, the explicit descriptions above imply that the $\Gamma(n)$ converge algebraically to $\Gamma_L$, and the $\langle \Gamma(n), \sfa,\sfs,\sfd_1\rangle$ to $\Gamma$.

If on an infinite subsequence, $\langle \Gamma(n),\sfa,\sfs,\sfd_1\rangle$ were contained in  $\mathrm{Comm}(\Gamma(n))$ properly, then a further subsequence of the $\mathrm{Comm}(\Gamma(n))$ would converge to a discrete group $\Gamma_0$ with covolume a proper fraction of that of $\Gamma$.  This follows from the fact that the Chabauty topology on discrete subgroups of $\mathrm{PSL}_2(\mathbb{C})$ with bounded covolume is compact, see eg.~\cite[Corollary E.1.7]{BenPet}.  In this case, since $\langle \Gamma(n), \sfa,\sfs,\sfd_1\rangle \to \Gamma$ and limits are unique in this topology (see eg.~\cite[Lemma E.1.1]{BenPet}), we would have $\Gamma < \Gamma_0$ properly, contradicting Proposition \ref{prop:comm}.  Thus for all but finitely many $n$, $\mathrm{Comm}(\Gamma(n)) = \langle\Gamma(n),\sfa,\sfs,\sfd_1\rangle$.

Fixing a horosphere $\calh$ centered at the ideal vertex $r_3$ of $\calo(n)$, a fundamental domain for the action on $\calh$ of its stabilizer in $\mathrm{Comm}(\Gamma(n))$ is thus the rectangle $\calo(n) \cap \calh$.  Two parallel sides of this rectangle are given by the intersection of $\calh$ with the white sides of $\calo(n)$ containing $r_3$.  One of these, contained in the side with ideal vertices $r_2$, $r_3$, $\sfs(r_3)$, $r_5$, and $r_6$, is stabilized by the reflection $\sfd_2 \in \Gamma(n)$ defined in Lemma \ref{O(n) group}.  The other is stabilized by the reflection $\sfd_3 \doteq \sfa\sfs\sfa \in \mathrm{Comm}(\Gamma(n))$, defined in analogy with the identically-named element of $\Gamma$ from Lemma \ref{commensurator}.  For the other pair of parallel sides of this rectangle, the parabolic $\sfc$ fixing $r_3$ acts by translation taking one to the other.

The stabilizer of $\calh$ in $\mathrm{Comm}(\Gamma(n))$ is thus $\langle \sfc, \sfd_2,\sfd_3\rangle$.  If $\Gamma'$ were a reflection group commensurable with $\Gamma(n)$, then $\mathrm{Stab}_{\Gamma'}(r_3)$ would be a reflection group contained in $\langle \sfc,\sfd_2,\sfd_3\rangle$, acting on $\calh$ with finite coarea.  But since $\sfc$ translates parallel to the lines fixed by $\sfd_2$ and $\sfd_3$, every reflection in $\langle \sfc,\sfd_2,\sfd_3\rangle$ fixes a line parallel to the lines fixed by $\sfd_2$ and $\sfd_3$.   Hence no reflection subgroup of $\langle \sfc,\sfd_2,\sfd_3\rangle$ acts on $\calh$ with finite coarea. Therefore $\mathrm{Comm}(\Gamma(n))$ is not commensurable with a reflection group.
\end{proof}

Proving Proposition \ref{prop:comm} requires an explicit description of $\Gamma$.  This will follow from the lemma below, which describes an embedding of $\calq$ in the upper half-space model for $\mathbb{H}^3$.

\begin{lemma}\label{first calq embed}  There is an isometric embedding of $\calq$ in $\mathbb{H}^3$ determined by the following ideal vertices:  $r_2 = -1+i$, $r_3 = 0$, $r_5 = (\sqrt{3}+i)/2$, $r_7 = \infty$.  \end{lemma}

\begin{proof}  Our description of $\calq$ includes the following facts:  its edge joining the ideal vertex $r_7$ to $\sfc^{-1}\sfa(r_2)$ has a dihedral angle of $\pi/2$, and there are two quadrilateral faces with ideal vertices $r_7$, $\sfa(r_2)$, $r_5$, $r_2$ and $r_7$, $\sfa(r_2)$, $r_3$, $\sfc^{-1}\sfa(r_2)$, respectively.  We will choose an embedding of $\calq$ that sends the latter face into the geodesic plane of $\mathbb{H}^3$ with ideal boundary $\mathbb{R}\cup\{\infty\}$, taking $r_3$ to $0$ and $r_7$ to $\infty$  in particular.  

\begin{figure}
\begin{center}
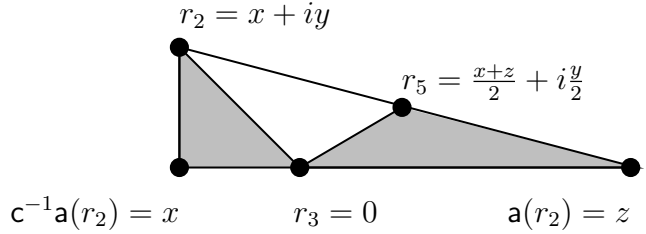
\end{center}
\caption{An embedding of $\calq$ in $\mathbb{H}^3$.}
\label{calq_embed}
\end{figure}

We have pictured such an embedding in Figure \ref{calq_embed}.  The ideal vertices $\sfc^{-1}\sfa(r_2)$ and $\sfa(r_2)$ go to points $x$ and $z$, respectively, in $\mathbb{R}$ on either side of $r_3 = 0$.  We take $x < 0$ and $z>0$.  Since the edge joining $r_7$ to $\sfc^{-1}\sfa(r_2)$ has dihedral angle $\pi/2$, the image of $r_2$ is of the form $x+iy$ for some $y \in \mathbb{R}$.  We may assume $y >0$, by reflecting through $\mathbb{R}$ if necessary.  The final ideal vertex $r_5$ lies somewhere on the line segment joining $r_2$ with $\sfa(r_2)$, since it is in the ideal boundary of a plane containing $r_2$, $r_7$, and $\sfa(r_2)$.  Its coordinates are determined by the fact that $\sfa$ preserves this plane, fixing $r_5$ and $r_7$.  

We have darkened the triangles in $\mathbb{C}$ that lie under the dark faces of $\calq$ after the  embedding described above.  The parabolic isometry $\sfc$ takes one to the other, fixing $r_3$, thus it is of the form $\left(\begin{smallmatrix} 1 & 0 \\ w & 1 \end{smallmatrix}\right)$ for some $w \in \mathbb{C}$.  Using the fact that $\sfc$ takes $\sfc^{-1}\sfa(r_2)=x$ to $\sfa(r_2)=z$, a computation implies $w = (x-z)/xz$.  Another computation, using the fact that $\sfc(r_2) = r_5$, determines $z = -x(\sqrt{3}+1)$.

We are free to choose $x < 0$, since one choice may be changed to another by applying a hyperbolic isometry fixing $0$ and $\infty$ to $\calq$.  Choosing $x = -1$ yields: \begin{align*} 
  & \sfc^{-1}\sfa(r_2) = x = -1 && r_2 = x+iy = -1+i  && \sfa(r_2) = z = \sqrt{3}+1 && r_5 = \frac{\sqrt{3} + i}{2}.  \end{align*}
This is the embedding described in the statement.  \end{proof}

A few additional parabolic fixed points that will be useful below we name as follows: let $r_1 = \sfd_3(r_2)$, $r_4 = \sfd_3(r_5)$, and $r_6 = \sfd_3^{\sfa}(r_5)$.  Note that $r_2$, $r_4$, $r_5$, and $r_6$ are each $\Gamma$-equivalent to $r_1$.

In proving Proposition \ref{prop:comm}, it will be convenient to use a different embedding of $\calq$ than that described in Lemma \ref{first calq embed} above.  Let us apply a M\"obius transformation taking $r_1$, $r_2$, and $r_3$ to $0$, $1$, and $\infty$, respectively.  Such a map is given by $z \mapsto \frac{1+i}{2} + i/z$.  This takes the other ideal vertices to:  \begin{align*}
  & r_4 = i\frac{1+\sqrt{3}}{2} && r_5 = 1+i\frac{1+\sqrt{3}}{2} && r_6 = 1+ i\frac{3+\sqrt{3}}{6} && r_7 = \frac{1+i}{2}  \end{align*}
The representation of $\Gamma$ determined by the embedding described above is related to that determined by the embedding of Lemma \ref{first calq embed} by conjugation by 
$$\left(\begin{smallmatrix} -i\frac{\sqrt{2}}{2} & \frac{\sqrt{2}}{2}(1-i) \\ -\frac{\sqrt{2}}{2}(1+i) & 0 \end{smallmatrix}\right).$$

Since $\calq$ is a fundamental domain for $\Gamma$, each cusp of $\mathbb{H}^3/\Gamma$ corresponds to a point on $\partial \mathbb{H}^3$ that is $\Gamma$-equivalent to an ideal vertex of $\calq$.  Inspection of the face pairings of Lemma \ref{commensurator} thus reveals that $\mathbb{H}^3/\Gamma$ has exactly three cusps.  We let $c_1$ correspond to the points of $\Gamma \cdot r_1$, $c_2$ to $\Gamma \cdot r_7$, and $c_3$ to $\Gamma \cdot r_3$.

Our explicit description of $\calq$ allows computation of the invariant trace field and cusp parameters.    This implies:

\begin{lemma}  $\Gamma$ is non-arithmetic.  The cusps $c_1$ and $c_2$ are commensurable to each other and are not commensurable to $c_3$.
\end{lemma}

\begin{proof}  An explicit description of generators for $\Gamma$, as may be obtained from Lemma \ref{first calq embed}, enables direct computation of the invariant trace field (see \cite[\S 3.5]{MaR}).  Performing this calculation, we find that $\Gamma$ has trace field $\mathbb{Q}(i,\sqrt{3})$.  Alternatively, the link $L$ may be entered into the computer program Snappea, and the resulting triangulation data into Snap, yielding the same description (see \cite{CGHN}).  Since every non-compact arithmetic hyperbolic $3$-manifold has an imaginary quadratic invariant trace field, $\Gamma$ is not arithmetic.

Using the embedding described in Lemma \ref{first calq embed}, we find that an index-$8$ subgroup of $\mathrm{Stab}_{\Gamma}(\infty)$ is generated by $z \mapsto z + 2(2+\sqrt{3})$ and $z \mapsto z+ 2i$; thus the parameter of the associated cusp $c_2$ is $\mathrm{PGL}_2(\mathbb{Q})$-equivalent to $i(2+\sqrt{3})$ (cf.~\cite[\S 4.3]{CD}).  After re-embedding as above, the stabilizer of $\infty$ corresponds to the cusp $c_3$.  An index-$2$ subgroup of this lattice is generated by $\sfc \co z \mapsto z+i\frac{1+\sqrt{3}}{2}$ and the product of reflections $\sfd_2\sfd_3\co z \mapsto z + 1$.  Thus the parameter of $c_3$ is $\mathrm{PGL}_2(\mathbb{Q})$-equivalent to $i(1+\sqrt{3})$.  A similar computation reveals that $c_1$ has the same parameter as $c_2$.  Since the complex modulus is a complete commensurability invariant for lattices in $\mathbb{C}^2$, and $i(1+\sqrt{3})$ is not $\mathrm{PGL}_2(\mathbb{Q})$-equivalent to $i(2+\sqrt{3})$, the lemma follows.  \end{proof}

From Margulis' theorem, we immediately obtain:

\begin{cor} $\text{Comm}(\Gamma)$ is a finite extension of $\Gamma$, and the minimal orbifold $O \doteq \mathbb{H}^3/\text{Comm}(\Gamma)$ has either two or three cusps.  \end{cor}

In particular, if $O$ has two cusps then $c_1$ and $c_2$ are identified by the covering map $\mathbb{H}^3/\Gamma \rightarrow O$.  We have used the algorithm of Goodman--Heard--Hodgson \cite{GHH} to compute $\text{Comm}(\Gamma)$.  Recall that we introduced the setting for this algorithm in Section \ref{sec:One cusp} between the statements of Propositions \ref{OP comm} and \ref{tiling}.

Let 
\begin{align*}
{\bf v}_1& \ = \ \left( -2, \,2, \,-1, \,3 \right)^T \\
{\bf v}_7& \ = \ \left( 0,  \,0,  \,9-4\sqrt{3},  \,9-4\sqrt{3} \right)^T \\
{\bf v}_3& \ = \ \left( 0,  \,0,  \,-3,  \,3 \right)^T. 
\end{align*}

These vectors are chosen so that there is an isometry $\Phi$ from the upper half space model to the hyperboloid model which takes the parabolic fixed point $r_i$ to the center of the horosphere $H_{{\bf v}_i}$ when $i=1,3,7$.  Under $\Phi$, the isometries ${\sf a,b,c}, {\sf d}_1,  {\sf d}_2,  {\sf d}_3$ correspond to the matrices ${\sf A, B, C}, {\sf D}_1,  {\sf D}_2,  {\sf D}_3 \in \text{O}_0(3,1)$ listed below.
\begin{align*}
 {\sf A}&\  = \ \left(
\begin{array}{cccc}
 -1 & 0 & -1/2 & 1/2\\
 0 & -1 & \sqrt{3}/2 & -\sqrt{3}/2 \\
 -1/2 &\sqrt{3}/2 & 1/2 & 1/2 \\
 -1/2 & \sqrt{3}/2 & -1/2 & 3/2
\end{array}
\right) & 
{\sf D}_1& \ = \ \left(
\begin{array}{cccc}
 1 & 0 & 0 & 0 \\
 0 & -1 & -1 & 1 \\
 0 & -1 & \frac{1}{2} & \frac{1}{2} \\
 0 & -1 & -\frac{1}{2} & \frac{3}{2}
\end{array}
\right) 
\end{align*}\begin{align*}
{\sf B}& \ = \ \left(
\begin{array}{cccc}
 1 & 0 & -1 & 1 \\
 0 & 1 & 0 & 0 \\
 1 & 0 & 1/2 & 1/2\\
 1 & 0 & -1/2 & 3/2
\end{array}
\right) &
{\sf D}_2& \ = \ \left(
\begin{array}{cccc}
 -1 & 0 & 2 & 2 \\
 0 & 1 & 0 & 0 \\
 2 & 0 & -1 & -2 \\
 -2 & 0 & 2 & 3
\end{array}
\right)
\end{align*}\begin{align*}
{\sf C}&\ = \ \left(
\begin{array}{cccc}
 1 & 0 & 0 & 0 \\
 0 & 1 & -1-\sqrt{3} & -1-\sqrt{3} \\
 0 & 1+\sqrt{3} & -1-\sqrt{3} & -2-\sqrt{3} \\
 0 & -1-\sqrt{3} & 2+\sqrt{3} & 3+\sqrt{3}
\end{array}
\right) &
{\sf D}_3& \ = \ \left(
\begin{array}{cccc}
 -1 & 0 & 0 & 0 \\
 0 & 1 & 0 & 0 \\
 0 & 0 & 1 & 0 \\
 0 & 0 & 0 & 1
\end{array}
\right) 
\end{align*}

Thus, $\Phi$ allows us to also think of $\Gamma$ as a subgroup of $\text{O}_0(3,1)$.   Each ${\bf v}_i$ is a horospherical vector for the cusp $c_i$ of $\mathbb{H}^3/\Gamma$ so $\{ {\bf v}_1, {\bf v}_3, {\bf v}_7 \}$ determines a $\Gamma$-invariant set $V$ as above.  We have $\{ {\bf v}_i \}_1^7$ given by ${\bf v}_i = \Phi(r_i)$
and these vectors may be calculated explicitly by applying appropriate isometries from $\Gamma$.  We have that ${\bf v}_i$ is the $i^\text{th}$ column of the matrix
\[  \left(
\begin{array}{ccccccc}
 -2 & 2 & 0 & -2 & 2 & 6 & 0 \\
 2 & 2 & 0 & -2 \sqrt{3} & -2 \sqrt{3} & -2 \sqrt{3} & 0 \\
 -1 & -1 & -3 & -3 & -3 & -1 & 9-4 \sqrt{3} \\
 3 & 3 & 3 & 5 & 5 & 7 & 9-4 \sqrt{3}
\end{array}
\right).\]

As discussed above,  we obtain all possibilities for canonical tilings associated to $\Gamma$ by using initial sets of the form $\{ {\bf v}_1, \beta {\bf v}_7, \gamma {\bf v}_3 \}$ where $\beta, \gamma \in \R^+$. We write $\calh(\beta, \gamma)$ to denote the set $\Gamma \cdot \{ {\bf v}_1, \beta {\bf v}_7, \gamma {\bf v}_3 \}$ and $\calt(\beta, \gamma)$ to denote the associated canonical tiling.

 Recall that $O=\mathbb{H}^3/\text{Comm}(\Gamma)$ has either 2 or 3 cusps.  If $O$ has 3 cusps then, for any pair $(\beta, \gamma)$, $\calh(\beta, \gamma)$ descends to cusp cross sections of $O$ and so $\text{Comm}(\Gamma) = \text{Sym}(\calt(\beta, \gamma))$.  If $O$ has 2 cusps then there is some ${\sf g} \in \text{Comm}(\Gamma)$ and $\beta_0$ with ${\sf g}({\bf v}_1) = \beta_0 {\bf v}_7$.  We have $\text{Comm}(\Gamma) = \text{Sym}(\calt(\beta_0, \gamma))$ for any $\gamma \in \R^+$.  Therefore, it suffices to compute the triangulations $\calt(\beta, 1)$ for $\beta \in \R^+$.  Either there exists a unique $\beta_0$ so that $\text{Sym}(\calt(\beta_0,1))$ contains an isometry taking ${\bf v}_1$ to $\beta_0 {\bf v}_7$ or there is no such $\beta$.  In the first case, $O$ has 2 cusps and $\text{Comm}(\Gamma) = \text{Sym}(\calt(\beta_0,1))$.  In the second case, $O$ has 3 cusps and $\text{Comm}(\Gamma) = \text{Sym}(\calt(\beta, 1))$ for every $\beta$. 
 
 \begin{lemma}  \label{lem:sym}
 $O$ has 3 cusps and $\text{Comm}(\Gamma) = \text{Sym}(\calt(\beta, 1))$ for every $\beta$. 
 \end{lemma}

\proof The proof follows by showing that there does not exist a unique $\beta$ so that $\text{Sym}(\calt(\beta_0,1))$ contains an isometry taking ${\bf v}_1$ to $\beta_0 {\bf v}_7$.  We first describe the canonical triangulations as $\beta$ decreases from $\infty$ to $0$.  The interval $(0,\infty)$ has a finite cell decomposition so that if two values for $\beta$ are chosen from the same cell then they determine the same canonical triangulation.  As $\beta$ moves to the boundary of a $1$-cell there is a pair of neighboring tiles $T_1$ and $T_2$ so that the tilt at their common face changes from positive to zero.  At the boundary value, these two tiles merge to form a tile in the new canonical triangulation.  The decomposition of $(0,\infty)$ and the associated tilings of $\mathbb{H}^3$ are described in Tables \ref{table:first two} -- \ref{table:last eleven}.   The triangulations $\calt(\beta,1)$ can be checked by repeatedly verifying the coplanar and positive tilt conditions on sets of $\Gamma$-generating tiles.  In the tables, we let $[p_1, \ldots , p_k]$ denote the convex hull in $\mathbb{H}^3$ of a collection $\{ p_1, \ldots , p_k\}\subset \bound \mathbb{H}^3$.  

\begin{table}%[ht]
\begin{tabular}{lclcl}  \\ \\ & &$\beta$ & & $\Gamma$-Generating Tiles \\
\hline \hline 
&& && $\calp_1=[\bv_3, \bv_4, \bv_5, {\sf A}(\bv_2)]$ \\
$\calt_1$&& $\beta > \frac{3}{11}(4+3\sqrt{3})$ && $\calp_2=[\bv_3, \bv_5, {\sf A}(\bv_2), {\sf A}(\bv_3)]$\\
&& && $\calp_3=[\bv_3, \bv_5, {\sf CA}(\bv_2), {\sf CA}(\bv_3)]$\\
&& && $\calp_4=[\bv_3, \bv_4, \bv_5, {\sf CA}(v_2)]$\\
&& && $\calp_5=[\bv_4, \bv_5, {\sf CA}(v_2), {\sf C}(\bv_7)]$\\

\hline && && $\calp_1=[\bv_3, \bv_4, \bv_5, {\sf A}(\bv_2)]$ \\
$\calt_2$&& $\beta=\frac{3}{11}(4+3\sqrt{3})\sim 2.51$ && $\calp_2=[\bv_3, \bv_5, {\sf A}(\bv_2), {\sf A}(\bv_3)]$\\
&& && $\calp_3=[\bv_3, \bv_5, {\sf CA}(\bv_2), {\sf CA}(\bv_3)]$\\
&& && $\calp_4=[\bv_3, \bv_4, \bv_5, {\sf CA}(v_2), {\sf C}(\bv_7)]$\\

\hline
\end{tabular}
\caption{The data that determine the first two canonical tilings.}
\label{table:first two}
\end{table}

\begin{table}
\centering
\begin{tabular}{lclcl}  \\ \\ & &$\beta$ & & $\Gamma$-Generating Tiles \\
\hline \hline 

&&  && $ \calp_1=[\bv_3, \bv_4, \bv_5, {\sf C}(\bv_7)]$ \\
$\calt_3$&& $\frac{1}{22}(21+13\sqrt{3})<\beta<\frac{3}{11}(4+3\sqrt{3})$&& $\calp_2=[\bv_3, \bv_4, \bv_5, {\sf A}(\bv_2)]$\\
&& && $\calp_3=[\bv_3, \bv_5, {\sf A}(\bv_2), {\sf A}(\bv_3)]$\\
&& && $\calp_4=[\bv_3,\bv_5, {\sf C}(\bv7), {\sf CA}(\bv_2)]$\\
&& && $\calp_5=[\bv_3, \bv_5, {\sf CA}(\bv_2), {\sf CA}(\bv_3)]$\\

\hline && && $\calp_1=[\bv_3, \bv_4, \bv_5, {\sf C}(\bv_7)]$ \\
$\calt_4$&& $\beta =\frac{1}{22}(21+13\sqrt{3})\sim 1.978$&&$\calp_2=[\bv_3, \bv_4, \bv_5, {\sf A}(\bv_2)]$\\
&& && $\calp_3=[\bv_3, \bv_5, {\sf A}(\bv_2), {\sf A}(\bv_3)]$\\
&& && $\calp_4=[\bv_3,\bv_5, {\sf C}(\bv7), {\sf CA}(\bv_2), {\sf CA}(\bv_3)]$ \\ 

\hline &&  && $ \calp_1=[\bv_3, \bv_4, \bv_5, {\sf C}(\bv_7)]$ \\
$\calt_5$&& $\frac{1}{11}(9+4\sqrt{3})<\beta<\frac{1}{22}(21+13\sqrt{3})$&& $ \calp_2=[\bv_3, \bv_4, \bv_5, {\sf A}(\bv_2)]$\\
&& && $ \calp_3=[\bv_3,\bv_5, {\sf A}(\bv_2), {\sf A}(\bv_3)]$\\
&& && $ \calp_4=[\bv_3,\bv_7, {\sf A}(\bv_2), {\sf A}(\bv_3)]$\\

\hline && && $\calp_1=[\bv_3, \bv_4, \bv_5, {\sf C}(\bv_7)]$ \\
$\calt_6$&& $\beta = \frac{1}{11}(9+4\sqrt{3})\sim 1.45$&&$ \calp_2=[\bv_3, \bv_4, \bv_5, {\sf A}(\bv_2)]$\\
&& &&$ \calp_3=[\bv_3,\bv_5, \bv_7, {\sf A}(\bv_2), {\sf A}(\bv_3)]$\\ 

\hline && && $\calp_1=[\bv_1, \bv_2, \bv_3, \bv_7]$ \\
$\calt_7$&&$\frac{1}{121}(72+43\sqrt{3})<\beta<\frac{1}{11}(9+4\sqrt{3})$&& $\calp_2=[\bv_2, \bv_3, \bv_5, \bv_7]$\\
&& &&$\calp_3=[\bv_3,\bv_5,\bv_7, {\sf A}(\bv_2)]$\\
&& &&$\calp_4=[\bv_3, \bv_4,\bv_5, {\sf A}(\bv_2)]$\\

\hline && && $\calp_1=[\bv_1, \bv_2, \bv_3, \bv_7]$ \\
$\calt_8$&& $\beta = \frac{1}{121}(72+43\sqrt{3}) \sim 1.21$&&$\calp_2=[\bv_2, \bv_3, \bv_5, \bv_7]$\\
&& &&$\calp_3=[\bv_3,\bv_4,\bv_5,\bv_7,{\sf A}(\bv_2)]$ \\  

\hline && && $\calp_1=[\bv_1, \bv_2, \bv_3, \bv_7]$ \\
$\calt_9$&&$(-21+13\sqrt{3})^{-1}<\beta<\frac{1}{121}(72+43\sqrt{3})$&& $\calp_2=[\bv_2, \bv_3, \bv_5, \bv_7]$\\
&& &&$\calp_3=[\bv_3,\bv_4,\bv_5,\bv_7]$\\
&& &&$\calp_4=[\bv_2, \bv_5, \bv_6, \bv_7]$\\

\hline && && $\calp_1=[\bv_1, \bv_2, \bv_3, \bv_7]$ \\
$\calt_{10}$&& $\beta = (-21+13\sqrt{3})^{-1} \sim 0.659$&&$\calp_2=[\bv_2, \bv_3, \bv_5, \bv_7]$\\
&& &&$\calp_3=[\bv_3,\bv_4,\bv_5,\bv_7]$\\
&& &&$\calp_4=[\bv_2, \bv_5, \bv_6, \bv_7,{\sf D}_2(\bv_7)]$\\ \hline 

\end{tabular}
\caption{More canonical tilings.}
\label{table:first ten}
 \end{table}
 
\begin{table}
\centering
\begin{tabular}{lclcl}  \\ \\ & &$\beta$ & & $\Gamma$-Generating Tiles \\
\hline \hline 
&& && $\calp_1=[\bv_3, \bv_4, \bv_5, {\sf C}(\bv_7)]$\\
$\calt_{11}$ && $\frac{1}{143}(48+25\sqrt{3})<\beta<(-21+13\sqrt{3})^{-1}$ && $\calp_2=[\bv_2, \bv_3, \bv_5, \bv_7]$ \\
&& && $\calp_3=[\bv_3, \bv_4, \bv_5, \bv_7]$\\
&& && $\calp_4=[\bv_5, \bv_6, \bv_7, {\sf D}_2(\bv_7)]$\\
&& && $\calp_5=[\bv_2, \bv_5, \bv_7,{\sf D}_2(\bv_7) ]$\\

\hline && && $\calp_1=[\bv_3, \bv_4, \bv_5, {\sf C}(\bv_7)]$\\
$\calt_{12}$ && $\beta = \frac{1}{143}(48+25\sqrt{3}) \sim 0.638$ && $\calp_2=[\bv_2, \bv_3, \bv_5, \bv_7,{\sf D}_2(\bv_7) ]$\\
&& && $\calp_3=[\bv_3, \bv_4, \bv_5, \bv_7]$\\
&& && $\calp_4=[\bv_5, \bv_6, \bv_7, {\sf D}_2(\bv_7)]$\\

\hline 
&& && $\calp_1=[\bv_3, \bv_4, \bv_5, {\sf C}(\bv_7)]$\\
$\calt_{13}$ && $\frac{1}{11}(6-\sqrt{3})<\beta<\frac{1}{143}(48+25\sqrt{3})$ && $\calp_2=[\bv_3,\bv_4,\bv_5,\bv_7]$\\
&& && $\calp_3=[\bv_3,\bv_5,\bv_7,{\sf D}_2(\bv_7)]$\\
&& && $\calp_4=[\bv_2, \bv_3,\bv_7,{\sf D}_2(\bv_7)]$\\
&& && $\calp_5=[\bv_5, \bv_6, \bv_7, {\sf D}_2(\bv_7)]$\\

\hline
&& && $\calp_1=[\bv_3, \bv_4, \bv_5, \bv_7, {\sf C}(\bv_7)]$\\
$\calt_{14}$ && $\beta = \frac{1}{11}(6-\sqrt{3}) \sim 0.39$ && $\calp_2=[\bv_3,\bv_5,\bv_7,{\sf D}_2(\bv_7)]$\\
&& && $\calp_3=[\bv_2, \bv_3,\bv_7,{\sf D}_2(\bv_7)]$\\
&& && $\calp_4=[\bv_5, \bv_6, \bv_7, {\sf D}_2(\bv_7)]$\\

\hline 
&& && $\calp_1=[\bv_3, \bv_5, \bv_7, {\sf C}(\bv_7)]$\\
$\calt_{15}$ && $\frac{1}{33}(3+5\sqrt{3})<\beta<\frac{1}{11}(6-\sqrt{3})$ && $\calp_2=[\bv_4, \bv_5, \bv_7, {\sf C}(\bv_7)]$\\
&& && $\calp_3=[\bv_3,\bv_5,\bv_7,{\sf D}_2(\bv_7)]$\\
&& && $\calp_4=[\bv_2, \bv_3,\bv_7,{\sf D}_2(\bv_7)]$\\
&& && $\calp_5=[\bv_5, \bv_6, \bv_7, {\sf D}_2(\bv_7)]$\\

\hline
&& && $\calp_1=[\bv_3, \bv_5, \bv_7, {\sf C}(\bv_7)]$\\
$\calt_{16}$ && $\beta=\frac{1}{33}(3+5\sqrt{3}) \sim 0.353$ && $\calp_2=[\bv_4, \bv_5, \bv_7, {\sf C}(\bv_7), {\sf D}_1^{\sf C}(\bv_7)]$\\
&& && $\calp_3=[\bv_3,\bv_5,\bv_7,{\sf D}_2(\bv_7)]$\\
&& && $\calp_4=[\bv_2, \bv_3,\bv_7,{\sf D}_2(\bv_7)]$\\
&& && $\calp_5=[\bv_5, \bv_6, \bv_7, {\sf D}_2(\bv_7)]$\\

\hline
&& && $\calp_1=[\bv_3, \bv_5, \bv_7, {\sf C}(\bv_7)]$\\
$\calt_{17}$ && $ \frac{1}{143}(24+7\sqrt{3}) <\beta<\frac{1}{33}(3+5\sqrt{3})$ && $\calp_2=[\bv_5, \bv_7, {\sf C}(\bv_7), {\sf D}_1^{\sf C}(\bv_7)]$\\
&& && $\calp_3=[\bv_3,\bv_5,\bv_7,{\sf D}_2(\bv_7)]$\\
&& && $\calp_4=[\bv_2, \bv_3,\bv_7,{\sf D}_2(\bv_7)]$\\
&& && $\calp_5=[\bv_5, \bv_6, \bv_7, {\sf D}_2(\bv_7)]$\\

\hline
&& && $\calp_1=[\bv_3, \bv_5, \bv_7, {\sf C}(\bv_7),{\sf D}_2(\bv_7), {\sf D}_2{\sf C}(\bv_7) ]$\\
$\calt_{18}$ && $\beta=\frac{1}{143}(24+7\sqrt{3}) \sim 0.252$ && $\calp_2=[\bv_5, \bv_7, {\sf AC}(\bv_7), {\sf D}_2(\bv_7)]$\\

\hline
&& && $\calp_1=[\bv_3, \bv_7, {\sf C}(\bv_7), {\sf D}_2(\bv_7), {\sf D}_2{\sf C}(\bv_7) ]$\\
$\calt_{19}$ && $\frac{1}{33}(6-\sqrt{3}) <\beta<\frac{1}{143}(24+7\sqrt{3})$ && $\calp_2=[\bv_5, \bv_7, {\sf C}(\bv_7), {\sf D}_2(\bv_7), {\sf D}_2{\sf C}(\bv_7) ]$\\
&& && $\calp_3=[\bv_5, \bv_7, {\sf AC}(\bv_7),{\sf D}_2(\bv_7) ]$\\

\hline
&& && $\calp_1=[\bv_3, \bv_7, {\sf C}(\bv_7), {\sf D}_2(\bv_7), {\sf D}_2{\sf C}(\bv_7) ]$\\
$\calt_{20}$ && $\beta=\frac{1}{33}(6-\sqrt{3})\sim 0.129 $ && $\calp_2=[\bv_5, \bv_7, {\sf C}(\bv_7), {\sf D}_2(\bv_7), {\sf D}_2{\sf C}(\bv_7), {\sf AC}(\bv_7) ]$\\

\hline
&& && $\calp_1=[\bv_3, \bv_7, {\sf C}(\bv_7), {\sf D}_2(\bv_7), {\sf D}_2{\sf C}(\bv_7) ]$\\
$\calt_{21}$ && $ \beta<\frac{1}{33}(6-\sqrt{3}) $ && $\calp_2=[ \bv_7, {\sf C}(\bv_7), {\sf D}_2(\bv_7), {\sf D}_2{\sf C}(\bv_7),{\sf AC}(\bv_7)]$\\
&& && $\calp_3=[\bv_5, {\sf C}(\bv_7), {\sf D}_2{\sf C}(\bv_7),{\sf AC}(\bv_7)]$\\

\hline 
\end{tabular} 
\caption{The remaining tilings.} 
\label{table:last eleven}
\end{table}

From our earlier observations, it remains only to check that there are no symmetries of the even numbered tilings that interchange vertices of $\Gamma.\bv_1$ with those of $\Gamma.\bv_7$.  The arguments for each of the  cases are very similar, we start with $\calt_2$ as a model case.  Recall that $\bv_2$, $\bv_4$, $\bv_5$, and $\bv_6$ are each $\Gamma$-equivalent to $\bv_1$.

Suppose there exists $\gamma \in \mathrm{Sym}(\calt_2)$ exchanging $\Gamma.\bv_1$ with $\Gamma.\bv_7$.  Then $\gamma(\calp_4)$ is a tile of $\calt_2$ with exactly five vertices.  $\calp_4$ is the unique generating tile with five vertices so there exists $\gamma' \in \Gamma$ with $\gamma' \gamma(\calp_4)=\calp_4$.  Since $\gamma' \in \Gamma$ it preserves the cusp classes of the vertices of tiles in $\calt_2$.  On the other hand, since $\gamma$ exists, the minimal orbifold must have exactly two cusps, hence $\gamma' \gamma$ must exchange the vertices of $\calp_4$ in $\Gamma.\bv_1$ with those in $\Gamma.\bv_7$.  But our explicit description implies that there are three of the former and only one of the latter, a contradiction.

The same sort of argument also works for the remaining even numbered triangulations with the exception of $\calt_{12}$ and $\calt_{14}$.  Consider the case of $\calt_{12}$.  Suppose there is $\gamma \in \mathrm{Sym}(\calt_{12})$ exchanging $\Gamma.\bv_1$ with $\Gamma.\bv_7$. Arguing as before, we have an element $\delta \in \text{Comm}(\Gamma)$ with $\delta(\calp_2)=\calp_2$ and which interchanges its vertices in $\Gamma.\bv_1$ with those in $\Gamma.\bv_7$.    Since $\calp_2$ has two vertices in $\Gamma.\bv_1$ and two in $\Gamma.\bv_7$, we have not yet arrived at a contradiction.  But such a $\delta$ still cannot exist since it can be seen that the two vertices in $\Gamma.\bv_7$ are connected by an edge but those in $\Gamma.\bv_1$ are not.

The argument for $\calt_{14}$ follows the outline of the argument for $\calt_{12}$.  Here $\calp_1$ is the unique generating tile with five vertices, its two vertices in $\Gamma.\bv_1$ are connected by an edge, and those in $\Gamma.\bv_7$ are not.
\endproof

\proof[Proof of Proposition \ref{prop:comm}]
By Lemma \ref{lem:sym}, we have $\text{Comm}(\Gamma) = \text{Sym}(\calt_{18})$ so to prove the theorem we need to show $\text{Sym}(\calt_{18})=\Gamma$.  We already know that $\Gamma \subset \text{Sym}(\calt_{18})$.  

Suppose that $\gamma \in \text{Sym}(\calt_{18})-\Gamma$ is non-trivial.  Since $\calt_{18}$ is $\Gamma$-generated by $\calp_1$ and $\calp_2$ and these two polyhedra have different numbers of vertices, we may assume that $\gamma(\calp_1)=\calp_1$.  By composing $\gamma$ with ${\sf d}_2$, if necessary, we may assume that $\gamma$ is orientation preserving.  $\calp_1$ has one vertex in $\Gamma.\bv_1$, one in $\Gamma.\bv_3$, and four in $\Gamma.\bv_7$; thus by Lemma \ref{lem:sym}, $\gamma$ fixes $\bv_5$ (which is in $\Gamma.\bv_1$) and $\bv_3$.  

Using our embedding in the upper half space model, the vertices of $\calp_1$ in $\Gamma.\bv_7$ are taken to: 
\begin{align*}
& r_7=\frac{1+i}{2}, && {\sf c}(r_7) =\frac{1+i(2+\sqrt{3})}{2}, && {\sf d}_2{\sf c}(r_7)=\frac{3+i(2+\sqrt{3})}{2}, && {\sf d}_2(r_7) = \frac{3+i}{2}.  \end{align*}
Since $\gamma$ is an elliptic isometry preserving $\calp_1$ and fixing $r_3=\infty$ and $r_5 = 1 + \frac{i}{2}(1+\sqrt{3})$, it must act as a cyclic permutation on the set $\{r_7, \,{\sf c}(r_7), \,{\sf d}_2{\sf c}(r_7), \,{\sf d}_2(r_7)\}$.  But it is easy to see that the axis of $\gamma$ is not perpendicular to the plane that these points span, so this is impossible.
\endproof

\bibliographystyle{plain}
%\bibliography{special}

\end{document}